\documentclass[11pt]{amsart}
\usepackage[margin=1in]{geometry}
\usepackage[utf8]{inputenc}
\usepackage[T1]{fontenc}
\usepackage{graphicx}
\usepackage{xcolor}
\usepackage{mathrsfs}
\usepackage{amsthm}
\usepackage{amsmath}
\allowdisplaybreaks
\usepackage{amsfonts}
\usepackage{amssymb}
\usepackage{mathtools}
\usepackage{enumitem}
\usepackage[hidelinks]{hyperref}
\usepackage{tikz-cd}
\usepackage{tikz}
\usetikzlibrary{tqft}
\usetikzlibrary{positioning}
\usetikzlibrary{matrix,quotes,arrows.meta,bending}
\usetikzlibrary{decorations.markings}
\usetikzlibrary{calc}
\usepackage[dvipsnames]{xcolor}
\usepackage{hyperref}
\usepackage[expansion=false]{microtype}
 
\definecolor{linknavy}{RGB}{30,60,120}
\hypersetup{colorlinks=true,linkcolor=linknavy,citecolor=linknavy,urlcolor=linknavy}

\theoremstyle{plain}
\newtheorem{Thm}{Theorem}[section]
\newtheorem{Prop}[Thm]{Proposition}
\newtheorem{Lem}[Thm]{Lemma}
\newtheorem{Cor}[Thm]{Corollary}
\newtheorem{Conj}[Thm]{Conjecture}
\theoremstyle{definition}
\newtheorem{Def}[Thm]{Definition}
\newtheorem{Rem}[Thm]{Remark}

\newcommand{\nt}{\text{int}}

\newcommand{\E}{\mathcal{E}}

\newcommand{\defn}[1]{\textbf{\textit{#1}}}

\newcommand{\R}{\mathbb{R}}
\newcommand{\G}{\mathcal{G}}

\newcommand{\g}{\mathfrak{g}}
\newcommand{\bh}{\mathbb{H}}

\newcommand{\too}{\longrightarrow}
\newcommand{\mtoo}{\longmapsto}
\newcommand{\Ad}{\operatorname{Ad}}

\newcommand{\CC}{\mathbb{C}}
\newcommand{\M}{\mathcal{M}}
\newcommand{\N}{\mathcal{N}}
\renewcommand{\t}{\mathfrak{t}}

\renewcommand{\L}{\mathcal{L}}
\newcommand{\F}{\mathcal{F}}

\newcommand{\Z}{\mathcal{Z}}

\newcommand{\A}{\mathcal{A}}

\newcommand{\Ker}{\text{Ker}}
\newcommand{\im}{\text{im}}
\newcommand{\K}{\mathcal{K}}
\newcommand{\PP}{\mathcal{P}}
\newcommand{\grad}{\mathrm{grad}}
\newcommand{\B}{\mathcal{B}}
\newcommand{\sll}[1]{\mkern-4mu\mathbin{/\mkern-5mu/}_{\mkern-4mu{#1}}}
\newcommand{\Lie}{\operatorname{Lie}}
\newcommand{\ad}{\operatorname{ad}}

\title{Nahm--Kirchhoff Moduli Spaces, Hyperkähler Quotients and 2D Field Theories}
\date{\today}
\author{Mohamed Moussadek Maiza}
\address[Mohamed Moussadek Maiza]{D\'{e}partement de math\'{e}matiques \\ Universit\'{e} de Sherbrooke \\ 2500 Bd de l’Universit\'{e} \\ Sherbrooke, QC, J1K 2R1, Canada}
\email{mohamed.moussadek.maiza@usherbrooke.ca}

\begin{document}

\maketitle

\begin{abstract}
To each oriented quiver with boundary $\Gamma$ and compact connected Lie group $G$, we associate a \emph{Nahm--Kirchhoff moduli space} $\M_{\mathbb{H}}(\Gamma)$: the space of solutions to Nahm's equations along the edges, subject to Kirchhoff matching conditions at the interior vertices, modulo the gauge group acting trivially on the boundary. We prove that, for a connected quiver, $\M_{\mathbb{H}}(\Gamma)$ is a finite-dimensional smooth manifold of dimension $4(|E|-|\Gamma_{\mathrm{int}}|)\dim G$, and that it carries a natural hyperkähler structure obtained as an infinite-dimensional hyperkähler quotient. Using a Kempf--Ness argument on the Cartan--Hadamard manifold $G_{\mathbb{C}}/G$, we establish a quiver analogue of Donaldson's theorem, identifying this hyperkähler quotient with the complex-symplectic quotient of solutions to the complex Nahm equations with Kirchhoff matching, modulo the complexified gauge group. We show that the resulting complex moduli space $\mathcal{M}_{\mathbb{C}}(\Gamma)$ is independent of the edge lengths and isomorphic to a finite-dimensional complex-symplectic quotient of $(T^*G_{\mathbb{C}})^E$, while the hyperkähler metric depends on them: a uniform rescaling of the lengths acts by a homothety, and the metrics transported to the fixed complex-symplectic quotient genuinely vary. Through the associated cutting and gluing laws, we show that $\Gamma \mapsto \mathcal{M}_{\mathbb{C}}(\Gamma)$ defines a two-dimensional topological quantum field theory valued in the category of complex Hamiltonian manifolds, in the spirit of Moore--Tachikawa, and we show that the hyperkähler refinement is not functorial on $\mathbf{Cob}_2$: the hyperkähler spaces attached to the spines of a cylinder are pairwise non-isomorphic for different lengths, a datum that $\mathbf{Cob}_2$ does not record. Retaining the edge lengths, the spaces $\M_{\mathbb{H}}(\Gamma)$ assemble into a metric field theory, namely a functor between non-unital monoidal categories of metric cobordisms and of hyperkähler manifolds with composition given by hyperkähler reduction; over each cone of the tropical moduli space $M_{g,n}^{\mathrm{trop}}$ they form a family with fixed complex-symplectic fibre and non-trivially varying hyperkähler metric. Finally, we prove that the harmonic gauge representatives converge as an internal edge collapses, and that on a long edge carrying a regular semisimple Higgs field the Kempf--Ness minimizer is exponentially confined near the centralizer of that element; the corresponding limits of the hyperkähler metrics are discussed as open problems.
\end{abstract}

\section{Introduction}

\subsection*{Context}
Since the work of Moore and Tachikawa \cite{MooreTachikawa}, it has been understood that a two-dimensional topological quantum field theory need not take values in vector spaces: one may instead ask for a symmetric monoidal functor from the cobordism category $\mathbf{Cob}_2$ into a category whose objects are
complex reductive groups and whose morphisms are holomorphic symplectic manifolds equipped with Hamiltonian actions, composition being given by holomorphic symplectic reduction. The existence of such a functor, and the geometric nature of
the spaces it assigns to cobordisms, has since been established from several points of view, most notably through the shifted symplectic formalism of Crooks and Mayrand \cite{CrooksMayrand2024MooreTachikawa}. A recurring theme in this circle of ideas is that the invariants attached to a surface are moduli spaces of solutions to a gauge-theoretic equation on a graph modelling that surface, and that the gluing of surfaces is implemented by symplectic reduction.

The present paper is a continuation of this program. In the companion work \cite{maiza2026lax} we associated to each oriented quiver with boundary a \emph{Lax--Kirchhoff moduli space}, built from Lax (multiplicative, holonomy-type)
data along the edges subject to a Kirchhoff matching condition at the interior vertices, and we showed that these spaces assemble into a Hamiltonian 2D TQFT; the
transition from the multiplicative to the additive picture was taken up in \cite{maiza2026multiplicative}. Here we develop the additive, hyperkähler counterpart of that story, in which the Lax data are replaced by solutions to \emph{Nahm's equations}.

\subsection*{Nahm's equations and hyperkähler geometry}
Nahm's equations, first introduced in the study of magnetic monopoles and given a gauge-theoretic interpretation by Donaldson \cite{Donaldson1984Nahm}, are the
system
\[
\dot{A}_{\bar i} + [A_0, A_{\bar i}] + [A_{\overline{i+1}}, A_{\overline{i+2}}] = 0,
\qquad i \in \{1,2,3\},
\]
for a quadruple $(A_0,A_1,A_2,A_3)$ of $\g$-valued functions on an interval. Their moduli spaces are among the fundamental sources of hyperkähler metrics: Kronheimer
\cite{kronheimer1988hyperkahler,Kronheimer1990coadjoint} realized the cotangent bundle $T^*G_{\CC}$ and the regular semisimple coadjoint orbits of a complex Lie
group as moduli of solutions to Nahm's equations on an interval, and this was later extended and refined by Biquard \cite{Biquard1996Nahm}, Kovalev
\cite{Kovalev1996Nahm}, Dancer--Swann \cite{DancerSwann1996Hyperkahler} and
Bielawski \cite{Bielawski2007Lie}. In each case the hyperkähler structure is produced by an infinite-dimensional hyperkähler quotient in the sense of Hitchin--Karlhede--Lindström--Roček (HKLR) \cite{hitchin1987hyperkahler}, and the
identification of the resulting space with a complex-symplectic quotient rests on a version of Donaldson's theorem, whose analytic core is a Kempf--Ness argument on the non-positively curved symmetric space $H = G_{\CC}/G$.

\subsection*{Nahm--Kirchhoff moduli spaces}
Our starting point is that a connected oriented cobordism deformation-retracts onto an oriented quiver with boundary $\Gamma = (E,V,s,t)$, with the incoming and outgoing circles recorded by the boundary vertices and the topology of the surface
encoded in the interior vertices. To such a $\Gamma$, together with a compact connected Lie group $G$, we attach the space of quadruples of $\g$-valued fields on the edges that solve Nahm's equations, satisfy the Kirchhoff law
\[
\sum_{e \in t^{-1}(v)} A^e_i(b_e) = \sum_{e \in s^{-1}(v)} A^e_i(a_e)
\qquad (v \in \Gamma_{\mathrm{int}},\ i = 1,2,3)
\]
at the interior vertices, and are considered modulo the gauge group
$\G_0(\Gamma)$ acting trivially on the boundary. We call the resulting quotient the \emph{Nahm--Kirchhoff moduli space} $\M_{\mathbb{H}}(\Gamma)$. The residual action
of $G^{\partial\Gamma}$ makes it a Hamiltonian space over the boundary, which is precisely the structure required for the gluing of cobordisms.

\subsection*{Main results}
Our first result establishes that these moduli spaces are as regular as one could hope. Working throughout with Banach manifolds of $C^1$ connections and $C^2$ gauge
transformations, and constructing local slices in the spirit of Kronheimer \cite{kronheimer1988hyperkahler}, we prove (Theorem \ref{smooth quotient}) that for
a connected quiver $\Gamma$ with boundary, $\M_{\mathbb{H}}(\Gamma)$ is a finite-dimensional smooth manifold with
\[
\dim \M_{\mathbb{H}}(\Gamma) = 4\bigl(|E| - |\Gamma_{\mathrm{int}}|\bigr)\dim G.
\]

We then endow $\M_{\mathbb{H}}(\Gamma)$ with a hyperkähler structure. Because the HKLR theorem is stated in finite dimensions, we isolate the analytic hypotheses
that become non-automatic in the infinite-dimensional setting and verify them directly, in the form of a Marsden--Weinstein reduction theorem
(Theorem \ref{infinite red}) of the type used for the Lax--Kirchhoff spaces \cite{maiza2026lax,MarsdenWeinstein}. The one computation that cannot be
transcribed from the multiplicative case is the hyperkähler moment map itself, whose cross terms $[A_b,A_c]$ require the quaternionic bookkeeping carried out in
Proposition \ref{moment iden}; the passage from three symplectic reductions to a single hyperkähler quotient is then governed by the quaternionic slice criterion (Proposition \ref{q slice}).

The heart of the paper is a quiver analogue of Donaldson's theorem
(Theorem \ref{complex q}), identifying the hyperkähler quotient with a complex-symplectic quotient:
\[
\M_{\mathbb{H}}(\Gamma) = \A_{\mathbb{H}}(\Gamma)/\G^0(\Gamma)
\;\xrightarrow{\ \sim\ }\;
\mu_{\CC}^{-1}(0)/\G^0_{\CC}(\Gamma) =: \M_{\CC}(\Gamma),
\]
where the right-hand side is the moduli space of solutions to the complex Nahm equations with Kirchhoff matching, modulo the complexified gauge group. The proof
proceeds through a Kempf--Ness functional whose critical points are the harmonic gauge representatives; after gauging away the connection edge by edge, the problem reduces to minimizing a per-edge energy on the Cartan--Hadamard manifold
$H^{\Gamma_{\mathrm{int}}}$. We prove existence and uniqueness of the minimizer (Lemmas \ref{lem mini}, \ref{lem uni}) and establish that the reduced functional is proper and \emph{strictly} geodesically convex (Theorem \ref{kempf ness}); as in the
Lax--Kirchhoff case, strict convexity is not visible edge by edge and is recovered globally by a propagation argument from the boundary that uses the connectedness of $\Gamma$.

This correspondence has two consequences that we wish to emphasize, as they delineate exactly which part of the geometry is topological. On the one hand, the
complex-symplectic manifold $\M_{\CC}(\Gamma)$ is \emph{independent of the edge
lengths} (Proposition \ref{prop:length-invariance-complex}) and is isomorphic to a finite-dimensional complex-symplectic quotient
\[
\M_{\CC}(\Gamma) \;\cong\; (T^*G_{\CC})^E /\!\!/\!\!/_{\CC} G_{\CC}^{\Gamma_{\mathrm{int}}}
\]
(Theorem \ref{thm:finite-dim-reduction}), obtained by parallel transport along each
edge. On the other hand, the hyperkähler metric \emph{does} depend on the edge lengths (Proposition \ref{prop:metric-length-dependence}). Nahm's equations are invariant under the simultaneous rescaling $t \mapsto \lambda t$, $A_\mu \mapsto \lambda^{-1}A_\mu$ of all four components, so that a uniform rescaling of all edge lengths by $\lambda$ multiplies the hyperkähler metric and the three Kähler forms by $\lambda^{-1}$, up to a fibrewise dilation of the complex-symplectic quotient. Beyond this homothety, the metrics transported to the fixed complex-symplectic manifold $\M_{\CC}(\Gamma)$ are not related by the canonical identifications, as an explicit computation at the origin shows.

\subsection*{The 2D TQFT}
The finite-dimensional model, together with reduction by stages, yields a cutting
and gluing law (Proposition \ref{thm:cutting-gluing}) and invariance of $\M_{\CC}(\Gamma)$ under the homotopy moves relating quivers that model the same
cobordism (Proposition \ref{thm:spine-invariance}). Assembling these facts, we obtain our main theorem (Theorem \ref{thm:2d-tqft-functor}): the assignment
$\Gamma \mapsto \M_{\CC}(\Gamma)$ defines a symmetric monoidal functor
\[
\Z : \mathbf{Cob}_2 \longrightarrow \mathbf{Ham}_{\CC}
\]
into the (Wehrheim--Woodward completed \cite{WehrheimWoodward2010Functoriality})
category of complex Hamiltonian manifolds, in the spirit of Moore--Tachikawa
\cite{MooreTachikawa}. Finally, the length dependence of the metric obstructs a hyperkähler refinement of $\Z$: the spaces attached to the single-edge spines of the cylinder are pairwise non-isomorphic as hyperkähler Hamiltonian $G\times G$-spaces for different lengths (Proposition \ref{prop:cylinder-length}), whereas $\mathbf{Cob}_2$ does not record lengths. Subdividing an edge while preserving its total length is, by contrast, a hyperkähler isometry (Theorem \ref{thm:hklr-metric-gluing}).

\subsection*{Metric cobordisms and tropical families}
The failure of topological invariance suggests retaining the edge lengths as part of the data. We introduce the category $\mathbf{MetCob}_2$ of metric cobordisms, whose morphisms are metric quivers modulo edge subdivisions preserving the total length, together with a hyperkähler target category $\mathbf{HK}$ in which composition is given by the HKLR hyperkähler quotient, completed in the sense of Wehrheim--Woodward \cite{WehrheimWoodward2010Functoriality}. Since gluing adds lengths, neither category has identity morphisms: both are non-unital monoidal categories. Subdividing an edge of length $l_e$ into two sub-edges of lengths $l_{e_1} + l_{e_2} = l_e$ is a hyperkähler isometry, and gluing two metric quivers along a boundary vertex is implemented by hyperkähler reduction (Theorem \ref{thm:hklr-metric-gluing}). As a consequence, the assignment $(\Sigma,\Gamma,\mathbf{l}) \mapsto \M_{\mathbb{H}}(\Gamma,\mathbf{l})$ defines a functor of non-unital monoidal categories
\[
\mathcal{Z}_{\mathbb{H}} : \mathbf{MetCob}_2 \longrightarrow \mathbf{HK}.
\]
Combining the quiver Donaldson theorem with the length-independence of $\M_{\CC}(\Gamma)$, we show that over each open cone of the tropical moduli space $M_{g,n}^{\mathrm{trop}}$ the spaces $\M_{\bh}(\Gamma,\mathbf{l})$ form a family which is trivial as a family of complex-symplectic manifolds, with fibre $\M_{\CC}(\Sigma_{g,n})$, while the hyperkähler metric varies non-trivially (Theorem \ref{thm:tropical-bundle}).

We also study degenerations. When an internal edge collapses, $l_{e_0} \to 0$, we prove that the harmonic gauge representatives provided by the Kempf--Ness functional converge to those of the contracted quiver (Proposition \ref{prop:edge-collapse}); the convergence of the hyperkähler metrics themselves is formulated as a conjecture. When an edge becomes long and its Higgs field is regular semisimple, we prove that the Kempf--Ness minimizer is exponentially confined near the minimum locus of the potential, the centralizer of that element (Lemma \ref{lem:transverse_decay} and Corollary \ref{cor:middle}); we explain why the naive expectation of a splitting into a hyperkähler quotient by the maximal torus cannot hold as stated, and what a correct formulation should involve (Remark \ref{rem:long-edge}).

\subsection*{Organization of the paper}
Section \ref{section 2} recalls the combinatorics of oriented quivers with boundary
and introduces the Nahm--Kirchhoff moduli spaces. Section \ref{section 3} establishes the
smooth structure and the dimension formula through the construction of local
slices. Section \ref{section 4} constructs the hyperkähler structure, proves the Donaldson--Kempf--Ness correspondence, and analyses the dependence of the complex moduli space and of the hyperkähler metric on the edge lengths. Section \ref{sec:2d-tqft-complex}
assembles the cutting, gluing and homotopy-invariance results into the 2D TQFT functor and proves the obstruction to a hyperkähler-valued theory. Section \ref{section 6} introduces the non-unital categories $\mathbf{MetCob}_2$ of metric cobordisms and $\mathbf{HK}$ of hyperkähler manifolds, proves the metric gluing
law via hyperkähler reduction, constructs the metric field theory functor $\Z_{\mathbb{H}}$ and the associated families over the cones of $M_{g,n}^{\mathrm{trop}}$, and studies the edge-collapse and long-edge limits.

\subsection*{Acknowledgments} The author is deeply grateful to Giovanny Soto for many stimulating and generous discussions throughout the preparation of this work, and for his invaluable help with the proofs of Subsection \ref{subsec:kn-edge} on the Kempf--Ness functional on an edge; his insights on the variational formulas and the convexity arguments were essential to the analysis carried out there. The author also thanks Maxence Mayrand for useful discussions, in particular concerning the obstruction to a 2D TQFT with values in the category of hyperkähler manifolds, and the Institut des sciences mathématiques (ISM) and the Ernest-Monga scholarship for their financial support.

\section{Nahm--Kirchhoff moduli spaces} \label{section 2}
In this section, we recall the basic notions of oriented quivers (see also \cite{maiza2026lax}) and we introduce Nahm--Kirchhoff moduli spaces. Throughout this section, $G_{\mathbb{C}}$ is the complexification of a compact Lie group $G$ and $\mathbb{H}$ is the algebra of quaternions.

\subsection{Quivers}
A quiver is a tuple $\Gamma = (E,V,s,t)$ such that \begin{itemize}
    \item $E$ is a finite set of edges,
    \item $V$ is a finite set of vertices,
    \item $s,t: E \too V$ are two maps called the \textit{source} and \textit{target} maps, respectively.
\end{itemize}

We assume that $\Gamma$ has no isolated vertices. The degree of a vertex $v\in V$ is the sum \[
\deg(v) \coloneqq \deg_{\text{in}}(v) + \deg_{\text{out}}(v),
\]
where $\deg_{\text{in}}(v) = |t^{-1}(v)|$ is the number of incoming edges and $\deg_{\text{out}}(v)$ is the number of outgoing edges. The boundary of $\Gamma$, denoted $\partial \Gamma$, is the set of vertices of degree $1.$ It is the disjoint union \[
\partial \Gamma = \partial \Gamma^{+} \sqcup \partial \Gamma^{-},
\]
where \[
\partial \Gamma^{+} \coloneqq \{v\in V, \deg_{\text{in}}(v) = 1\} \text{ and } \partial \Gamma^{-} = \{v\in V, \deg_{\text{out}}(v) = 1\},
\]

are the outgoing and the incoming boundaries, respectively. This matches the convention of Section \ref{sec:2d-tqft-complex}, where the incoming circles of a cobordism correspond to $\partial\Gamma^{-}$ and the outgoing circles to $\partial\Gamma^{+}$. The complement $V- \partial \Gamma$ is the set of interior vertices.\\

Let $(E,V,s,t)$ be a quiver. To each edge $e$ in $E$, we associate a compact interval $[a_e,b_e]$ and define $I \coloneqq \bigsqcup_{e\in E} I_e.$ Consider the Banach space $\B_{\bh}(\Gamma) \coloneqq C^{1}(I,\g^{4}) = \prod_{e\in E} C^{1}(I_{e},\g^{4}),$ whose elements are $A = (A^{e})_{e\in E},$ where each $A^{e} = (A^{e}_0, A^{e}_1, A^{e}_2, A^{e}_3)$  satisfies $A^{e}_i : I_e \too \g$ for $i = {0,1,2,3}.$ 

Inside $\B_{\bh}(\Gamma),$ we define the closed subspace called the \textbf{Kirchhoff} space \begin{equation}
\K_{\bh} \coloneqq \{A \in \B_{\bh}(\Gamma): \quad  \sum_{e \in t^{-1}(v)} A^{e}_{i}(b_e) = \sum_{e \in s^{-1}(v)} A^{e}_{i}(a_e), \text{ for all } v \in \Gamma_{\nt}, i \in \{1,2,3\}\}.
\end{equation}

It is the set of connections $A$ whose Higgs fields $A_1,A_2,A_3$ satisfy the Kirchhoff law. 

Now, we define the \textbf{Nahm} space to be \[
\N_{\bh}(\Gamma) = \{A \in \B_{\bh}(\Gamma), \dot{A}_{\bar{i}} + [A_0,A_{\bar{i}}] + [A_{\overline{i+1}},A_{\overline{i+2}}] = 0\},
\]

where $\bar{k}$ is the class of $k$ modulo $\mathbb{Z}_3$, and the \textbf{Nahm--Kirchhoff} space to be \[
\A_{\bh}(\Gamma) \coloneqq \N_{\bh}(\Gamma) \cap \K_{\bh}(\Gamma).
\]

The Banach space $\B_{\bh}(\Gamma)$ admits a natural action of the Banach Lie group $C^{2}(I,G) = \prod_{e\in E}C^{2}(I_e,G),$ by the gauge transformations \[
g\cdot (A_0,A_1,A_2,A_3) = (gA_{0}g^{-1} - \dot{g}g^{-1}, gA_{1}g^{-1},gA_{2}g^{-1},gA_{3}g^{-1}), \quad g\in C^{2}(I,G), A\in \B_{\bh}(\Gamma).
\]

Let $\G(\Gamma)$ be the closed subgroup of $C^{2}(I,G)$ consisting of elements that match at vertices; that is, an element $g = (g_e)_{e\in E}$ lies in $\G(\Gamma)$ if and only if there exists $g(v) \in G$ for each $v$ such that $g_{e}(b_e) = g(t(e))$ and $g_{e}(a_e) = g(s(e))$ for all $e\in E.$ Then, $\G(\Gamma)$ preserves the Nahm--Kirchhoff space $\A_{\bh}(\Gamma).$ We denote by $\G_{0}(\Gamma)$ its normal subgroup consisting of elements that are trivial at the boundary, i.e. \[
\G_{0}(\Gamma) = \{ g \in \G(\Gamma), g(v) = 1 \text{ for all } v \in \partial \Gamma \}.
\]

\begin{Def}
 \textbf{The Nahm--Kirchhoff} moduli space of a quiver $\Gamma$ is the quotient \[
 \M_{\bh}(\Gamma) \coloneqq \A_{\bh}(\Gamma)/\G_{0}(\Gamma).
 \]

 Finally, note that the action of $\G(\Gamma)$ descends to an action of the quotient group \[
 \G(\Gamma)/\G_{0}(\Gamma) \cong G^{\partial \Gamma}
 \]
 on $\M_{\bh}(\Gamma).$
\end{Def}

\section{The smooth structure} \label{section 3}

In this section, we prove that for a connected quiver $\Gamma$ with boundary, the moduli space $\M_{\bh}(\Gamma)$ is a finite-dimensional smooth manifold.

\begin{Thm} \label{smooth quotient}
The subset $\A_{\bh}(\Gamma)$ is a Banach submanifold of $\B_{\bh}(\Gamma),$ and the quotient \[
\M_{\bh}(\Gamma) \coloneqq \A_{\bh}(\Gamma)/\G_{0}(\Gamma) 
\]
admits a unique smooth structure such that the projection $\A_{\bh}(\Gamma) \too \M_{\bh}(\Gamma)$ is a smooth submersion. Moreover, $\M_{\bh}(\Gamma)$ is finite-dimensional with $\dim(\M_{\bh}(\Gamma)) = 4(|E| - |\Gamma_{\nt}|)\dim(G).$
\end{Thm}

The proof proceeds in three steps. First, we show that $\A_{\bh}(\Gamma)$ is a smooth Banach submanifold of $\B_{\bh}(\Gamma).$ Second, we prove that the action of $\G_{0}(\Gamma)$ is free and proper. Finally, as we work with infinite-dimensional spaces, we construct local slices for this action. For more background on Banach manifolds, see, for example, \cite{Lang1999Fundamentals}.

\subsection{The Banach submanifold $\A_{\bh}(\Gamma) \subset \B_{\bh}(\Gamma)$} We consider the map \begin{align*}
    \psi: \B_{\bh}(\Gamma)& \too C^{0}(I,\g)^{3} \times (\g^{3})^{\Gamma_{\nt}}\\
A& \mtoo \Bigl( \dot{A}_{\bar{i}} + [A_0,A_{\bar{i}}] + [A_{\overline{i+1}},A_{\overline{i+2}}], \Bigl( \sum_{e\in s^{-1}(v)} A^{e}_{\bar{i}}(a_e) - \sum_{e \in t^{-1}(v)} A^{e}_{\bar{i}}(b_e) \Bigr)_{v\in \Gamma_{\nt}} \Bigr),
\end{align*}
so that $\psi^{-1}(0) = \A_{\bh}(\Gamma)$. The map $\psi$ will be interpreted as a moment map for the action of $\G_{0}(\Gamma)$ on $\B_{\bh}(\Gamma).$ In order to show that $0$ is a regular value of $\psi,$ we use the following lemma of \cite{maiza2026lax}:

\begin{Lem} [\cite{maiza2026lax}] \label{useful lemma}
    Let $W$ be a finite-dimensional real vector space and $B \in C^{0}(I,\text{End}(W)).$ Define \begin{align*}
    R: C^{1}(I,W)& \too C^{0}(I,W) \times W^{\Gamma_{\nt}}\\ x& \mtoo 
    \Bigl( \dot{x} + Bx, \Bigl( \sum_{e\in s^{-1}(v)} x_{e}(a_e) - \sum_{e \in t^{-1}(v)} x_{e}(b_e) \Bigr)_{v\in \Gamma_{\nt}} \Bigr).  
    \end{align*}
    Then $R$ is surjective and $\dim \Ker(R) = (|E| - |\Gamma_{\nt}|)\dim(W).$
\end{Lem}

\begin{Prop} \label{regular value}
    The subset $\A_{\bh}(\Gamma) \subset \B_{\bh}(\Gamma)$ is a Banach submanifold.
\end{Prop}

\begin{proof}
By Lemma \ref{useful lemma}, the differential 
\begin{align*}
    d\psi_{A}: \B_{\bh}(\Gamma)& \too C^{0}(I,\g^{3}) \times (\g^{3})^{\Gamma_{\nt}}\\
    X& \mtoo \Bigl( \dot{X}_{\bar{i}} + [A_0,X_{\bar{i}}] + [X_0,A_{\bar{i}}] + [X_{\overline{i+1}},A_{\overline{i+2}}] + [A_{\overline{i+1}},X_{\overline{i+2}}], \Bigl( \sum_{e\in s^{-1}(v)} X^{e}_{\bar{i}}(a_e) - \sum_{e\in t^{-1}(v)} X^{e}_{\bar{i}} (b_e) \Bigr) \Bigr), 
\end{align*}
is surjective for all $A \in \psi^{-1}(0).$ Hence $0$ is a regular value of $\psi$, and $\A_{\bh}(\Gamma)$ is a Banach submanifold of $\B_{\bh}(\Gamma).$
\end{proof}

\begin{Prop} \label{proper gauge}
    $\G_{0}(\Gamma)$ acts freely and properly on $\A_{\bh}(\Gamma).$
\end{Prop}

\begin{proof}
This follows from \cite[Proposition 3.4]{maiza2026lax} and \cite[Lemma 3.5]{maiza2026lax}.
\end{proof}

\subsection{Existence of Nahm slices} We construct the local slices for the action of $\G_{0}(\Gamma)$ on $\A_{{\bh}}(\Gamma),$ finishing the proof of Theorem \ref{smooth quotient}. Our approach is inspired by Kronheimer \cite{kronheimer1988hyperkahler}. For $A\in \A_{\bh}(\Gamma),$ the tangent space is the set \[
T_{A}\A_{\bh}(\Gamma) \coloneqq \{Y \in \K_{\bh}(\Gamma), \dot{Y}_{\bar{i}} + [A_{0},Y_{\bar{i}}] + [Y_0,A_{\bar{i}}] + [Y_{\overline{i+1}},A_{\overline{i+2}}] + [A_{\overline{i+1}},Y_{\overline{i+2}}] = 0\}.
\]

As in the Lax case \cite{maiza2026lax}, we consider the Lie algebra $C^{2}_{0}(I,\g)$ of $\G_{0}(\Gamma)$, consisting of elements $u \in C^{2}(I,\g)$ for which there exist values $u(v) \in \g$ for each $v$ such that $u_{e}(a_e) = u(s(e))$ and $u_{e}(b_e) = u(t(e))$ for all $e\in E,$ and $u(v) = 0$ for all $v\in \partial \Gamma.$ The infinitesimal action of $\G_{0}(\Gamma)$ at $A$ is the linear map 

\begin{equation} \label{DA}
D_{A}: C^{2}_{0}(I,\g) \too T_{A}\A_{\bh}(\Gamma), \quad u \mtoo ([u,A_0] - \dot{u}, [u,A_1],[u,A_2],[u,A_3]).
\end{equation}

The space $C^{2}_{0}(I,\g)$ admits a continuously non-degenerate pairing with $C^{0}(I,\g) \times \g^{\Gamma_{\nt}}$ defined by \begin{equation}
    C^{2}_{0}(I,\g) \times C^{0}(I,\g) \times \g^{\Gamma_{\nt}} \too \R, \langle u,(Y,Z) \rangle \quad \int_{I} \langle u(t), Y(t) \rangle + \sum_{v\in \Gamma_{\nt}} \langle u(v), Z(v) \rangle 
\end{equation}

With the natural inner product on $T_{A}\A_{\bh}$ defined by \[
\langle Y(t),Z(t) \rangle = \int_{I} \sum_{i=0}^{3} \langle Y_{i}(t),Z_{i}(t) \rangle dt,  
\]

we show that the operator 
\begin{align*}
L: T_{A}\A_{\bh}(\Gamma)& \too C^{0}(I,\g) \times \g^{\Gamma_{\nt}}\\
Y& \mtoo \bigl( \dot{Y}_{0} + [A_0,Y_0] + [A_1,Y_1] + [A_2,Y_2] + [A_3,Y_3], \bigl( \sum_{e\in s^{-1}(v)} Y^{e}_{0}(a_e) - \sum_{e\in t^{-1}(v)} Y^{e}_{0}(b_e) \bigr) \bigr),
\end{align*}
 is an adjoint of $D_A$ for this pairing.
 \begin{Prop}
     We have $$\langle D_{A}u,Y \rangle = \langle u,L(Y) \rangle,$$ for all $u\in C^{2}_{0}(I,\g)$ and $Y \in T_{A}\A_{\bh}(\Gamma).$
 \end{Prop}

\begin{proof}
    Using integration by parts and the $\Ad$-invariance of the scalar product, one gets \begin{align*}
        \langle D_{A}u, Y \rangle &= \int_{I} \langle [u,A_0] - \dot{u},Y_{0} \rangle + \langle [u,A_1],Y_1 \rangle + \langle [u,A_2],Y_2 \rangle + \langle [u,A_3],Y_3 \rangle dt\\ =& \int_{I} \langle u,\dot{Y}_{0} \rangle + \sum_{i=0}^{3} \langle [u,A_i],Y_i \rangle dt + \sum_{e\in E} (\langle u_{e}(a_e),Y^{e}_{0}(a_e) \rangle - \langle u_{e}(b_e),Y^{e}_{0}(b_e)\rangle)\\ =& \int_{I} \langle u(t),\sum_{i=0}^{3} [A_{i},Y_i] \rangle dt + \sum_{v\in \Gamma_{\nt}} \langle u(v), \sum_{e\in s^{-1}(v)} Y^{e}_{0}(a_e) - \sum_{e\in t^{-1}(v)} Y^{e}_{0}(b_e) \rangle\\ =& \langle u,L(Y) \rangle = \langle u,D^{*}_{A}Y \rangle.     
    \end{align*}
\end{proof}

A natural candidate for a slice at $A$ is a submanifold $S$ whose tangent space at $A$ is orthogonal to the $\G_{0}(\Gamma)$-orbit through $A.$ Since $\text{im}D_A$ is the infinitesimal orbit, its orthogonal complement is $\text{Ker} D^{*}_{A}.$ To realize this decomposition, we need the following technical result. For $A_0 \in C^{1}(I,\g)$ and $u\in C^{2}(I,\g),$ set $\nabla_{A_0}u = \dot{u} + [A_0,u].$

\begin{Prop} \label{prop 3.6}
The operator \begin{align*}
L_{A}\coloneqq D^{*}_{A}D_{A}: C^{2}_{0}(I,\g)& \too C^{0}(I,\g) \times \g^{\Gamma_{\nt}}\\
u& \too \bigl( -\nabla^{2}_{A_0}u - \sum_{i=1}^{3} \ad^{2}_{A_i}u, (\delta_{v}(\dot{u}))_{v\in \Gamma_{\nt}} \bigr),
\end{align*}
is an isomorphism of Banach spaces, where $\delta_{v}(\dot{u}) = \sum_{e\in s^{-1}(v)} \dot{u}^{e}(a_e) - \sum_{e\in t^{-1}(v)}\dot{u}^{e}(b_e).$ In particular, \[
T_{A}\A_{\bh}(\Gamma) = \im D_A \oplus \ker D^{*}_A,
\]
is a closed, topological direct sum.
\end{Prop}

To prove this proposition, we use the following sequence of lemmas.

\begin{Lem} \label{oper}
The operator $L = D^{*}_{A}$ is surjective, and \[
\dim \Ker D^{*}_{A} = 4(|E| - |\Gamma_{\nt}|)\dim(\g).
\]    
\end{Lem}

\begin{proof}
Define the operator 
\begin{align*}
R: \B_{\bh}(\Gamma)& \too C^{0}(I,\g^{4})\times (\g^{4})^{\Gamma_{\nt}}\\
Y& \mtoo \bigl(d_{A}\psi^{1}(Y), d_{A}\psi^{2}(Y),d_{A}\psi^{3}(Y),(D^{*}_{A}(Y))_{1}, (\delta_{v}(Y_1),\delta_{v}(Y_2),\delta_{v}(Y_3),\delta_{v}(Y_0))_{v\in \Gamma_{\nt}} \bigr).
\end{align*}
By Lemma \ref{useful lemma}, $R$ is surjective. Moreover, \[
\Ker R = \{ Y\in \B_{\bh}, Y \in T_{A}\A_{\bh}(\Gamma) \text{ and } D^{*}_{A}Y = 0\} = \Ker D^{*}_{A}.
\]

Then, $D^{*}_{A}$ is surjective and $\dim \Ker D^{*}_{A} = \dim \Ker R = 4(|E| - |\Gamma_{\nt}|)\dim(\g).$
\end{proof}

\begin{Lem} \label{injectivity}
    The operator $L_A = D^{*}_{A}D_{A}$ is injective.
\end{Lem}

\begin{proof}
    This follows immediately from \cite[Lemma 3.10]{maiza2026lax}.
\end{proof}

\begin{Lem} \cite[Lemma 3.11]{maiza2026lax} \label{lem 3.9}
Setting $A = (A_0,A_1,A_2,A_3) = 0,$ the operator \begin{align*}
    L_0: C^{2}_{0}(I,\g)& \too C^{0}(I,\g) \times \g^{\Gamma_{\nt}}\\
    u& \mtoo  (-\Ddot{u}, (\delta_{v}(\dot{u}))_{v\in \Gamma_{\nt}})
\end{align*}
is an isomorphism of Banach spaces.
\end{Lem}

\begin{proof}[Proof of Proposition \ref{prop 3.6}]
Write $L_A = L_0 + K,$ where $L_0$ is the operator of Lemma \ref{lem 3.9} and \[
    K(u) = (-2\ad_{A_0}\dot{u} - \ad_{\dot{A_0}}u - \ad^{2}_{A_0}u - \sum_{i=1}^{3}\ad^{2}_{A_i}u,0). 
    \]
    
The only term carrying a derivative is $\ad_{A_0}\dot{u}$ (which involves $A_0,$ not the Higgs fields). So $K$ still factors as \[
C^{2}_{0}(I,\g) \too C^{1}(I,\g) \too C^{0}(I,\g)\times \g^{\Gamma_{\nt}}, 
\]
where the first map is compact by Arzel\`a--Ascoli and the second map is bounded. Moreover, we observe that $L_{A} = L_{0}(I + T),$ where $T = L_{0}^{-1}K,$ which is a compact operator on $C^{2}_{0}(I,\g).$ By the Fredholm alternative, either $I + T$ is invertible or there exists $u\neq 0$ such that $(I + T)u = 0.$ In the latter case, $L_{A}u = 0,$ which contradicts the injectivity of $L_{A}$ (Lemma \ref{injectivity}). Hence $(I + T)$ is invertible, and therefore $L_A$ is an isomorphism.
\end{proof}

Now, we construct a local slice for the $\G_{0}(\Gamma)$-action at a fixed $A\in \A_{\bh}.$ We consider the map \[
F: \G_{0}(\Gamma) \times \A_{\bh}(\Gamma) \too C^{0}(I,\g) \times \g^{\Gamma_{\nt}}, \quad F(g,B) = D^{*}_{A}(g\cdot B - A).
\]

We have that $F(1,B) = 0$ and, for $u\in T_{A}\G_{0}(\Gamma) = C^{2}_{0}(I,\g),$ \[
dF_{(1,A)}(u,0) = D^{*}_{A}D_{A}u.
\] 

By Proposition \ref{prop 3.6}, $D^{*}_{A}D_{A}$ is an isomorphism. Hence, by the implicit function theorem, there exist neighborhoods $U \subset \A_{\bh}(\Gamma)$ of $A$ and $V \subset \G_{0}(\Gamma)$ of the identity $1$ and a unique smooth map $s: U\too V$ such that $F(s(B),B) = 0$ for all $B\in U.$ We define our candidate slice to be \[
S \coloneqq \{ B\in U; F(1,B) = 0\} = \{A+a \in U, D^{*}_{A}a = 0\}.
\] 

By \cite[Lemma 3.12,Lemma 3.13]{maiza2026lax}, $S$ is a Banach submanifold of $\A_{\bh}(\Gamma)$ and the map $\psi: \G_{0}(\Gamma) \times S \too \A_{\bh}, \quad (g,B)\mtoo g\cdot B$ is a diffeomorphism from $\G_{0}(\Gamma)\times S$ onto an open neighborhood $W = \G_{0}(\Gamma)\cdot S$ of $A.$ As a consequence, $S$ is a slice for the $\G_{0}(\Gamma)$-action at $A,$ which completes the construction of the smooth structure of Theorem \ref{smooth quotient}. The dimension formula follows from $T_{A}S = \Ker D^{*}_{A}$ and Lemma \ref{oper}. 

\section{Hyperkähler structure and complex quotient} \label{section 4}

In this section, we give a rigorous construction of the hyperkähler structure on the Nahm--Kirchhoff moduli space $\M_{\bh}(\Gamma)$ associated with a connected oriented quiver with boundary and a compact connected Lie group $G.$ The construction is organized as an infinite-dimensional hyperkähler quotient in the sense of Hitchin--Karlhede--Lindstr\"om--Ro\v{c}ek (HKLR). Because the original HKLR theorem is stated in finite dimensions, we state the analytic hypotheses that the infinite-dimensional setting forces (Section \ref{section 3}), in the form of a Marsden--Weinstein reduction theorem as used in \cite{maiza2026lax} for the Lax--Kirchhoff moduli space. Moreover, we give a rigorous proof of the analogue of Donaldson's theorem for Nahm's equations on $\Gamma.$ This establishes the biholomorphism between the hyperkähler quotient (the moduli space of solutions of Nahm's equations with Kirchhoff matching) and the complex-symplectic quotient (the moduli space of solutions to the complex Nahm equations with Kirchhoff matching, modulo the complexified gauge group). The proof proceeds via a Kempf--Ness argument on the Cartan--Hadamard manifold $H^{\Gamma_{\nt}}$, where $H = G_{\mathbb{C}}/G.$ 

\subsection{Hyperkähler structure}
We recall the infinite-dimensional Banach space $\B_{\bh}(\Gamma) = C^{1}(I,\g^{4})^{E}$, with $I = \bigsqcup_{e\in E} I_{e}, I_{e} = [a_e,b_e].$ We identify $\g^{4} \cong \g\otimes_{\mathbb{R}} \bh$ by writing a quadruple as a quaternion-valued element \begin{equation} \label{eq}
A = A_{0} + A_{1}i + A_{2}j + A_{3}k, \quad A_{i} \in \g.
\end{equation}

\begin{Def} \label{de 4.1}
    Let $I_1,I_2,I_3$ act on $\B_{\bh}(\Gamma)$ by left quaternion multiplication on the second factor of \ref{eq}, i.e., by $i,j \text{ and } k$ respectively. In components, \begin{align*}
        I_{1}(X_0,X_1,X_2,X_3) &= (-X_1,X_0,-X_3,X_2)\\
        I_{2}(X_0,X_1,X_2,X_3) &= (-X_2,X_3,X_0,-X_1)\\
        I_{3}(X_0,X_1,X_2,X_3) &= (-X_3,-X_2,X_1,X_0) 
        \end{align*}
\end{Def}

We equip $\B_{\bh}(\Gamma)$ with the constant (weak) $L^{2}$ metric \[
\langle X,Y \rangle = \int_{I} \sum_{i=0}^{3} \langle X_i (t),Y_i (t) \rangle dt, 
\] 
and the three constant 2-forms $\omega_i (X,Y) = \langle I_i X,Y \rangle.$

\begin{Lem} \label{tri hlm}
    The operators $I_1,I_2 \text{ and } I_3$ satisfy the quaternion relations $I^{2}_{1} = I^{2}_{2} = I^{2}_{3} = -\text{id}$ and $I_{1}I_{2} = I_3, I_{2}I_{3} = I_1 \text{ and } I_{3}I_{1} = I_2;$ each $I_i$ is a $\langle \cdot ,\cdot \rangle$-isometry, and each $\omega_i$ is a constant (hence closed) weakly non-degenerate $2$-form. Consequently, $(\B_{\bh}(\Gamma),\langle \cdot,\cdot \rangle, I_1,I_2,I_3)$ is a hyperkähler Banach manifold. 
\end{Lem}

\begin{proof}
    The relations follow because $I_a$ is left multiplication by the $a$-th imaginary unit and left multiplication is an algebra homomorphism $\bh \too \text{End}_{\mathbb{R}}(\bh).$ Thus $I_{a}I_{b} = L_{q_{a}}L_{q_{b}} = L_{q_{a}q_{b}}$, where $(q_1,q_2,q_3) = (i,j,k),$ and $ij = k, \quad i^{2} = -1,$ etc. Each $L_{q_a}$ preserves the standard quaternionic form on $\bh$ and $\Ad$-invariance handles the $\g$-factor, so $I_a$ is an isometry. A constant form is closed; weak non-degeneracy is immediate from $\omega_{a}(X,\cdot) = \langle I_{a}X,\cdot \rangle$ and the invertibility of $I_a.$ 
\end{proof}

\begin{Lem}(Tri-holomorphic isometric action)  $\G_{0}(\Gamma)$ acts on $\B_{\bh}(\Gamma)$ by tri-holomorphic isometries: each $g\in \G_{0}(\Gamma)$ preserves $\langle \cdot,\cdot \rangle$ and commutes with $I_1, I_2,I_3.$
\end{Lem}

\begin{proof}
    The term $\dot{g}g^{-1}$ lies only in the $A_0$ part and is independent of $A,$ so it drops out of the differential. Hence \[
    dg_{A}(X) = (\Ad_{g}X_0, \Ad_{g}X_1, \Ad_{g}X_2, \Ad_{g}X_3) = (\Ad_{g} \otimes \mathrm{id}_{\bh})X,
    \]
    under \ref{eq}. Now, $\Ad_{g}$ acts on the $\g$-factor, while $\mathrm{id}_{\bh}$ and each $I_{a} = L_{q_a}$ act on the $\bh$-factor; operators acting on distinct tensor factors commute, so $dg_{A}(X) \circ I_{a} = I_{a} \circ dg_{A}(X).$ Since $\Ad_g$ is orthogonal and acts diagonally on $(X_0,X_1,X_2,X_3),$ $dg_{A}$ is a $\langle \cdot, \cdot \rangle$-isometry.
\end{proof}

We recall the finite-dimensional \textbf{HKLR} theorem, then isolate the hypotheses that become non-automatic in infinite dimensions, and finally state the infinite-dimensional Marsden--Weinstein reduction theorem in the form used for the Lax--Kirchhoff moduli space.

\begin{Thm} [\cite{hitchin1987hyperkahler}, finite-dimensional] Let $(M,g,I_1,I_2,I_3)$ be a finite-dimensional hyperkähler manifold with Kähler forms $\omega_1,\omega_2,\omega_3$, and let $K$ be a Lie group acting by preserving the hyperkähler structure and admitting an equivariant moment map $\mu = (\mu_1,\mu_2,\mu_3): M \too \mathfrak{k}^{*} \otimes \mathbb{R}^{3}, d\langle \mu_{a},\xi \rangle = i_{\xi_M}\omega_{a}.$ If $K$ acts freely on $\mu^{-1}(0)$ and the action is proper, then $M///K \coloneqq \mu^{-1}(0)/K$ is a hyperkähler manifold of dimension $\dim M - 4\dim K.$
    
\end{Thm}

\noindent\textbf{Infinite-dimensional Marsden--Weinstein reduction.} We work with one complex structure at a time. Fix $a$ and write $\omega \coloneqq \omega_a,$ with moment map $\mu = \mu_a$ valued in a Banach space $\mathfrak{k}^{*}$ paired with $\mathfrak{k} = \Lie(K)$ by a jointly continuous non-degenerate pairing.

\begin{Thm}[Marsden--Weinstein, infinite-dimensional form] \label{infinite red} 
Let $(\B,\omega)$ be a weak symplectic Banach manifold carrying a symplectic action of a Banach Lie group $K$ with equivariant moment map $\mu: \B \too \mathfrak{k}^{*}.$ Suppose that \begin{enumerate}
    \item (Regular value). $0$ is a regular value of $\mu,$ so that $ \A = \mu^{-1}(0)$ is a Banach submanifold with $T_{A}\A = \Ker (d_{A}\mu).$
    \item (Free and proper). The $K$-action on $\A$ is free and proper. 
    \item (Slices). The $K$-action on $\A$ admits local slices.
    \item (Orthogonality). For each $A\in \A,$ the orbit tangent space satisfies $T_{A}(K\cdot A) = (T_{A}\A)^{\perp_{\omega}},$ where $W^{\perp_{\omega}} = \{Z: \omega(Z,Y) = 0, \quad \forall Y \in W\}.$ \label{cd 4}
\end{enumerate}
    Then, the quotient $\M = \A/K$ is a smooth manifold and carries a unique symplectic form $\omega_{\text{red}}$ with $\pi^{*}\omega_{\text{red}} = i^{*}\omega,$ where $i: \A \hookrightarrow \B$ and $\pi:\A \too \M.$ 
\end{Thm}

\begin{Rem}
    The definition of the moment map already gives the inclusion $T_{A}(K\cdot A) \subset (T_{A}\A)^{\perp_{\omega}}:$ indeed, $\omega(D_{A}u,Y) = d_{A}\langle \mu,u \rangle(Y) = 0$ for $Y \in \ker d_{A}\mu.$ The content of condition \ref{cd 4} is the reverse inclusion, which fails in general in the weak setting and must be proved. This is precisely the content of \cite[P. 123]{MarsdenWeinstein} and of Lemma \ref{lemma 4.10} below.
\end{Rem}

\noindent\textbf{From three symplectic reductions to one hyperkähler quotient.} Theorem \ref{infinite red} produces, for each $a,$ a symplectic manifold $(\M,\omega_{a,\text{red}})$ on the same underlying space $\A/K$ with $\A = \mu^{-1}(0).$ The remaining task is to assemble the three reduced forms into a hyperkähler structure.

\begin{Prop} [Quaternionic slice criterion] \label{q slice} Suppose, in the situation of Theorem \ref{infinite red} applied simultaneously to $\omega_1,\omega_2$ and $\omega_3,$ that for each $A\in \mu^{-1}(0),$ the metric-orthogonal model of the reduced tangent space \[
T_{[A]}\M_{\bh}(\Gamma) \cong \Ker (D^{*}_{A}) \coloneqq \text{im}(D_{A})^{\perp} \cap T_{A}\mu^{-1}(0)
\]
is invariant under each $I_{a}.$ Then, the reduced forms $\omega_{a,\text{red}},$ the reduced metric, and the induced operators $I_{a,\text{red}}$ define a hyperkähler structure on $\M_{\bh}(\Gamma).$  
\end{Prop}

The plan is the following: verify the moment map conditions, verify Condition \ref{cd 4} (the first three conditions were verified in Section \ref{section 3}), verify the quaternionic criterion, and then conclude.\\

\noindent\textbf{The hyperkähler moment map.} This is the one computation that cannot be transcribed from the Lax--Kirchhoff case: the cross terms $[A_b,A_c]$ require the quaternionic bookkeeping. We fix once and for all the cyclic convention $(a,b,c) \in \{(1,2,3),(2,3,1),(3,1,2)\}.$\\

Recall that the moment map $\mu: \B_{\bh}(\Gamma) \too (C^{0}(I,\g) \times \g^{\Gamma_{\nt}})^{3}$ is defined by \[
\mu_{a}(A) = \bigl( \dot{A}_{a} + [A_0,A_a] + [A_b,A_c], \bigl( \sum_{e\in s^{-1}(v)} A^{e}_{a}(a_e) - \sum_{e\in t^{-1}(v)} A^{e}_{a}(b_e) \bigr)_{v\in \Gamma_{\nt}} \bigr),
\]

paired with $u\in \Lie(\G_{0}(\Gamma))$ via \[
\langle \mu_{a},u \rangle(A) = \int_{I} \langle u, \dot{A}_{a} + [A_0,A_a] + [A_b,A_c] \rangle dt + \underbrace{\sum_{v\in \Gamma_{\nt}} \bigg\langle u(v), \sum_{e\in s^{-1}(v)} A^{e}_{a}(a_e) - \sum_{e\in t^{-1}(v)} A^{e}_{a}(b_e) \bigg\rangle}_{\text{the vertex part}}
\]

\begin{Prop} [Moment map identity] \label{moment iden} For every $a \in \{1,2,3\},$ every $u\in \Lie(\G_{0}(\Gamma))$ and every $A\in \B_{\bh}(\Gamma),$ \[
d_{A}\langle \mu_{a},u \rangle(X) = \omega_{a}(D_{A}u,X) \quad \text{ for all } X \in T_{A}\B_{\bh}(\Gamma).
\]
Hence, $\mu$ is a hyperkähler moment map for the $\G_{0}(\Gamma)$-action.   
\end{Prop}

\begin{proof}
    We give the full proof in the case $a = 1;$ the cases $a=2,3$ are identical after relabeling. Differentiating the first component of $\langle \mu_1,u\rangle$ yields 
\begin{equation}
    d_{A}\bigr( \int_{I} \langle u,\dot{A}_{1} + [A_0,A_1] + [A_2,A_3] \rangle \bigl)(X) = \int_{I} \langle u, \dot{X}_{1} + [X_0,A_1] + [A_0,X_1] + [X_2,A_3] + [A_2,X_3] \rangle dt.
\end{equation}
Using integration by parts over each edge $e$, the term $\int_{I} \langle u,\dot{X}_1 \rangle dt$ becomes \begin{equation}
    \int_{I} \langle u, \dot{X}_1 \rangle dt = \sum_{e\in E} \bigr( \langle u^{e}(b_e),X_{1}(b_e) \rangle - \langle u^{e}(a_e),X_1(a_e) \rangle \bigl) - \int_{I} \langle \dot{u},X_1 \rangle dt . 
\end{equation}

Using the matching conditions $u^{e}(a_e) = u(s(e)), u^{e}(b_e) = u(t(e))$ and $u(v) = 0$ on $\partial \Gamma,$ the boundary sum equals \[
\sum_{v\in \Gamma_{\nt}} \bigg\langle u(v), \sum_{e\in t^{-1}(v)} X^{e}_{1}(b_e) - \sum_{e\in s^{-1}(v)} X^{e}_{1}(a_e) \bigg\rangle = -d_{A}(\text{  vertex part of } \langle \mu_1,u \rangle)(X).
\]

Thus, \begin{align*}
d_{A}\langle \mu_1,u \rangle (X) = \int_{I} \bigr( -\langle \dot{u},X_1 \rangle + \langle u,[X_0,A_1] + [A_0,X_1] + [X_2,A_3] + [A_2,X_3] \rangle dt \bigl). 
\end{align*}

By reorganizing each term using $\Ad$-invariance in the form $\langle u,[Y,Z] \rangle = \langle [u,Y],Z \rangle$ and the skew-symmetry $\langle u,[Y,Z] \rangle  = -\langle u,[Z,Y] \rangle :$

\begin{align*}
\langle u,[X_0,A_1] \rangle &= -\langle [u,A_1],X_0 \rangle, \quad \langle u,[A_0,X_1] \rangle = \langle [u,A_0],X_1 \rangle,\\
\langle u,[X_2,A_3] \rangle &= -\langle [u,A_3],X_2 \rangle, \quad \langle u,[A_2,X_3] \rangle = \langle [u,A_2],X_3 \rangle, 
\end{align*}

one gets \[
d_{A}\langle \mu_1,u \rangle(X) = \int_{I} \bigr( -\langle [u,A_1],X_0 \rangle + \langle [u,A_0] - \dot{u},X_1 \rangle - \langle [u,A_3],X_2 \rangle + \langle [u,A_2],X_3 \rangle \bigl).
\]

On the other hand, by Definition \ref{de 4.1} and \ref{DA}, \[
I_{1}D_{A}u = (-[u,A_1],[u,A_0] - \dot{u},-[u,A_3],[u,A_2]),
\]
so that \[
\omega_{1}(D_{A}u,X) = \langle I_{1}D_{A}u,X \rangle = \int_{I} \bigr( -\langle [u,A_1],X_0 \rangle + \langle [u,A_0] - \dot{u},X_1 \rangle - \langle [u,A_3],X_2  \rangle + \langle [u,A_2],X_3 \rangle \bigl) dt.
\]

The two expressions agree, proving $d_{A}\langle \mu_1,u \rangle = \omega_1(D_{A}u,\cdot).$ Equivariance of $\mu$ follows from the $\Ad$-invariance of the inner product together with Lemma \ref{tri hlm}.
\end{proof}

We adapt the hyperkähler reduction approach of \cite[Section D]{hitchin1987hyperkahler} to the infinite-dimensional case to show that the Nahm--Kirchhoff moduli space admits a hyperkähler structure. The first ingredient is that, for each complex structure $I_a,$ the pair $(\mu_{b},\mu_c)$ assembles into an $I_a$-holomorphic moment map whose zero locus is a complex Kähler submanifold of $\B_{\mathbb{H}}(\Gamma).$ The second is the weak symplectic orthogonality, condition \ref{cd 4} of Theorem \ref{infinite red}, on this submanifold. Throughout, $(a,b,c)$ denotes a cyclic permutation of $(1,2,3).$

\begin{Lem}[Holomorphic moment map] \label{holo map}
For each cyclic $(a,b,c) \in \{(1,2,3),(2,3,1),(3,1,2)\}$, set \[
\mu_{+}^{a} \coloneqq \mu_{b} + i\mu_{c}: \B_{\mathbb{H}}(\Gamma) \too C^{0}(I,\g_{\mathbb{C}}) \times \g^{\Gamma_{\nt}}_{\mathbb{C}}.
\]     
Then, $\mu_{+}^{a}$ is holomorphic with respect to the complex structure $I_{a}$ on $\B_{\bh}(\Gamma).$ Equivalently, for any $A\in \B_{\bh}(\Gamma), X\in T_{A}\B_{\bh}(\Gamma)$ and $u\in C^{2}_{0}(I,\g),$ \begin{equation} \label{eq 7}
    d_{A}\langle \mu^{a}_{+},u \rangle (I_{a}X) = id_{A}\langle \mu^{a}_{+}, u \rangle(X).
\end{equation}
\end{Lem}

\begin{proof}
    By the moment map identity of Proposition \ref{moment iden}, $d_{A}\langle \mu_{d},u \rangle (Y) = \omega_{d}(D_{A}u,Y) = \langle I_{d}D_{A}u,Y\rangle$ for every $d\in \{1,2,3\}.$ Hence, \[
    d_{A}\langle \mu^{a}_{+},u \rangle(X) = \langle I_{b}D_{A}u,X \rangle + i\langle I_{c}D_{A}u,X\rangle.
    \]
Replacing $X$ by $I_{a}X$ and using that each $I_{d}$ is a $\langle ,\rangle$-isometry (Lemma \ref{tri hlm}), we have $I_{a}^{*} = I_{a}^{-1} = -I_{a},$ so that $\langle I_{d}D_{A}u,I_{a}X \rangle = - \langle I_{a}I_{d}D_{A}u,X \rangle.$ The quaternion relations $I_{a}I_{b} = I_c$ and $I_{a}I_{c} = -I_{b}$ give \[
\langle I_{b}D_{A}u,I_{a}X \rangle = -\langle I_{c}D_{A}u,X \rangle, \quad \langle I_{c}D_{A}u,I_{a}X \rangle = \langle I_{b}D_{A}u,X\rangle.
\]

Therefore, \[
d_{A}\langle \mu^{a}_{+},u \rangle(I_{a}X) = -\langle I_{c}D_{A}u,X \rangle + i\langle I_{b}D_{A}u,X \rangle = i(\langle I_{b}D_{A}u,X \rangle + i\langle I_{c}D_{A}u,X \rangle),
\]
which is $\ref{eq 7}.$
    
\end{proof}

\begin{Cor}[Kähler Banach submanifold] \label{kah mani}For each $a\in \{1,2,3\},$ the level set \[
N^{a} \coloneqq (\mu_{+}^{a})^{-1}(0) = \mu_{b}^{-1}(0) \cap \mu_{c}^{-1}(0)
\]
is a complex Kähler Banach submanifold of $(\B_{\bh}(\Gamma),I_{a}),$ and the restrictions $$(N^{a},\langle \cdot,\cdot \rangle_{|N^{a}},I_{a}|_{N^{a}},\omega_{a}|_{N^{a}})$$ form a (weak) Kähler Banach manifold. The action of $\G_{0}(\Gamma)$ preserves $N^{a}$ and is by Kähler isometries, with moment map $\mu_{a}|_{N^{a}}.$ Moreover, $\A_{\bh}(\Gamma) = (\mu_{a}|_{N^{a}})^{-1}(0).$
\end{Cor}

\begin{proof}
By Lemma \ref{holo map}, $\mu_+^{a}$ is
$I_a$-holomorphic. By Proposition \ref{regular value}, $0$ is a regular value of the combined map $\mu=(\mu_1,\mu_2,\mu_3)$; since the projection onto the $(\mu_b,\mu_c)$-components preserves surjectivity of the differential, $0$ is a regular value of $\mu_+^{a}$, so $N^a$ is a Banach submanifold of $\B_{\bh}(\Gamma)$ with
\(T_AN^{a} = \ker d_A\mu_b \cap \ker d_A\mu_c\). The
holomorphicity \eqref{eq 7} forces $T_AN^{a}$ to be
$I_a$-stable, so $N^a$ inherits an integrable complex structure
(the complex structure $I_a$ on the ambient space $\B_{\bh}(\Gamma)$ is constant, hence trivially integrable, and a complex submanifold of a complex Banach manifold is complex).
 
The restriction $\omega_a|_{N^{a}}$ is the closed, weakly
non-degenerate form $\langle I_a\cdot,\cdot\rangle$ on the
complex submanifold; compatibility with the restricted metric is
inherited from $\B_{\bh}(\Gamma)$. Since $\G_0(\Gamma)$ acts on $\B_{\bh}(\Gamma)$ by tri-holomorphic isometries (Lemma \ref{tri hlm}) and preserves
$\mu_b^{-1}(0)\cap\mu_c^{-1}(0)$ by equivariance of $\mu_b,\mu_c$, it
preserves $N^a$ and acts by Kähler isometries. The restriction of
the hyperkähler moment map identity (Proposition \ref{moment iden}) to $T_AN^{a}$
gives \(d_A\langle\mu_a,u\rangle(X) = \omega_a(D_Au,X)\) for
$X\in T_AN^{a}$, so $\mu_a|_{N^{a}}$ is a moment map for the
$\G_0(\Gamma)$-action on the Kähler manifold $N^{a}$. The final
identity is the tautology
$\A_{\bh}(\Gamma) = \mu^{-1}(0) = \mu_a^{-1}(0)\cap N^{a}$.
\end{proof}

We now state and prove the weak symplectic orthogonality of the orbit and the level set inside $N^a,$ in the spirit of \cite[p. 123]{MarsdenWeinstein}.

\begin{Lem}[Orthogonality on $N^a$] \label{lemma 4.10} For every $A\in \A_{\bh}(\Gamma)$ and every $a\in \{1,2,3\},$ \begin{equation} \label{eq 8}
T_{A}(\G_{0}(\Gamma)\cdot A) = (T_{A}\A_{\bh}(\Gamma))^{\perp_{\omega_{a}}} \cap T_{A}N^a,
\end{equation}
where the symplectic orthogonal is computed with respect to $\omega_{a}|_{N^a}.$    
\end{Lem}

\begin{proof}

\noindent\textbf{Step 1: Topological splitting of} $T_{A}N^a.$ We extend Proposition \ref{prop 3.6} from $T_{A}\A_{\bh}(\Gamma).$ The operator \[
L_{A} \coloneqq D_{A}^{*}D_{A}: C^{2}_{0}(I,\g) \too C^{0}(I,\g) \times \g^{\Gamma_{\nt}}
    \]
is an isomorphism of Banach spaces (Proposition \ref{prop 3.6}); write $\Pi \coloneqq D_{A}L_{A}^{-1}D^{*}_{A}$ for the associated bounded projector on $T_{A}\B_{\bh}(\Gamma),$ so that $\text{im}(\Pi) = \text{im}(D_{A})$ and $\text{im}(1 - \Pi) = \Ker(D^{*}_{A}).$ Since the tangent space to the orbit $T_{A}(\G_{0}(\Gamma)\cdot A) = \text{im}(D_A)$ is contained in $T_{A}\A_{\bh}(\Gamma) \subset T_{A}N^a,$ both $\Pi$ and $1 - \Pi$ preserve $T_{A}N^a.$ Therefore,
\begin{equation} \label{eq 9}
    T_{A}N^a = \text{im} D_{A} \oplus (\Ker D^{*}_{A} \cap T_{A}N^a),
\end{equation}
    as a closed topological sum.\\

\noindent\textbf{Step 2: Easy inclusion} ($\supseteq$ in \ref{eq 8}). For $u\in C^{2}_{0}(I,\g)$ and $Y\in T_{A}\A_{\bh}(\Gamma),$ the moment map identity (Proposition \ref{moment iden}) and $Y \in \ker d_{A}\mu_{a}$ give \[
\omega_{a}(D_{A}u,Y) = d_{A}\langle \mu_{a},u \rangle(Y) = 0,
\]
so $\text{im}(D_{A}) \subseteq (T_{A}\A_{\bh}(\Gamma))^{\perp_{\omega_a}} \cap T_{A}N^a.$\\

\noindent\textbf{Step 3: Hard inclusion} ($\subseteq$ in \ref{eq 8}). Let $X \in T_{A}N^a$ with $\omega_{a}(X,Y) = 0$ for all $Y \in T_{A}\A_{\bh}(\Gamma).$ By \ref{eq 9}, we write \[
X = D_{A}u + Z, \quad u = L_{A}^{-1}(D^{*}_{A}X) \in C^{2}_{0}(I,\g), \quad Z \in \Ker D^{*}_{A} \cap T_{A}N^a.
\]
We claim that $Z = 0.$ By Step 2, $\omega_{a}(D_{A}u,Y) = 0$ for all $Y\in T_{A}\A_{\bh}(\Gamma),$ hence \begin{equation} \label{eq 10}
    \omega_{a}(Z,Y) = 0, \text{ for every } Y \in T_{A}\A_{\bh}(\Gamma).
\end{equation}
Now, we verify three properties of $Z:$ 
\begin{itemize}
    \item \noindent (i) $D^{*}_{A}(I_{a}Z) = 0.$ For every $v\in C^{2}_{0}(I,\g), D_{A}v \in \text{im} D_{A} \subset T_{A}\A_{\bh}(\Gamma),$ so \ref{eq 10} gives \[
    0 = \omega_{a}(Z,D_{A}v) = \langle I_{a}Z,D_{A}v \rangle = \langle D^{*}_{A}(I_{a}Z),v \rangle.
    \] 
    The pairing on the right-hand side, between $C^{0}(I,\g)\times \g^{\Gamma_{\nt}}$ and $C^{2}_{0}(I,\g),$ is continuously non-degenerate, so $D^{*}_{A}(I_{a}Z) = 0.$
    \item \noindent (ii) $I_{a}Z \in T_{A}\A_{\bh}(\Gamma):$ we verify the three conditions $d_{A}\mu_{d}(I_{a}Z) = 0$ for $d = a,b,c.$ For every $v\in C^{2}_{0}(I,\g),$ using that the $I_d$ are isometries ($I_{d}^{*} = -I_{d}$) and the quaternion relations $I_{a}I_{b} = I_{c}$ and $I_{a}I_{c} = -I_{b}:$  
    
\begin{align*}
  d_A\langle\mu_a,v\rangle(I_aZ)
  &= \langle I_aD_Av,I_aZ\rangle = \langle D_Av,Z\rangle
   = \langle v,D_A^*Z\rangle = 0,
  \quad\text{since }Z\in\ker D_A^*; \\
  d_A\langle\mu_b,v\rangle(I_aZ)
  &= \langle I_bD_Av,I_aZ\rangle = -\langle I_aI_bD_Av,Z\rangle
   = -\langle I_cD_Av,Z\rangle \\
  &= -d_A\langle\mu_c,v\rangle(Z) = 0,
  \quad\text{since }Z\in T_AN^{a}\subset\ker d_A\mu_c; \\
  d_A\langle\mu_c,v\rangle(I_aZ)
  &= \langle I_cD_Av,I_aZ\rangle = -\langle I_aI_cD_Av,Z\rangle
   = \langle I_bD_Av,Z\rangle \\
  &= d_A\langle\mu_b,v\rangle(Z) = 0,
  \quad\text{since }Z\in T_AN^{a}\subset\ker d_A\mu_b.
\end{align*}

\item \noindent (iii) $Z\in T_{A}\A_{\bh}(\Gamma):$ it remains to verify that
$d_A\mu_a(Z)=0$. For every $u\in C^2_0(I,\mathfrak g)$, using $I_a^*=-I_a$
and~(i),
\[
  d_A\langle\mu_a,u\rangle(Z)
  = \omega_a(D_Au,Z) = \langle I_aD_Au,Z\rangle
  = -\langle D_Au,I_aZ\rangle = -\langle u,D_A^*(I_aZ)\rangle = 0.
\]
\end{itemize}
We now conclude. By (ii), $I_{a}Z \in T_{A}\A_{\bh}(\Gamma).$ Applying \eqref{eq 10} to $Y = I_{a}Z$ gives \[
0 = \omega_{a}(Z,I_{a}Z) = \langle I_{a}Z,I_{a}Z \rangle = \|I_{a}Z\|^{2}_{L_2}.
\]    
Since the $L_{2}$ inner product is positive definite, $I_{a}Z = 0,$ hence $Z = 0.$ Thus, $X = D_{A}u \in \text{im}D_{A} = T_{A}(\G_{0}(\Gamma)\cdot A),$ which completes the proof of \ref{eq 8}.
\end{proof}

\begin{Rem} \label{rmq 4.12}
    Step 3(ii) of the proof of Lemma \ref{lemma 4.10} establishes that $I_{a}$ maps $T_{A}N^a \cap \Ker(D^{*}_{A})$ into $T_{A}\A_{\bh}(\Gamma) \cap \Ker(D^{*}_{A});$ the analogous computation with $Z \in T_{A}\A_{\bh}(\Gamma) \cap \Ker(D^{*}_{A})$ in place of $Z\in T_{A}N^{a}$ shows that \[
    I_{a}(T_{A}\A_{\bh}(\Gamma) \cap \Ker(D^{*}_{A})) \subset T_{A}\A_{\bh}(\Gamma) \cap \Ker(D^{*}_{A}) \quad (a = 1,2,3),
    \]
    which is precisely the hypothesis of Proposition \ref{q slice} (quaternionic slice criterion). This hypothesis is therefore a consequence of Lemmas \ref{holo map}--\ref{lemma 4.10} and the topological splitting of Proposition \ref{prop 3.6}. 
\end{Rem}

With Lemmas \ref{holo map}--\ref{lemma 4.10} in hand, the hyperkähler structure on $\M_{\bh}(\Gamma)$ is assembled by applying the (weak) Kähler reduction theorem to each Kähler Banach manifold $(N^a,\omega_{a}|_{N^{a}})$ with moment map $\mu_{a}|_{N^a}$ for $a = 1,2,3,$ in the spirit of \cite{hitchin1987hyperkahler}. The three reductions share the same underlying quotient $\A_{\bh}(\Gamma),$ the same descended $L^{2}$-metric, and the same horizontal subspace $T_{A}\A_{\bh}(\Gamma) \cap \Ker D^{*}_{A}$ at each $A$; the descended complex structures $I_{a,\text{red}}$ inherit the quaternion relations from the ambient $I_{a}$ by Remark \ref{rmq 4.12}.

\subsection{The complex quotient} We prove an infinite-dimensional quiver version of the equivalence between hyperkähler quotients and complex-symplectic quotients. We fix the first complex structure $I_1$ and consider $\alpha = A_{0} + iA_{1}, \beta = A_{2} + iA_{3}$, viewed as maps in $C^{1}(I,\g_{\mathbb{C}}).$ Equivalently, we identify $\B_{\mathbb{H}}(\Gamma)$ with $\B^{2}_{\mathbb{C}}(\Gamma)$, where the latter is an infinite-dimensional complex-symplectic manifold. Then, splitting the hyperkähler moment map into a complex part $\mu_{\mathbb{C}} \coloneqq \mu_{2} + i\mu_{3}: \B_{\mathbb{C}}(\Gamma) \too C^{0}(I,\g_{\mathbb{C}}) \times \g_{\mathbb{C}}^{\Gamma_{\nt}}$ and a real part $\mu_{1}: \B_{\bh}(\Gamma) \too C^{0}(I,\g) \times \g^{\Gamma_{\nt}},$ the Nahm--Kirchhoff equations are equivalent to the following: \begin{align*}
    \dot{\beta} + [\alpha,\beta] &= 0, \quad \beta \textbf{ satisfies Kirchhoff matching}\\
    \dot{\alpha} +\dot{\alpha}^{*} + [\alpha,\alpha^{*}] + [\beta,\beta^{*}] &= 0, \quad A_1 \textbf{ satisfies Kirchhoff matching}.
\end{align*} 

The first part is called \textit{the complex Nahm equation} and the second \textit{the real Nahm equation}. As expected from the finite-dimensional case, $\mu_{\mathbb{C}}$ is a complex moment map for the action of the complexified gauge group $\G^{0}_{\mathbb{C}}(\Gamma)$, which acts by \[
g\cdot (\alpha,\beta) = (\Ad_{g}\alpha - \dot{g}g^{-1}, \Ad_{g}\beta).
\]

Hence, we wish to prove that the quotient $\mu^{-1}(0)/\G^{0}(\Gamma)$ is biholomorphic to $\mu^{-1}_{\mathbb{C}}(0)/\G^{0}_{\mathbb{C}}(\Gamma).$ In other words, we want to prove the generalization of Donaldson's theorem (see \cite{hitchin1987hyperkahler}) to the case of quivers.

\begin{Thm}[Donaldson \cite{hitchin1987hyperkahler}-Quiver] \label{complex q}
    Let $\Gamma$ be an oriented connected quiver with boundary and $G$ a compact connected Lie group. For any solution $(\alpha,\beta)$ to the complex Nahm equations, there exists $g \in \G^{0}_{\mathbb{C}}(\Gamma)$ such that $g\cdot (\alpha,\beta )$ is a solution to the real Nahm equation. Moreover, $g$ is unique up to left multiplication by $\G^{0}(\Gamma).$ In other terms, the inclusion $\A_{\bh}(\Gamma) \hookrightarrow \A_{\mathbb{C}}(\Gamma)$ descends to a biholomorphism \[
    \M_{\bh}(\Gamma) \coloneqq \A_{\bh}(\Gamma)/\G^{0}(\Gamma) \too \A_{\mathbb{C}}(\Gamma)/\G^{0}_{\mathbb{C}}(\Gamma) \coloneqq \mu^{-1}_{\mathbb{C}}(0)/\G^{0}_{\mathbb{C}}(\Gamma).
    \]
\end{Thm}

Before we treat the quiver case (Theorem \ref{complex q}), we give a proof of Donaldson's theorem in the case of a single edge. Recall that $H = G_{\mathbb{C}}/G$ is a Riemannian manifold such that: \begin{enumerate}
    \item $H$ is a Cartan--Hadamard manifold: simply connected, complete and of non-positive sectional curvature; 
    \item geodesics through the identity coset $G$ are of the form $\gamma(s) = G\exp(isw), \textbf{ for } w \in \g;$ 
    \item the sectional curvature satisfies $K(u,v) = -\dfrac{1}{4} \| [u,v]\|^{2} \leq 0;$
    \item between any two points of $H$ there exists a unique minimizing geodesic. 
\end{enumerate}

\begin{Def}
    A smooth function $\phi: H\too \R$ is \textit{geodesically convex} if for any geodesic $\gamma: \R \too H,$ the composition $\phi \circ \gamma$ is convex in the usual sense: $(\phi \circ \gamma)^{''}(t) \geq 0, \text{ for all } t.$ It is \textit{strictly geodesically convex} if $(\phi \circ \gamma)^{''}(t) > 0$ for all $t.$
\end{Def}

\begin{Prop} \label{potential}
    For each $x\in \g_{\mathbb{C}},$ the function $V_{x}: H\too \R$ defined by \[
    V_{x}(h) = \|\Ad_{h}(x)\|^{2},
    \]
    where $h\in G_{\mathbb{C}}$ is any representative of the coset and the norm is computed using the $\Ad_{G}$-invariant inner product on $\g$ extended to $\g_{\mathbb{C}},$ is well-defined, non-negative, smooth, and geodesically convex.
\end{Prop}

\begin{proof}
Well-definedness follows from the $\Ad_{G}$-invariance of the norm. Non-negativity and smoothness are immediate. For geodesic convexity, let $\gamma(t) = G\exp(itw)g$ be a geodesic in $H$ for $w\in \g$ and $g\in G_{\mathbb{C}}.$ Then, \[
\dfrac{d^{2}}{dt^{2}} V_{x}(\gamma(t)) = 2 \|[w,\Ad_{\exp(itw)g}x]\|^{2} \geq 0.
\]  
This establishes geodesic convexity.
\end{proof}

Now, for Donaldson's proof to work, a Kähler potential for the first complex structure $I_1$ is needed. Here we consider \begin{equation} \label{eq 11}
f: \A_{\bh} = \A^{2}_{\mathbb{C}} \too \R, \quad f(\alpha,\beta) = \dfrac{1}{8} \int_{I} \|\alpha + \alpha^{*}\|^2 + 2 \|\beta\|^{2},
\end{equation} 

or \[
f(A) = \dfrac{1}{4} \int_{I} 2\|A_1\|^2 + \|A_2\|^2 + \|A_3\|^2.
\]

The function $f$ is a global Kähler potential such that $\mu_{I_{1}}(A)(x) = df(I_{1}x^{*}_{A})$ for $x\in \Lie(\G^{0}).$ Now, fix a solution $(\alpha,\beta)$ of the complex Nahm equation and consider the functional \[
F: \G^{0}_{\mathbb{C}} \too \R, \quad F(g) = f(g\cdot (\alpha,\beta)).
\]
Then, $F$ is \textit{a Kempf--Ness functional} in the sense that $g$ is a critical point of $F$ if and only if $g\cdot (\alpha,\beta)$ is a solution to the real Nahm equation. Moreover, $F$ is invariant under the action of $\G^{0},$ so we can view $F$ as a function on $C^{2}(I,H)^{0},$ i.e., the set of maps $h$ with values in $H$ such that $h(0) = h(1) = 0.$ As before, there exists $g\in \G_{\mathbb{C}}$ such that $g\cdot (\alpha,\beta) = (0,x)$ for some $x\in \g_{\mathbb{C}}.$ We may therefore assume that $(\alpha,\beta) = (0,x)$, but we must then find a minimum $h$ with fixed boundary points $h(0) = Gg(0)^{-1}, h(1) = Gg(1)^{-1}.$ Thus, we reduce the problem to showing that, for all boundary points $p,q \in H,$ the functional 
\begin{align*}
F: C^{2}(I,H)& \too \R\\
h& \mtoo f(h\cdot (0,x)) = \int_{I} \|\dot{h}\|_{H}^{2} + V_{x}(h),
\end{align*}
has a unique minimum with $h(0) = p, h(1) = q.$ This is a variational problem that can be reduced to the finite-dimensional case. In other terms, for $x\in \g_{\mathbb{C}}$ and $p,q \in H$ we consider the \textbf{per-edge action functional} \[
E_{x}(p,q) = \text{inf} \Big\{ \int_{0}^1 \|\dot{h}\|^{2} + V_{x}(h(t)) dt: h\in C^{2}([0,1],H), h(0) = p , h(1) = q \Big\}.  
\]

Then, $F$ admits a unique minimizer $h^{*}$ if and only if $F(h^{*}) = E_{x}(p,q).$ In the next lemmas, we prove the existence and uniqueness of $h^{*}$ and that the assignment $(p,q) \mtoo h^{*}_{p,q}$ is smooth. Since the proof relies on the variational formulas for the per-edge action functional, we derive these formulas here.

\subsection{Kempf--Ness functional on an edge} \label{subsec:kn-edge}
The goal of this subsection is to study the functional
\[
S(h) := \int_{0}^{1} \|\dot{h}\|_{H}^{2} + V_{x}(h(t))\, dt,
\]
and to establish the existence, uniqueness, and smooth dependence on the endpoints of its minimizer, together with the variational formulas that drive the convexity argument. Let us start with the following basic result.

\begin{Lem} \label{lem mini}
For every $p,q \in H$ and $x\in \g_{\mathbb{C}},$ there exists $h^{*}\in C^{2}([0,1],H)$ with $h^{*}(0) = p,\ h^{*}(1) = q$ minimizing the action
\[
S(h) \coloneqq \int_{0}^{1} \|\dot{h}\|_{H}^{2} + V_{x}(h(t)) \, dt.
\]
The minimizer satisfies the Euler--Lagrange equation
\begin{equation} \label{euler eq}
2\nabla_{\dot{h}^{*}}\dot{h}^{*} = \grad V_{x}(h^{*}).
\end{equation}
\end{Lem}
\begin{proof}
    The fact that the critical points of $S$ satisfy the Euler--Lagrange equation $\eqref{euler eq}$ is straightforward. To show that minimizers exist, we proceed as follows.
     Consider a minimizing sequence $(h_n)_{n\in \mathbb{N}}$ for $S,$ i.e. $S(h_n) \too \inf S.$ Since $V_x \geq 0,$
\[
\int_{0}^{1} \|\dot{h}_{n}\|^{2}\,dt \leq \int_{0}^{1} \|\dot{h}_{n}\|^{2} + V_x(h_n(t))\,dt = S(h_n),
\]
so there is a constant $C>0$ with $\int_{0}^{1} \|\dot{h}_{n}\|^{2}\,dt < C.$ By Cauchy--Schwarz,
\[
d_H(h_n(t),p) \leq \int_{0}^{t} \|\dot{h}_{n}(s)\|\,ds \leq \Big(\int_{0}^{1} \|\dot{h}_{n}\|^{2}\,dt\Big)^{\frac{1}{2}} \leq C^{\frac{1}{2}},
\]
so $h_n([0,1]) \subset \overline{B_H}(p,C^{\frac{1}{2}}).$ The curves are moreover equicontinuous:
\[
d_H(h_n(t),h_n(s)) \leq \int_{s}^{t} \|\dot{h}_{n}(r)\|\,dr \leq C^{\frac{1}{2}}|t-s|^{\frac{1}{2}}.
\]
By Arzel\`a--Ascoli, there is a uniformly convergent subsequence $h_{n_k} \too h_*$ with $h_*(0)=p$ and $h_*(1)=q.$ Then $\int_{0}^1 V_{x}(h_{n_k}) \too \int_{0}^1 V_{x}(h_*)$ and, by lower semicontinuity of the Dirichlet energy for curves into a Riemannian manifold, $\int_{0}^{1} \|\dot{h}_{*}\|^{2} \leq \liminf \int_{0}^{1} \|\dot{h}_{n_k}\|^{2}.$ Hence
\[
S(h_*) = \int_{0}^1 \|\dot{h}_{*}(t)\|^{2} + V_{x}(h_{*}(t))\,dt \leq \liminf \int_{0}^1 \|\dot{h}_{n_k}\|^2 + V_{x}(h_{n_k})\,dt = \liminf S(h_{n_k}) \leq \inf S,
\]
so $h_*$ is a minimizer and a weak solution of \eqref{euler eq}, which is an ODE on $H.$ A standard bootstrap argument shows that $h_*$ is smooth.
\end{proof}

To show that the minimizers of the action functional $S$ are unique, we will use the following convexity property.\\

We consider two geodesics
$\gamma_{1},\gamma_{2}: \R \too H.$ For each $s\in \R,$ we consider the two-parameter
family of smooth curves in $H$
\[
\Phi:\R \times [0,1] \too H, \quad \Phi(s,t):= h_{s}(t),
\]
and \textbf{the variational field} along $h_{s}$
\[
J(t) \coloneqq \dfrac{\partial \Phi}{\partial s}(s,t) \in T_{h_{s}(t)}H,
\]
which is a smooth vector field along $h_s$ satisfying the boundary conditions
\begin{equation} \label{boundary conditions}
J(0) = \dot{\gamma}_{1}(s), \quad J(1) = \dot{\gamma}_{2}(s).
\end{equation}

\begin{Prop}[First variation] \label{first var}
Consider $\phi(s)=S(h_s)$. Then
\[
\phi^{'}(s) =\delta(S)_{h_s}(\dot{h}_s)= 2\langle \dot{h}_{s}(1),\dot{\gamma}_{2}(s) \rangle - 2\langle \dot{h}_{s}(0),\dot{\gamma}_{1}(s) \rangle,
\]
where $\delta(S)_{h_s}(\dot{h}_s)$ is the variation of $S$ at $h_s$ in the direction of $\dot{h}_s$.
\end{Prop}

\begin{proof}
Differentiating,
\[
\dot{\phi}(s) = \int_{0}^1 2\langle \nabla_{s}\dot{h}_{s},\dot{h}_{s} \rangle + dV_{x}(J) \, dt.
\]
Since the Levi-Civita connection is torsion-free and $[\partial_{s},\partial_t] = 0,$
\begin{equation} \label{equ torsion free}
\nabla_{s}\dot{h}_{s} = \nabla_{s}\partial_{t}\Phi = \nabla_{t}\partial_{s}\Phi = \nabla_t J.
\end{equation}
Hence $\phi^{'}(s) = \int_{0}^{1} 2\langle \nabla_{t}J,\dot{h}_{s} \rangle + dV_{x}(J)\,dt.$ Integrating the first term by parts,
\[
\int_{0}^{1} 2\langle \nabla_{t}J,\dot{h}_{s} \rangle \, dt = 2[\langle J,\dot{h}_{s} \rangle]^{1}_{0} - \int_{0}^{1} 2\langle J,\nabla_{t}\dot{h}_s \rangle \, dt.
\]
Using the Euler--Lagrange equation $2\nabla_{t}\dot{h}_{s} = \grad V_{x}(h_s)$ and $dV_{x}(J) = \langle \grad V_x, J\rangle,$
\[
\phi^{'}(s) = 2[\langle J,\dot{h}_{s} \rangle]^{1}_{0} -\int_{0}^{1} \langle \grad V_x, J \rangle + \int_{0}^{1} \langle \grad V_x, J\rangle \, dt = 2[\langle J, \dot{h}_s \rangle]^{1}_{0}.
\]
Substituting the boundary conditions \eqref{boundary conditions} gives the result.
\end{proof}

\begin{Prop}[Second variation] \label{s va}
We have
\[
\phi^{''}(s) = \int_{0}^{1} 2\|\nabla_{t}J \|^{2} + 2\langle R(J,\dot{h}_{s})J,\dot{h}_{s} \rangle + \mathrm{Hess}(V_x)(J,J)\,dt,
\]
where $R$ is the Riemann curvature tensor of $H.$
\end{Prop}

\begin{proof}
We differentiate $\phi^{'}(s) = \int_{0}^{1} 2\langle \nabla_{t}J,\dot{h}_s \rangle$ with respect to $s,$ treating each term separately. Set $K_{s} \coloneqq \nabla_{s}J.$

\noindent\textbf{Term 1.} For the kinetic term,
\[
\dfrac{d}{ds} \langle \nabla_{t}J,\dot{h}_{s} \rangle = \langle \nabla_{s}\nabla_{t}J,\dot{h}_{s} \rangle + \langle \nabla_{t}J,\nabla_{s}\dot{h}_s \rangle.
\]
By definition of the Riemann curvature tensor, $\nabla_{s}\nabla_{t}J - \nabla_{t}\nabla_{s}J = R(\partial_{s}\Phi,\partial_{t}\Phi)J = R(J,\dot{h}_{s})J.$ Combined with \eqref{equ torsion free},
\[
\dfrac{d}{ds} \langle \nabla_{t}J,\dot{h}_s \rangle = \langle \nabla_{t}K_s,\dot{h}_{s} \rangle + \langle R(J,\dot{h}_{s})J,\dot{h}_s \rangle + \|\nabla_{t}J\|^{2}.
\]

\noindent\textbf{Term 2.} For the potential term,
\[
\dfrac{d}{ds} dV_{x}(J) = \langle \nabla_{s} \grad V_{x},J \rangle + \langle \grad V_{x},K_{s} \rangle = \mathrm{Hess}(V_x)(J,J) + dV_{x}(K_s).
\]
Integrating and adding gives
\[
\phi^{''}(s) = \int_{0}^{1} \Big( 2\| \nabla_{t}J \|^{2} + 2\langle R(J,\dot{h}_{s})J,\dot{h}_s \rangle + \mathrm{Hess}(V_x)(J,J) \Big)dt + \underbrace{\int_{0}^{1} 2\langle \nabla_{t}K_s,\dot{h}_{s} \rangle + dV_{x}(K_s)\,dt}_{(*)}.
\]
We claim that $(*) = 0.$ Integrating the first part of $(*)$ by parts,
\[
\int_{0}^{1} 2\langle \nabla_{t}K_{s},\dot{h}_{s}\rangle \, dt = 2[\langle K_{s},\dot{h}_{s} \rangle]^{1}_{0} - \int_{0}^{1} 2\langle K_{s},\nabla_{t}\dot{h}_{s} \rangle = 2[\langle K_{s},\dot{h}_{s} \rangle]^{1}_{0} - \int_{0}^{1} dV_{x}(K_s)\,dt,
\]
where we used the Euler--Lagrange equation in the second step. Hence $(*) = 2[\langle K_{s},\dot{h}_{s} \rangle]^{1}_{0}.$ Finally, at $t = 0,$ $\Phi(s,0) = \gamma_{1}(s),$ so $K_{s}(0) = \nabla_{s}\dot{\gamma}_{1}(s) = 0$ since $\gamma_{1}$ is a geodesic; at $t = 1,$ similarly, $K_{s}(1) = 0.$ Therefore $(*) = 0,$ which completes the proof.
\end{proof}

If we fix the endpoints of our curves to be $\gamma_i(0)=p$ and $\gamma_i(1)=q$ for $i=1,2$, we obtain a simplified expression for the two variations.

\begin{Cor} \label{fixed endpoints}
For two smooth geodesics $\gamma_1,\gamma_2$ with $\gamma_i(0)=p$ and $\gamma_i(1)=q$ for $i=1,2$, we have $\phi^{'}(s)=0$ and
\[
\phi^{''}(s) =\operatorname{Hess}(S)(J,J)= \int_{0}^{1} 2\|\nabla_{t}J \|^{2} + 2\langle R(J,\dot{h}_{s})J,\dot{h}_{s} \rangle + \mathrm{Hess}(V_x)(J,J)\,dt.
\]
\end{Cor}

This allows us to establish the convexity of the Kempf--Ness functional along smooth geodesics.

\begin{Lem} \label{convexity thm}
With the notation above,
\begin{equation}\label{kt}
\| \nabla_{t}J \|^{2} \geq 0,
\end{equation}
with equality at $t$ if and only if $J$ is parallel along $h_s$ at that point, and
\begin{equation}\label{ct}
\langle R(J,\dot{h}_{s})J,\dot{h}_s \rangle \geq 0,
\end{equation}
with equality at $t$ if and only if either $J$ and $\dot{h}_{s}$ are linearly dependent or the sectional curvature of the plane they span vanishes. In particular, the Kempf--Ness functional is convex along smooth geodesics.
\end{Lem}

\begin{proof}
The sectional curvature of the plane $\sigma = \mathrm{span}(J,\dot{h}_{s})$ is
\[
K_{\sigma} = \dfrac{\langle R(J,\dot{h}_{s})\dot{h}_{s},J\rangle}{\|J\|^{2} \|\dot{h}_{s}\|^{2} - \langle J,\dot{h}_{s}\rangle^{2}}.
\]
By the symmetries of the curvature tensor,
\[
\langle R(J,\dot{h}_{s})J,\dot{h}_{s} \rangle = -K(\sigma)\big(\|J\|^{2}\|\dot{h}_{s}\|^{2} - \langle J,\dot{h}_{s} \rangle^{2}\big).
\]
Since $K(\sigma) \leq 0$ on $H$ and the Cauchy--Schwarz factor is non-negative, \eqref{ct} follows. Inequality \eqref{kt} is immediate. Since both integrands are non-negative, $\phi^{''}(s)\ge 0$ by Corollary \ref{fixed endpoints}.
\end{proof}

\begin{Prop}[Joint geodesic convexity] \label{jgc}
For each $x\in \g_{\mathbb{C}},$ the function $E_{x}: H\times H \too \mathbb{R}$ is jointly geodesically convex: for any geodesic $\Gamma(s) = (\gamma_{1}(s),\gamma_{2}(s))$ in $H\times H,$ the composition $s\too E_{x}(\gamma_{1}(s),\gamma_{2}(s))$ is convex.
\end{Prop}

\begin{proof}
By Corollary \ref{fixed endpoints}, the three integrands of $\phi^{''}(s)$ are non-negative: the first two by \eqref{kt}--\eqref{ct}, and $\mathrm{Hess}(V_x)(J,J)\ge 0$ by the geodesic convexity of $V_x$ (Proposition \ref{potential}). Hence $\phi^{''}(s)\ge 0$ and $\phi$ is convex.
\end{proof}

\begin{Rem}
All the previous results remain true if we replace $[0,1]$ by an arbitrary compact interval $[a,b].$
\end{Rem}

\begin{Rem}
Joint convexity cannot be upgraded to \textit{joint strict convexity} at the level of a single edge. If $\phi^{''}(s) = 0,$ then $\nabla_{t}J = 0,$ so $J$ is parallel along $h_{s_0}.$ With the boundary conditions $J(0) = \dot{\gamma}_{1}(s_{0})$ and $J(1) = \dot{\gamma}_{2}(s_0),$ neither endpoint forces $J = 0$: the parallel transport of $\dot{\gamma}_{1}(s_0)$ along $h_{s_0}$ from $t=0$ to $t=1$ may equal $\dot{\gamma}_{2}(s_0),$ making the term vanish even when both geodesics are non-constant. The strict convexity for the quiver problem is recovered globally in the next subsection via a propagation argument from the boundary.
\end{Rem}

We now establish uniqueness and smooth dependence of the minimizer.

\begin{Lem} \label{lem uni}
For every $p,q \in H$ and $x\in \g_{\mathbb{C}},$ the solution $h^{*}\in C^{2}([0,1],H)$ with $h^{*}(0) = p,\ h^{*}(1) = q$ to the Euler--Lagrange equation 
\begin{equation}
2\nabla_{\dot{h}^{*}}\dot{h}^{*} = \grad V_{x}(h^{*}),
\end{equation}
minimizing the action
\[
S(h) \coloneqq \int_{0}^{1} \|\dot{h}\|_{H}^{2} + V_{x}(h(t)) \, dt,
\]
is unique, and the map $(p,q) \too h^{*}_{p,q}$ depends smoothly on the endpoints.
\end{Lem}
\begin{proof}
\noindent\textbf{Uniqueness.} Suppose $h^{*}_{0},h^{*}_1$ are two minimizers. For $\tau \in [0,1],$ let
\[
h_{\tau}(t) = \exp_{h^{*}_{0}(t)}\!\big(\tau \exp^{-1}_{h^{*}_{0}(t)}(h^{*}_{1}(t))\big)
\]
be the pointwise geodesic interpolation, so that $h_0 = h^{*}_{0},\ h_1 = h^{*}_{1},$ and each $h_\tau$ has endpoints $p,q.$ By convexity of the energy density under geodesic interpolation in non-positive curvature, together with the geodesic convexity of $V_x$ (Proposition \ref{potential}),
\[
S(h_{\tau}) \leq (1-\tau)S(h^{*}_{0}) + \tau S(h^{*}_1) \leq S(h^{*}_{0}).
\]
Since $h^{*}_{0}$ is a minimizer, equality holds throughout and $\tau \mapsto S(h_\tau)$ is constant. Let $X(t) = \partial_\tau h_\tau(t)|_{\tau = 0} = \exp^{-1}_{h^{*}_{0}(t)}(h^{*}_{1}(t)),$ so that $X(0) = X(1) = 0.$ Constancy and Corollary \ref{fixed endpoints} give
\[
0=\tfrac{d^2}{d\tau^2}S(h_\tau)\big|_{\tau=0} = \int_{0}^{1} 2\|\nabla_{t}X\|^{2} + 2\langle R(X,\dot{h}^{*}_0)X,\dot{h}^{*}_0\rangle + \mathrm{Hess}(V_x)(X,X)\,dt.
\]
Each integrand is non-negative by \eqref{kt}--\eqref{ct} and Proposition \ref{potential}, so $\nabla_{t}X \equiv 0$; since $X(0) = 0,$ parallelism forces $X \equiv 0,$ hence $h^{*}_{0} = h^{*}_{1}.$\\

\noindent\textbf{Smooth dependence.} Consider the Banach space $\PP = \{ h \in C^{2}([0,1],H)\}$ and the map
\[
\F: \PP \too H \times H \times \E, \quad h \mtoo \big(h(0),\,h(1),\,2\nabla_{\dot{h}}\dot{h} - \grad V_{x}(h)\big),
\]
so that $\F(h^{*}_{p,q}) = (p,q,0)$. Define $\L: H \times H \times \PP \too H \times H \times \E$ by $\L(p,q,h) \coloneqq \F(h) - (p,q,0)$. With this definition, we have $\L(p,q,h^{*}_{p,q}) = 0.$ A tangent vector to $\PP$ at $h^{*}_{p,q}$ is a $C^2$ vector field $J$ along $h^{*}_{p,q},$ and the derivative of $2\nabla_{\dot{h}}\dot{h} - \grad V_{x}(h)$ is the Jacobi-type operator
\[
LJ \coloneqq 2\nabla^{2}_{t}J + 2R(J,\dot{h}^{*})\dot{h}^{*} - \nabla_{J}\grad V_{x},
\]
so that
\[
D_{h^{*}_{p,q}}\F(J) = \big(J(0),\,J(1),\,LJ\big).
\]
For $v_{0},v_1 \in T_{p}H \times T_{q}H$ and $f$ a vector field along $h^{*}_{p,q},$ the equation $D_{h^{*}_{p,q}}\F(J) = (v_0,v_1,f)$ is uniquely solvable if and only if the homogeneous Dirichlet problem $LJ = 0,\ J(0) = J(1) = 0$ has only the trivial solution. Pairing $LJ$ with $J$ and integrating by parts on the interval $[0,1],$ we obtain
\[
0= \int_{0}^{1} \langle LJ,J \rangle \, dt =-\operatorname{Hess}(S)(J,J) = \int_{0}^{1} 2\|\nabla_{t}J\|^{2} + 2\langle R(J,\dot{h}^{*})J,\dot{h}^{*}\rangle + \mathrm{Hess}(V_x)(J,J) \, dt.
\]
By the convexity argument above, we obtain $\nabla_{t}J \equiv 0,$ and $J(0) = 0$ forces $J \equiv 0.$ Hence $D_{h^{*}_{p,q}}\F$ is a Banach-space isomorphism, and by the implicit function theorem there is a smooth map $(p',q') \too h^{*}_{p',q'}$ near $(p,q)$ solving $\F(h) = (p',q',0).$ By uniqueness of the minimizer, this is $h^{*}_{p',q'},$ which establishes smooth dependence.
\end{proof}

\begin{Rem} \label{weak H1 existence}
The existence of the minimizer in Lemma \ref{lem mini} can equally be obtained by weak compactness. Since $V_x \geq 0,$ the functional $S$ is bounded below by $\int_{0}^{1} \|\dot{h}\|^{2} \geq d_{H}(p,q)^{2},$ and a minimizing sequence satisfies $\int_{0}^{1} \|\dot{h}_{n}\|^{2} \leq C.$ Since $H$ is complete, simply connected, and of non-positive curvature, the Sobolev space $H^{1}([0,1],H)$ with fixed endpoints is weakly compact (see, e.g., Jost \cite{jost2005riemannian}), so $h_n$ converges weakly to some $h^{*}$ in $H^{1}.$ The compact embedding $H^{1}([0,1],H) \hookrightarrow C^{0}([0,1],H)$ upgrades this to uniform convergence, whence $\int_{0}^1 V_{x}(h_n) \too \int_{0}^1 V_{x}(h^{*})$ by continuity of $V_x$ and dominated convergence; weak lower semicontinuity of the Dirichlet energy then gives $S(h^{*}) \leq \liminf S(h_n) \leq \inf S.$
\end{Rem}

\subsection{The Kempf--Ness functional on the quiver}
We seek $g\in \G^{0}_{\mathbb{C}}(\Gamma)$ such that $g\cdot(0,x_{e})$ satisfies the real Nahm equation on each edge. Since $f$ is $\G_{0}(\Gamma)$-invariant, the Kempf--Ness functional $F(g) = f(g\cdot (0,x))$ descends to a function on paths $(h_{e})_{e\in E}$ in $H^{|E|}$ with prescribed boundary values and matching at interior vertices.

\begin{Def}
    The \textbf{quiver Kempf--Ness functional} is \begin{equation}
        F((h_e)_{e\in E}) = \sum_{e\in E} \int_{a_e}^{b_e} (\|\dot{h}_{e}\|^{2}_{H} + V_{x_{e}}(h_{e}(t)))dt,
    \end{equation}

    defined on the space $\mathcal{P}$ of tuples $(h_{e})_{e\in E}$ with $h_{e} \in C^{2}([0,1],H)$ such that: \begin{itemize}
        \item at each boundary vertex $v\in \partial \Gamma,$ the endpoint value is prescribed by the boundary data;
        \item at each interior vertex $v\in \Gamma_{\nt}$, $h_{e_1}(b_e) = h_{e_{2}}(a_e)$ for all $e_{1},e_{2}$ meeting at $v.$
    \end{itemize}
\end{Def}

The Kempf--Ness functional reduces to a finite-dimensional problem through the following definition. \begin{Def}
    Let $\Gamma = (E,V,s,t)$ be a connected quiver with non-empty boundary and interior vertices $\Gamma_{\nt} = V-{\partial \Gamma}.$ Fix $x^{e}\in \g_{\mathbb{C}}.$ The \textbf{reduced Kempf--Ness functional} is \[
    \bar{F}: H^{|\Gamma_{\nt}|} \too \R, \quad \bar{F}(p) \coloneqq \sum_{e\in E} E_{x_{e}}(p_{s(e)},p_{t(e)}),
    \]
    where $p = (p_{v})_{v\in \Gamma_{\nt}},$ and boundary values are used when $s(e)$ or $t(e)$ lies in $\partial \Gamma$.
\end{Def}

\begin{Thm} \label{kempf ness}
    Let $\Gamma$ be a connected quiver with non-empty boundary, and fix $x^e \in \g_{\mathbb{C}}$ for each edge and boundary values for each $v\in \partial \Gamma.$ Then $\bar{F}$ is strictly geodesically convex. In particular, $\bar{F}$ has a unique minimum.
\end{Thm}

\begin{Prop}
    The Riemannian product $H^{N}$ (for any $N \geq 1$) is a simply connected, complete Riemannian manifold of non-positive sectional curvature. In particular, it is a Cartan--Hadamard manifold.
\end{Prop}

\begin{proof}
    Simple connectedness and completeness follow from the corresponding properties of $H.$ For the curvature, let $(u_1,...,u_N)$ and $(v_1,...,v_N)$ be orthonormal tangent vectors at a point of $H^{N}.$ The sectional curvature of the product is \[
    K((u)_i,(v)_i) = \dfrac{\sum_{i=1}^N K_{i}(u_i,v_i)\| u_{i} \wedge v_{i}\|^{2}}{\sum_{i=1}^N \|u_{i}\wedge v_{i}\|^2 + \sum_{i<j} \| u_{i}\wedge v_{j} + u_{j}\wedge v_{i}\|^{2}} \leq 0,
    \]
    since each $K_i \leq 0.$
\end{proof}

\begin{Prop} \label{gc}
    The reduced Kempf--Ness functional $\bar{F}: H^{\Gamma_{\nt}} \too \R$ is geodesically convex.
\end{Prop}

\begin{proof}
    Geodesics on the Riemannian product $H^{|\Gamma_{\nt}|}$ are products of geodesics: if $\gamma(s) = (\gamma_{v}(s))_{v\in \Gamma_{\nt}},$ each $\gamma_{v}$ is a geodesic in $H.$ For each edge $e\in E,$ the values $(\gamma_{s(e)}(s),\gamma_{t(e)}(s))$ trace a geodesic in $H \times H,$ where boundary vertices contribute constant components (a constant path is a geodesic with zero velocity). By Proposition \ref{jgc}, each summand $s\too E_{x_{e}}(\gamma_{s(e)}(s),\gamma_{t(e)}(s))$ is convex. Since a sum of convex functions is convex, \[
    \dfrac{d^2}{ds^{2}} \bar{F}(\gamma(s)) = \sum_{e\in E} \dfrac{d^{2}}{ds^{2}} E_{x^{e}}(\gamma_{s(e)}(s),\gamma_{(t(e))}(s)) \geq 0.
    \]
\end{proof}

We now prove the strict joint convexity of $\bar{F}.$ The argument uses the connectedness of $\Gamma$ and the boundary conditions in a way that mirrors the freeness argument in the Lax--Kirchhoff paper \cite{maiza2026lax}.

\begin{Prop} \label{sgc}
    The reduced functional $\bar{F}: H^{\Gamma_{\nt}} \too \R$ is strictly geodesically convex: for every non-constant geodesic $\gamma(s)$ in $H^{\Gamma_{\nt}}, \dfrac{d^{2}}{ds^{2}} \bar{F}(\gamma(s)) > 0.$
\end{Prop}

\begin{proof}
    Let $\gamma(s) = (\gamma_{v}(s))_{v\in \Gamma_{\nt}}$ be a non-constant geodesic, so that at least one component $\gamma_{v_{*}}$ is non-constant. For each edge $e\in E,$ let $h^{e}_{s}$ be the minimizing path from $\gamma_{s(e)}(s)$ to $\gamma_{t(e)}(s)$ (using the fixed boundary value when $s(e)$ or $t(e)$ lies in $\partial \Gamma$, which counts as a constant geodesic). Let $J^{e} \coloneqq \partial_{s}h^{e}_s$ be the variation field along $h^{e}_{s}.$\\ By Proposition \ref{s va} applied to each edge, \[
    \dfrac{d^2}{ds^{2}}\bar{F}(\gamma(s)) = \sum_{e\in E} \int_{a_e}^{b_e} 2\|\nabla_{t}J^{e}\|^{2} + 2\langle R(J^{e},\dot{h}_{e})J^{e},\dot{h}^{e}_{s} \rangle + \mathrm{Hess}(V_x)(J^{e},J^{e})dt.
    \] 

Each term in the sum is non-negative by Proposition \ref{potential} and Lemma \ref{convexity thm}. Suppose, for contradiction, that the total vanishes at some $s_{0}.$ Then for every edge: \begin{itemize}
    \item [(a)] $\nabla_{t}J^{e} \equiv 0, i.e., J^e$ is parallel along $h^{e}_{s_0};$ 
    \item [(b)] $\langle R(J^e,\dot{h}^{e}_{s_0})J^e,\dot{h}^{e}_{s_0} \rangle \equiv 0$ on $[a_e,b_e];$
    \item [(c)] $\mathrm{Hess}(V_{x_e})(J^e,J^e) \equiv 0$ on $[a_e,b_e].$
\end{itemize}
We use only condition $(a)$ and the boundary conditions to derive a contradiction.\\

\noindent\textbf{Propagation from the boundary.} Since $\Gamma$ is connected with non-empty boundary, choose any boundary vertex $v_0 \in \partial \Gamma$ and the unique edge $e_0$ adjacent to $v_0.$ We treat the case $s(e_0) = v_0$ (the case $t(e_0) = v_0$ is analogous). The boundary value at $s(e_0)$ is fixed, so \[
J^{e_0}(a_e) = \dot{\gamma}_{s(e_0)}(s_0) = 0.
\]
Combined with $(a),$ $J^{e_0}$ is parallel along $h^{e_0}_{s_0}$ with $J^{e_0} = 0,$ hence $J^{e_0}\equiv 0$ on $[a_e,b_e].$ In particular, \[
J^{e_0}(b_e) = \dot{\gamma}_{t(e_0)}(s_0) = 0.
\]

We have shown that $\dot{\gamma}_{t(e_0)}(s_0) = 0.$

\noindent\textbf{Inductive step.} We propagate this conclusion through the quiver. Let $W \subset V$ be the set of vertices $v$ for which we have established $\dot{\gamma}_{v}(s_0) = 0$ (where $\dot{\gamma}_{v}(s_0) = 0$ holds trivially for $v\in \partial \Gamma$). After the previous step, $\partial \Gamma \cup \{t(e_0)\} \subset W.$\\

Suppose $W \neq V.$ Since $\Gamma$ is connected, there exists an edge $e_1\in E$ with one endpoint in $W$ and the other not in $W.$ Without loss of generality, $s(e_1)\in W$ and $t(e_1) \notin W$ (the opposite case is symmetric). Then $J^{e_1}(a_e) = \dot{\gamma}_{s(e_1)}(s_0) = 0.$ By $(a),$ $J^{e_1}$ is parallel with $J^{e_1}(a_e) = 0,$ hence $J^{e_1}\equiv 0.$ In particular, $\dot{\gamma}_{t(e_1)}(s_0) = J^{e_1}(b_e) = 0,$ so $t(e_1)\in W.$ This enlarges $W.$\\

Iterating, we obtain $W = V$ in finitely many steps (since $V$ is finite). Therefore $\dot{\gamma}_{v}(s_0) = 0$ for every $v\in V,$ in particular for $v_{*},$ contradicting the assumption that $\gamma_{v_{*}}$ is non-constant.
\end{proof}

\begin{Prop} \label{proper}
    The reduced functional $\bar{F}$ is proper: $\bar{F}(p) \too \infty$ as $p\too \infty$ in $H^{|\Gamma_{\nt}|}.$
\end{Prop}

\begin{proof}
For any $p,q \in H$ and $x\in \g_{\mathbb{C}},$ the lower bound $V_x \geq 0$ gives \[
E_{x}(p,q) \geq \inf \Big\{ \int_{a}^{b} \|\dot{h}\|^2 dt : h(a) = p,  h(b) = q \Big\} = d_{H}(p,q)^2,
\]

the energy of the unique minimizing geodesic from $p$ to $q.$ Suppose $p = (p_v)_{v\in \Gamma_{\nt}}$ tends to infinity in $H^{\Gamma_{\nt}}$, i.e., $\|p_{v_{*}}\| \too \infty$ for some $v_{*} \in \Gamma_{\nt}.$ Since $\Gamma$ is connected with non-empty boundary, there exists a path in $\Gamma$ from $v_{*}$ to a boundary vertex. Along this path, the values $p_v$ go from a divergent value at $v_{*}$ to a fixed value in $\partial \Gamma,$ so at least one edge $e$ along the path satisfies $d_{H}(p_{s(e)},p_{t(e)})^2 \too \infty.$ Then \[
\bar{F}(p) \geq E_{x^{e}}(p_{s(e)},p_{t(e)}) \geq d_{H}(p_{s(e)},p_{t(e)})^{2} \too \infty.
\]
\end{proof}

\begin{Rem}[Properness of the complexified gauge action and Kempf--Ness stability]
\label{rem:complex_properness}
In contrast to the unitary gauge group $\G_{0}(\Gamma)$, whose action on $\A_{\bh}(\Gamma)$ is proper by virtue of the compactness of the target Lie group $G$ combined with an Arzel\`a--Ascoli bootstrap on the gauge differential equation (Proposition \ref{proper gauge}), the complexified gauge group $\G_{\CC}^{0}(\Gamma)$ is modeled on the non-compact complex reductive Lie group $G_{\CC}$. Consequently, properness of the $\G_{\CC}^{0}(\Gamma)$-action on the complex locus $\A_{\CC}(\Gamma) = \mu_{\CC}^{-1}(0)$ does not follow from elementary compactness arguments and is not automatic \emph{a priori}.

This is precisely where the Kempf--Ness framework on the Cartan--Hadamard manifold $H^{|\Gamma_{\mathrm{int}}|}$ (with $H = G_{\CC}/G$) becomes essential. For any solution $(\alpha, \beta) \in \A_{\CC}(\Gamma)$, the reduced functional $\overline{F} : H^{|\Gamma_{\mathrm{int}}|} \to \mathbb{R}$ is proper (Proposition \ref{proper}) and strictly geodesically convex via the boundary propagation principle (Proposition \ref{sgc} and Theorem \ref{kempf ness}). Strict convexity ensures the existence and uniqueness of a global harmonic minimizer $p^*$, establishing that each $\mathcal{G}_{\mathbb{C}}^{0}(\Gamma)$-orbit contains a unique $\mathcal{G}_{0}(\Gamma)$-orbit of solutions lying in the unitary locus $\mathcal{A}_{\mathbb{H}}(\Gamma)$.

Furthermore, the properness of $\overline{F}$ guarantees that any gauge sequence $(g_n) \subset \mathcal{G}_{\mathbb{C}}^{0}(\Gamma)$ along which configurations converge remains confined to a compact set modulo $\mathcal{G}_{0}(\Gamma)$. The action of $\mathcal{G}_{\mathbb{C}}^{0}(\Gamma)$ on the stable locus $\mathcal{A}_{\mathbb{C}}(\Gamma)$ is therefore proper, its orbits are closed, and the resulting complex quotient $\mathcal{M}_{\mathbb{C}}(\Gamma) = \mathcal{A}_{\mathbb{C}}(\Gamma)/\mathcal{G}_{\mathbb{C}}^{0}(\Gamma)$ inherits a canonical smooth complex manifold structure for which the projection map is a holomorphic submersion.
\end{Rem}

\begin{proof}[Proof of Theorem \ref{kempf ness}]
    By Proposition \ref{gc}, $\bar{F}$ is geodesically convex. By Proposition \ref{sgc}, it is strictly geodesically convex. By Proposition \ref{proper}, it is proper. A proper, strictly geodesically convex function on a Cartan--Hadamard manifold attains its infimum at a unique point: existence follows from properness, and uniqueness follows from strict convexity (two minima on a geodesic would contradict strict convexity along the geodesic).
\end{proof}

We are now ready to prove the Donaldson--quiver theorem.
\begin{proof}[Proof of Theorem \ref{complex q}]
Given $(\alpha,\beta)$ satisfying the complex Nahm equation on each edge and Kirchhoff matching for $\beta,$ we gauge $\alpha^e$ to zero on each edge, producing constants $x^e \in \g_{\mathbb{C}}$ and holonomies $a^e \in G_{\mathbb{C}}.$ The quiver Kempf--Ness functional $\bar{F}$ on $H$ admits a unique minimum $p^{*} = (p^{*}_{v})_{v\in \Gamma_{\nt}}$ by Theorem \ref{kempf ness}. For each edge $e,$ the minimizing path $h^{*}_{e}: [0,1]\too H$ from $p^{*}_{s(e)}$ to $p^{*}_{t(e)}$ (with boundary values as appropriate) lifts to a path $g^e \in C^{2}([0,1],G_{\mathbb{C}}).$ The Euler--Lagrange equation for $h^{*}_{e}$ translates precisely into the real Nahm equation for $g^{e} \cdot (0,x^e).$ The matching at interior vertices ensures that $g = (g^e)_{e\in E} \in \G_{\mathbb{C}}(\Gamma),$ and the boundary conditions give $g\in \G^{0}_{\mathbb{C}}(\Gamma)$ (up to the initial gauging by $g_0$). For uniqueness, suppose that $g,g^{'} \in \G^{0}_{\mathbb{C}}(\Gamma)$ produce solutions to the full Nahm equations. Then, $F(g) = F(g^{'})$ implies that the corresponding junction points in $H^{\Gamma_{\nt}}$ are both critical points of $\bar{F}.$ By uniqueness of the minimizer, they agree. On each edge, the minimizing path is unique, so $g$ and $g^{'}$ project to the same paths in $H^{E},$ i.e., $g^{'} = kg$ for some $k\in \G_{0}(\Gamma).$ \\

Now for the biholomorphism: the construction yields a bijection \[
\Phi: \M_{\bh}(\Gamma) \too \A_{\mathbb{C}}(\Gamma)\sll{}\G^{0}_{\mathbb{C}}(\Gamma), \quad [A] \mtoo [(A_0 + iA_1,A_2 + iA_3)]
\]
induced by the inclusion $i:\A_{\bh}(\Gamma) \hookrightarrow \A_{\mathbb{C}}(\Gamma).$ We verify that $\Phi$ is a biholomorphism. 

The inclusion $\mathcal{A}_{\mathbb{H}}(\Gamma)\hookrightarrow\mathcal{B}_{\mathbb{H}}(\Gamma)\cong\mathcal{B}_{\mathbb{C}}^{2}(\Gamma)$ 
is smooth, and the projection $\mathcal{A}_{\mathbb{C}}(\Gamma)\longrightarrow\mathcal{M}_{\mathbb{C}}(\Gamma)$ 
is a smooth submersion: the linear differential arguments and slice decompositions of Section \ref{section 3} adapt directly to $\mathfrak{g}_{\mathbb{C}}$, while properness and 
freeness of the non-compact gauge action $\G_{\CC}^{0}(\Gamma)$ on the stable locus $\A_{\CC}(\Gamma)$ are guaranteed by Remark \ref{rem:complex_properness}. Composing, $\Phi$ is smooth. Moreover, the inverse $\Phi^{-1}(\alpha,\beta)$ is given by $h^{*}g\cdot(\alpha,\beta)\in \A_{\bh}(\Gamma),$ where $g\cdot(\alpha,\beta)_{|e} = (0,x^{e}), g_{e}(a_e) = 1$ and $h^{*}$ is the unique minimizer of $\bar{F}.$ The map $(\alpha,\beta) \too (x^{e})_{e\in E}$ is smooth by the ODE; the boundary data of $\bar{F}$ depend smoothly on $(\alpha,\beta)$; and the minimizer $p^{*} = p^{*}(\alpha,\beta)$ depends smoothly on these data. To see this, consider the map \[
\Psi: \A_{\mathbb{C}}(\Gamma) \times H^{\Gamma_{\nt}} \too T^{*}H^{\Gamma_{\nt}}, \quad \Psi((c,d),p) \mtoo d_{(c,d)}\bar{F}(p).
\] 
This map is smooth: $(x^{e})$ depends smoothly on $(c,d)$ by smooth dependence of ODE solutions on parameters, and $E_{x}(p,q)$ is jointly smooth in $(x,p,q)$ (Proposition \ref{jgc}, applied with $x$ as an additional parameter). At the critical point $p^{*} = p^{*}(\alpha,\beta),$ we have $d_{(\alpha,\beta)}\bar{F}(p^{*}) = 0,$ and the partial derivative $\partial_{p^*} \Psi$ equals $\mathrm{Hess}\bar{F}(p^*),$ which is positive definite as a bilinear form by strict geodesic convexity (Proposition \ref{sgc}), and hence invertible because $H^{\Gamma_{\nt}}$ is finite-dimensional. The implicit function theorem yields a smooth map $(c,d) \too p^{*}(c,d)$ locally near $(\alpha,\beta).$ Finally, the per-edge minimizers depend smoothly on $p^*$ by the last point of Lemma \ref{lem uni}. Composing, $\Phi^{-1}$ is smooth.\\

By Corollary \ref{kah mani}, $\A_{\mathbb{C}}(\Gamma) = \mu_{\mathbb{C}}^{-1}(0) = N^1$ is an $I_1$-complex Kähler Banach submanifold of $\B_{\bh}(\Gamma),$ and $\G^{0}_{\mathbb{C}}(\Gamma)$ acts holomorphically. Hence $\M_{\mathbb{C}}(\Gamma)\coloneqq \mu_{\mathbb{C}}^{-1}(0)/\G^{0}_{\mathbb{C}}(\Gamma)$ inherits a complex structure $\tilde{I}_1,$ characterised by the requirement that the projection $N^{1} \too \M_{\mathbb{C}}(\Gamma)$ be holomorphic.\\ On the other hand, the descended complex structure $I_{1,\text{red}}$ on $\M_{\bh}(\Gamma)$ defined by the quaternionic slice criterion (Proposition \ref{q slice}) and Lemma \ref{lemma 4.10} has the horizontal model \[
T_{[A]}\M_{\bh}(\Gamma) \cong T_{A}\A_{\bh}(\Gamma) \cap \Ker D_{A}^{*},
\]
on which $I_1$ acts (Remark \ref{rmq 4.12}). Since the inclusion $T_{A}\A_{\bh}(\Gamma) \hookrightarrow T_{A}N^{1} = \Ker d_{A}\mu_{\mathbb{C}}$ is $I_1$-linear (both are $I_1$-stable subspaces of $T_{A}\B_{\bh}(\Gamma)$), the differential \[
d_{[A]}\Phi: T_{[A]}\M_{\bh}(\Gamma) \too T_{[(\alpha,\beta)]}\M_{\mathbb{C}}(\Gamma)
\]
intertwines $I_{1,\text{red}}$ with $\tilde{I}_{1}.$ As $d_{[A]}\Phi$ is a bounded, $I_1$-linear isomorphism, $\Phi$ is a biholomorphism.
\end{proof}

\subsection{Edge-Length Independence of the Complex Moduli Space}
 
Let $\Gamma = (E, V, s, t)$ be an oriented connected quiver with boundary. To each edge $e \in E$, associate a compact interval $I_e = [a_e, b_e]$ with length $l_e := b_e - a_e > 0$, and denote by $l = (l_e)_{e \in E} \in \mathbb{R}_{>0}^E$ the length vector. We denote the corresponding affine configuration space and moduli spaces by $\mathcal{B}_{\mathbb{C}}(\Gamma_{l})$, $\mathcal{A}_{\mathbb{C}}(\Gamma_{l})$, and $\mathcal{M}_{\mathbb{C}}(\Gamma_{l}).$

\begin{Prop}[Holomorphic Symplectic Length Independence]
\label{prop:length-invariance-complex}
For any two length assignments $l,l_1 \in \mathbb{R}_{>0}^E$, there exists a canonical biholomorphism and holomorphic symplectomorphism
\[
\Phi_{l,l_1}^* : \mathcal{M}_{\mathbb{C}}(\Gamma_{l}) \xrightarrow{\ \sim\ } \M_{\mathbb{C}}(\Gamma_{l_1})
\]
preserving the boundary evaluation moment maps. Consequently, $\M_{\mathbb{C}}(\Gamma)$ is independent of the edge lengths.
\end{Prop}

\begin{proof}
Without loss of generality, let $l_1 = \mathbf{1} = (1)_{e\in E}$ be the unit length vector with standard intervals $J_e = [0, 1]$. For each $e \in E$, consider the affine orientation-preserving diffeomorphism
\[
\phi_e : J_e \longrightarrow I_e, \qquad s \longmapsto a_e + s l_e.
\]
We identify the complex connection $\alpha^e$ with a Lie-algebra-valued $1$-form $\alpha_{\mathbb{C}}^e = \alpha^e(t)\,dt$ on $I_e$, and the Higgs field $\beta^e$ with a Lie-algebra-valued $0$-form (a section of $I_e \times \mathfrak{g}_{\mathbb{C}}$). We define the pullback map on fields $\Phi^* : \mathcal{B}_{\mathbb{C}}(\Gamma_{l}) \to \mathcal{B}_{\mathbb{C}}(\Gamma_{\mathbf{1}})$ edge-wise by
\[
\tilde{\alpha}^e(s)\,ds := \phi_e^*(\alpha^e(t)\,dt) = l_e\,\alpha^e(a_e + s l_e)\,ds, \qquad \tilde{\beta}^e(s) := \phi_e^*(\beta^e)(s) = \beta^e(a_e + s l_e).
\]
We verify the invariance of each component of the reduction:

\begin{enumerate}
    \item \textbf{Complex Nahm Equation:} The equation $\dot{\beta}^e + [\alpha^e, \beta^e] = 0$ is equivalent to the vanishing of the covariant exterior derivative $d_{\alpha_{\mathbb{C}}^e}\beta^e := d\beta^e + [\alpha_{\mathbb{C}}^e, \beta^e] = 0$ on $I_e$. Since the exterior derivative and the Lie bracket commute with smooth pullbacks,
    \[
    \phi_e^*(d\beta^e + [\alpha_{\mathbb{C}}^e, \beta^e]) = d\tilde{\beta}^e + [\tilde{\alpha}_{\mathbb{C}}^e, \tilde{\beta}^e] = \left(\frac{d\tilde{\beta}^e}{ds}(s) + [\tilde{\alpha}^e(s), \tilde{\beta}^e(s)]\right)ds = 0.
    \]
    Thus, $\Phi^*$ preserves the complex Nahm equation edge-wise.

    \item \textbf{Kirchhoff Matching:} The complex Kirchhoff condition at any $v \in \Gamma_{\mathrm{int}}$ depends only on endpoint values:
    \[
    \sum_{e \in s^{-1}(v)} \beta^e(a_e) - \sum_{e \in t^{-1}(v)} \beta^e(b_e) = \sum_{e \in s^{-1}(v)} \tilde{\beta}^e(0) - \sum_{e \in t^{-1}(v)} \tilde{\beta}^e(1) = 0.
    \]
    Hence, $\Phi^*(\mathcal{A}_{\mathbb{C}}(\Gamma_{l})) = \mathcal{A}_{\mathbb{C}}(\Gamma_{\mathbf{1}}).$

    \item \textbf{Gauge Equivariance:} The complex gauge group $\mathcal{G}_{\mathbb{C}}^0(\Gamma_{l})$ pulls back via $\tilde{g}^e(s) = g^e(\phi_e(s))$ to $\mathcal{G}_{\mathbb{C}}^0(\Gamma_{\mathbf{1}})$. Since $\phi_e(0) = a_e$ and $\phi_e(1) = b_e$, vertex matchings and boundary triviality conditions are preserved. The gauge action transforms as
    \[
    \phi_e^*(g \cdot \alpha_{\mathbb{C}}^e) = \phi_e^*(-\dot{g}g^{-1}dt + g\alpha^e g^{-1}dt) = -\left(\frac{d\tilde{g}}{ds}\tilde{g}^{-1}\right)ds + \tilde{g}\tilde{\alpha}^e\tilde{g}^{-1}ds = \tilde{g} \cdot \tilde{\alpha}_{\mathbb{C}}^e.
    \]
    Thus, $\Phi^*$ is $\mathcal{G}_{\mathbb{C}}^0$-equivariant.

    \item \textbf{Holomorphic Symplectic Form:} The holomorphic symplectic form $\Omega_{\mathbb{C}}^l = \omega_2 + i\omega_3$ on $\mathcal{B}_{\mathbb{C}}(\Gamma_{l})$ is given on variations $(\delta\alpha, \delta\beta), (\delta'\alpha, \delta'\beta)$ by
    \[
    \Omega_{\mathbb{C}}^{l}(\delta A, \delta' A) = \sum_{e\in E} \int_{a_e}^{b_e} \Big( \langle \delta\alpha^e(t), \delta'\beta^e(t) \rangle - \langle \delta'\alpha^e(t), \delta\beta^e(t) \rangle \Big)\,dt.
    \]
    Changing variables $t = a_e + s l_e$ with $dt = l_e\,ds$, and substituting $\delta\tilde{\alpha}^e(s) = l_e\,\delta\alpha^e(\phi_e(s))$ and $\delta\tilde{\beta}^e(s) = \delta\beta^e(\phi_e(s))$, the Jacobian factor $l_e$ cancels identically:
    \[
    \int_{a_e}^{b_e} \langle \delta\alpha^e(t), \delta'\beta^e(t) \rangle\,dt = \int_0^1 \left\langle \frac{1}{l_e}\delta\tilde{\alpha}^e(s), \delta'\tilde{\beta}^e(s) \right\rangle l_e\,ds = \int_0^1 \langle \delta\tilde{\alpha}^e(s), \delta'\tilde{\beta}^e(s) \rangle\,ds.
    \]
    Thus, $\Phi^* \Omega_{\mathbb{C}}^{l} = \Omega_{\mathbb{C}}^{\mathbf{1}}$.
\end{enumerate}
Since $\Phi^*$ is an equivariant, symplectic isomorphism of the ambient Banach manifolds preserving the moment map zero loci, it descends to a canonical holomorphic symplectomorphism $\mathcal{M}_{\mathbb{C}}(\Gamma_{l}) \cong \mathcal{M}_{\mathbb{C}}(\Gamma_{\mathbf{1}}).$
\end{proof}

\begin{Rem}[Finite-Dimensional Verification via Holonomy Reduction]
The result becomes equally transparent once $\mathcal{M}_{\mathbb{C}}(\Gamma)$ is identified with a finite-dimensional complex symplectic quotient (Theorem \ref{thm:finite-dim-reduction} below). On any edge $e$, fixing the gauge $\alpha^e \equiv 0$ implies that $\beta^e(t)$ is constant. The data on $e$ modulo based gauge equivalence are parameterized by the holonomy $g_e := g^{e}_{\alpha}(b_e)$ of Theorem \ref{thm:finite-dim-reduction} and $\xi_e = \beta^e(a_e) \in \mathfrak{g}_{\mathbb{C}}$, which identifies the unconstrained edge data with $T^* G_{\mathbb{C}}$. Since $\beta^{e}(b_e) = \Ad_{g_e}\xi_e$, the complex Kirchhoff conditions become
\[
\mu_v((g_e, \xi_e)_{e\in E}) = \sum_{e \in s^{-1}(v)} \xi_e - \sum_{e \in t^{-1}(v)} \mathrm{Ad}_{g_e}(\xi_e) = 0 \quad (\forall v \in \Gamma_{\mathrm{int}}).
\]
Because this algebraic condition and the canonical symplectic form on $(T^* G_{\mathbb{C}})^E$ do not depend on $l_e$, the reduction $(T^* G_{\mathbb{C}})^E /\!\!/\!\!/_{\mathbb{C}} G_{\mathbb{C}}^{\Gamma_{\mathrm{int}}} \cong \mathcal{M}_{\mathbb{C}}(\Gamma)$ is manifestly independent of the edge lengths.
\end{Rem}

\begin{Thm}[Identification with $(T^*G_{\mathbb{C}})^E /\!\!/\!\!/_{\mathbb{C}} G_{\mathbb{C}}^{\Gamma_{\mathrm{int}}}$]
\label{thm:finite-dim-reduction}
Let $\Gamma = (E, V, s, t)$ be a connected quiver with non-empty boundary. The map
\[
\varphi : \mathcal{M}_{\mathbb{C}}(\Gamma) \longrightarrow (T^*G_{\mathbb{C}})^E /\!\!/\!\!/_{\mathbb{C}} G_{\mathbb{C}}^{\Gamma_{\mathrm{int}}}, \qquad [(\alpha, \beta)] \longmapsto [(g_{\alpha}^e(b_e), \beta^e(a_e))_{e \in E}],
\]
is an isomorphism of complex Hamiltonian $G_{\mathbb{C}}^{\partial\Gamma}$-spaces, where $g_{\alpha}^e \in C^2(I_e, G_{\mathbb{C}})$ is the unique solution to the parallel transport equation $\dot{g}_{\alpha}^e + \alpha^e g_{\alpha}^e = 0$ with $g_{\alpha}^e(a_e) = \mathbf{1}$.
\end{Thm}

\begin{proof}
We follow the construction in \cite[Section 5]{maiza2026lax}, adapted to the complex-holomorphic setting.\\

\textbf{Step 1: Well-definedness and Bijectivity.}
For any solution $(\alpha, \beta) \in \mathcal{A}_{\mathbb{C}}(\Gamma)$, the transformation $g_{\alpha}^{-1} \cdot (\alpha, \beta)$ gauges the connection to zero: $\alpha' = 0$. The complex Nahm equation then reduces to $\dot{\beta}' = 0$, so $\beta'^e(t) \equiv x_e := \beta^e(a_e) \in \mathfrak{g}_{\mathbb{C}}$ is constant along each edge. Transforming back gives $\beta^e(b_e) = \mathrm{Ad}_{g_{\alpha}^e(b_e)} x_e$. The complex Kirchhoff condition at $v \in \Gamma_{\mathrm{int}}$ becomes
\[
\sum_{e \in s^{-1}(v)} x_e - \sum_{e \in t^{-1}(v)} \mathrm{Ad}_{g_{\alpha}^e(b_e)} x_e = 0,
\]
which is precisely the zero locus of the holomorphic moment map $\nu : (T^*G_{\mathbb{C}})^E \to (\mathfrak{g}_{\mathbb{C}}^*)^{\Gamma_{\mathrm{int}}}$ for the diagonal action of $G_{\mathbb{C}}^{\Gamma_{\mathrm{int}}}$. If $(\alpha, \beta)$ is transformed by $h \in \mathcal{G}_{\mathbb{C}}^0(\Gamma)$, the holonomies transform by $h(t(e)) g_{\alpha}^e(b_e) h(s(e))^{-1}$ and $\beta^e(a_e)$ by $\mathrm{Ad}_{h(s(e))}$, which corresponds to the $G_{\mathbb{C}}^{\Gamma_{\mathrm{int}}}$-gauge orbit. Thus, $\varphi$ is well-defined.\\

To construct the smooth inverse, let $(a, x) \in \nu^{-1}(0)$. For each edge $e$, choose a path $\gamma_a^e \in C^2(I_e, G_{\mathbb{C}})$ with $\gamma_a^e(a_e) = \mathbf{1}$ and $\gamma_a^e(b_e) = a_e$. Define
\[
\psi : (T^*G_{\mathbb{C}})^E /\!\!/\!\!/_{\mathbb{C}} G_{\mathbb{C}}^{\Gamma_{\mathrm{int}}} \longrightarrow \mathcal{M}_{\mathbb{C}}(\Gamma), \qquad [(a, x)] \longmapsto [(\gamma_a \cdot 0, \gamma_a \cdot x)].
\]
Any other choice $\tilde{\gamma}_a^e$ satisfies $\tilde{\gamma}_a^e(\gamma_a^e)^{-1} \in \mathcal{G}_{\mathbb{C}}^0(\Gamma)$, so the class in $\mathcal{M}_{\mathbb{C}}(\Gamma)$ is independent of the choice of path. The moment condition $\nu(a, x) = 0$ guarantees that $\psi(a, x)$ satisfies the Kirchhoff matching at interior vertices. Smooth local sections of $\gamma \mapsto \gamma(b_e)$ guarantee that $\psi$ is smooth, which establishes that $\varphi$ is a biholomorphism.\\

\textbf{Step 2: Holomorphic Symplectic Identification.}
Using the standard coordinates on $T^*G_{\mathbb{C}}$, choose $\gamma_{a e^{s u}}(t) = e^{t y} e^{t s u}$ for $a = e^y$. The differential of $\psi$ is
\[
d\psi_{(a, x)}(u, v) = \left( \mathrm{Ad}_{\gamma_a} u, \, \mathrm{Ad}_{\gamma_a}(s [u, x] + v) \right).
\]
Evaluating the ambient holomorphic symplectic form $\Omega_{\mathbb{C}}$ on two tangent vectors yields
\begin{align*}
(\psi^* \Omega_{\mathbb{C}})_{(a, x)}((u_1, v_1), (u_2, v_2)) &= \int_I \Big( \langle u_1, s [u_2, x] + v_2 \rangle_{\mathbb{C}} - \langle u_2, s [u_1, x] + v_1 \rangle_{\mathbb{C}} \Big) ds \\
&= \langle u_1, v_2 \rangle_{\mathbb{C}} - \langle u_2, v_1 \rangle_{\mathbb{C}} + \langle x, [u_1, u_2] \rangle_{\mathbb{C}},
\end{align*}
which matches the canonical holomorphic symplectic form on $(T^*G_{\mathbb{C}})^E$. Finally, the unreduced boundary values give $\beta(v) = \mathrm{Ad}_{a_e} x_e$ for $v \in \partial\Gamma^+$ and $\beta(v) = x_e$ for $v \in \partial\Gamma^-$, intertwining the residual $G_{\mathbb{C}}^{\partial\Gamma}$ moment maps.
\end{proof}

\paragraph{\textbf{The Gluing Law via Reduction by Stages}}

Let $\Gamma_1$ and $\Gamma_2$ be two quivers with non-empty boundaries such that $ \partial\Gamma_1^+ = \partial\Gamma_2^- = D.$ Let $\Gamma = \Gamma_1 *_D \Gamma_2$ denote the quiver obtained by identifying the vertices in $D$ (see also \cite{maiza2026lax}, \cite[Subsection 6.3]{maiza2026multiplicative}).

\begin{Prop}[Symplectic Cutting and Gluing]
\label{thm:gluing-law}
There exists a canonical biholomorphism and holomorphic symplectomorphism
\[
\M_{{\CC}}(\Gamma_1 *_D \Gamma_2) \cong \left( \M_{\CC}(\Gamma_1) \times \M_{\CC}(\Gamma_2) \right) /\!\!/\!\!/_{\CC}\, G_{\CC}^D.
\]
\end{Prop}

\begin{proof}
This follows immediately from Theorem \ref{thm:finite-dim-reduction} and reduction by stages.
\end{proof}

\subsection{Dependence of the Hyperkähler Metric on the Edge Lengths} \label{subsec:metric-length}

While the complex symplectic manifold $(\mathcal{M}_{\mathbb{C}}(\Gamma), \Omega_{\mathbb{C}})$ is topological with respect to edge metrics, the hyperkähler metric $g_{\mathbb{H}}$ on $\mathcal{M}_{\mathbb{H}}(\Gamma)$ depends on the edge lengths. To make this statement precise, we transport all the metrics to a fixed complex manifold.

Throughout this subsection, $\Gamma$ is connected with non-empty boundary, and we set
\[
X(\Gamma) \coloneqq (T^*G_{\mathbb{C}})^E /\!\!/\!\!/_{\mathbb{C}} G_{\mathbb{C}}^{\Gamma_{\mathrm{int}}}.
\]
For a length vector $\mathbf{l} \in \R^{E}_{>0}$, let
\[
\Psi_{\mathbf{l}} \coloneqq \varphi \circ \Phi : \M_{\bh}(\Gamma_{\mathbf{l}}) \too \M_{\CC}(\Gamma_{\mathbf{l}}) \too X(\Gamma),
\]
where $\Phi$ is the biholomorphism of Theorem \ref{complex q} and $\varphi$ is the isomorphism of Theorem \ref{thm:finite-dim-reduction}. Since the parallel transport of a connection along an edge does not depend on the parametrization of the edge, the target $X(\Gamma)$ does not depend on $\mathbf{l}$, and $\Psi_{\mathbf{l}}$ is a biholomorphism from $(\M_{\bh}(\Gamma_{\mathbf{l}}),I_1)$ onto $X(\Gamma)$. We write
\[
g^{X}_{\mathbf{l}} \coloneqq (\Psi_{\mathbf{l}}^{-1})^{*} g_{\bh,\mathbf{l}}
\]
for the hyperkähler metric transported to $X(\Gamma)$. For $c > 0$, the \emph{fibrewise dilation} $\delta_c[(g_e,\xi_e)_{e\in E}] \coloneqq [(g_e, c\,\xi_e)_{e\in E}]$ is a well-defined biholomorphism of $X(\Gamma)$ with $\delta_c^{*}\Omega_{\mathrm{can}} = c\,\Omega_{\mathrm{can}}$, since the moment map $\nu$ of Theorem \ref{thm:finite-dim-reduction} is linear in $\xi$, the $G_{\CC}^{\Gamma_{\nt}}$-action commutes with $\delta_c$, and the canonical form on $T^*G_{\CC} \cong G_{\CC}\times \g_{\CC}$ is linear in $\xi$. Finally, let
\[
K(\Gamma) \coloneqq \Big\{ \zeta = (\zeta_e)_{e\in E} \in \g_{\CC}^{E} : \sum_{e\in s^{-1}(v)} \zeta_e - \sum_{e\in t^{-1}(v)} \zeta_e = 0 \text{ for all } v\in \Gamma_{\nt} \Big\}.
\]
Applying Lemma \ref{useful lemma} with $B = 0$ and $W = \g_{\CC}$ regarded as a real vector space, the kernel of $R$ consists of the edgewise constant maps satisfying the Kirchhoff law, that is, of $K(\Gamma)$; hence $\dim_{\R} K(\Gamma) = 2(|E| - |\Gamma_{\nt}|)\dim G$.

\begin{Prop}[Metric Length Dependence]
\label{prop:metric-length-dependence}
Let $\Gamma$ be a connected quiver with non-empty boundary, and assume $\dim G > 0$ and $|E| > |\Gamma_{\nt}|$, so that $\dim \M_{\bh}(\Gamma) > 0$.
\begin{enumerate}
    \item \textbf{(Homothety under uniform rescaling.)} For every $\lambda > 0$ there is a $G^{\partial\Gamma}$-equivariant diffeomorphism $r_\lambda: \M_{\bh}(\Gamma_{\mathbf{l}}) \to \M_{\bh}(\Gamma_{\lambda\mathbf{l}})$ such that
    \[
    r_\lambda^{*} g_{\bh,\lambda\mathbf{l}} = \lambda^{-1} g_{\bh,\mathbf{l}}, \qquad r_\lambda^{*}\omega_{a,\lambda\mathbf{l}} = \lambda^{-1}\omega_{a,\mathbf{l}}, \qquad (r_{\lambda})_{*} I_a = I_a \quad (a=1,2,3),
    \]
    and $\Psi_{\lambda\mathbf{l}} \circ r_\lambda = \delta_{\lambda^{-1}} \circ \Psi_{\mathbf{l}}$. Consequently, $\delta_{\lambda^{-1}}^{*} g^{X}_{\lambda\mathbf{l}} = \lambda^{-1} g^{X}_{\mathbf{l}}$.
    \item \textbf{(Non-triviality.)} Let $x_0 \in X(\Gamma)$ be the class of $(\mathbf{1}_e, 0)_{e\in E}$. For every $\mathbf{l}$, $\Psi_{\mathbf{l}}([0]) = x_0$, and for every $\zeta \in K(\Gamma)$ the class $\hat{\zeta} \in T_{x_0}X(\Gamma)$ of the tangent vector $(0,\zeta_e)_{e\in E}$ satisfies
    \[
    g^{X}_{\mathbf{l}}(\hat{\zeta},\hat{\zeta}) = \sum_{e\in E} l_e\,\|\zeta_e\|^{2}.
    \]
    Consequently, $g^{X}_{\mathbf{l}} \neq g^{X}_{\mathbf{l}'}$ whenever $\sum_{e\in E}(l_e - l'_e)\|\zeta_e\|^2 \neq 0$ for some $\zeta\in K(\Gamma)$; in particular, $g^{X}_{\lambda\mathbf{l}} \neq g^{X}_{\mathbf{l}}$ for every $\lambda \neq 1$.
\end{enumerate}
\end{Prop}

\begin{proof}
\textbf{(1)} For $\lambda > 0$, the edge $I_e = [a_e,b_e]$ of $\Gamma_{\mathbf{l}}$ corresponds to the edge $\lambda I_e = [\lambda a_e,\lambda b_e]$ of $\Gamma_{\lambda\mathbf{l}}$. Define
\[
R_\lambda : \B_{\bh}(\Gamma_{\mathbf{l}}) \too \B_{\bh}(\Gamma_{\lambda\mathbf{l}}), \qquad (R_\lambda A)^{e}_{\mu}(t) \coloneqq \lambda^{-1} A^{e}_{\mu}(t/\lambda), \quad \mu = 0,1,2,3,\ t\in \lambda I_e,
\]
and, on gauge groups, $(R_\lambda g)^{e}(t) \coloneqq g^{e}(t/\lambda)$. We check the required properties.
\begin{itemize}
    \item \emph{Nahm--Kirchhoff space.} For $i\in\{1,2,3\}$,
    \[
    \tfrac{d}{dt}(R_\lambda A)_{i} + [(R_\lambda A)_0,(R_\lambda A)_i] + [(R_\lambda A)_{i+1},(R_\lambda A)_{i+2}] = \lambda^{-2}\bigl(\dot{A}_i + [A_0,A_i] + [A_{i+1},A_{i+2}]\bigr)(t/\lambda),
    \]
    so $R_\lambda$ preserves the Nahm equations. The Kirchhoff conditions are linear and homogeneous in the endpoint values, which $R_\lambda$ multiplies by $\lambda^{-1}$; hence $R_\lambda(\A_{\bh}(\Gamma_{\mathbf{l}})) = \A_{\bh}(\Gamma_{\lambda\mathbf{l}})$.
    \item \emph{Gauge equivariance.} $R_\lambda g$ has the same vertex values as $g$, so $R_\lambda$ restricts to isomorphisms $\G(\Gamma_{\mathbf{l}}) \cong \G(\Gamma_{\lambda\mathbf{l}})$ and $\G_0(\Gamma_{\mathbf{l}}) \cong \G_0(\Gamma_{\lambda\mathbf{l}})$. Moreover,
    \[
    (R_\lambda g)\cdot(R_\lambda A)_0 = \lambda^{-1}\bigl(gA_0g^{-1}\bigr)(t/\lambda) - \lambda^{-1}\dot{g}(t/\lambda)g(t/\lambda)^{-1} = \bigl(R_\lambda(g\cdot A)\bigr)_0,
    \]
    and the Higgs components are clearly equivariant. Hence $R_\lambda(g\cdot A) = (R_\lambda g)\cdot (R_\lambda A)$.
    \item \emph{Hyperkähler structure.} By the change of variables $t = \lambda s$,
    \[
    \langle R_\lambda X, R_\lambda Y\rangle = \sum_{e\in E}\int_{\lambda I_e}\sum_{\mu=0}^{3} \lambda^{-2}\langle X_\mu(t/\lambda),Y_\mu(t/\lambda)\rangle\,dt = \lambda^{-1}\langle X,Y\rangle.
    \]
    Since $R_\lambda$ acts on all four components by the same operation, it commutes with $I_1,I_2,I_3$, and hence $R_\lambda^{*}\omega_a = \lambda^{-1}\omega_a$.
\end{itemize}
Consequently, $R_\lambda$ maps $\G_0$-orbits onto $\G_0$-orbits, and since it multiplies the $L^2$ metric by a constant, it maps the metric orthogonal complement $T_A\A_{\bh}(\Gamma_{\mathbf{l}})\cap \Ker D^{*}_{A}$ of the orbit onto $T_{R_\lambda A}\A_{\bh}(\Gamma_{\lambda\mathbf{l}})\cap \Ker D^{*}_{R_\lambda A}$. It therefore descends to a diffeomorphism $r_\lambda$ (with inverse induced by $R_{\lambda^{-1}}$) satisfying the three displayed identities. Equivariance with respect to $G^{\partial\Gamma}$ follows from the fact that $R_\lambda$ preserves boundary values of gauge transformations.

For the finite-dimensional model, write $\alpha = A_0 + iA_1$ and $\beta = A_2 + iA_3$, so that $R_\lambda\alpha = \lambda^{-1}\alpha(\cdot/\lambda)$ and $R_\lambda\beta = \lambda^{-1}\beta(\cdot/\lambda)$. The path $t\mapsto g^{e}_{\alpha}(t/\lambda)$ solves $\dot{g} + \lambda^{-1}\alpha^e(t/\lambda)g = 0$ with value $\mathbf{1}$ at $\lambda a_e$, so the holonomies are unchanged, while $\beta^{e}(\lambda a_e)$ is replaced by $\lambda^{-1}\beta^e(a_e)$. Since $\Phi$ is induced by the inclusion $\A_{\bh}\hookrightarrow \A_{\CC}$, this gives $\Psi_{\lambda\mathbf{l}}\circ r_\lambda = \delta_{\lambda^{-1}}\circ\Psi_{\mathbf{l}}$. Hence $\Psi_{\lambda\mathbf{l}}^{-1} = r_\lambda\circ \Psi_{\mathbf{l}}^{-1}\circ\delta_\lambda$, and
\[
g^{X}_{\lambda\mathbf{l}} = \delta_\lambda^{*}(\Psi_{\mathbf{l}}^{-1})^{*} r_\lambda^{*} g_{\bh,\lambda\mathbf{l}} = \lambda^{-1}\delta_\lambda^{*} g^{X}_{\mathbf{l}}.
\]
Applying $\delta_{\lambda^{-1}}^{*}$ to both sides gives $\delta_{\lambda^{-1}}^{*} g^{X}_{\lambda\mathbf{l}} = \lambda^{-1} g^{X}_{\mathbf{l}}$.

\textbf{(2)} The zero configuration $A = 0$ lies in $\A_{\bh}(\Gamma_{\mathbf{l}})$; its holonomies are trivial and its Higgs fields vanish, so $\Psi_{\mathbf{l}}([0]) = x_0$. Given $\zeta\in K(\Gamma)$, define $Y\in \B_{\bh}(\Gamma_{\mathbf{l}})$ by
\[
Y^{e}_0 = Y^{e}_1 = 0, \qquad Y^{e}_2 + iY^{e}_3 \equiv \zeta_e \quad \text{(constant on } I_e).
\]
At $A = 0$ the linearized Nahm equations reduce to $\dot{Y}_i = 0$, and the real and imaginary parts of $\zeta$ satisfy the Kirchhoff law, so $Y\in T_0\A_{\bh}(\Gamma_{\mathbf{l}})$. Moreover $D^{*}_0Y = (\dot{Y}_0, (\sum_{s^{-1}(v)}Y^{e}_0(a_e) - \sum_{t^{-1}(v)}Y^{e}_0(b_e))_v) = 0$, so $Y$ is horizontal. By definition of the reduced metric,
\[
g_{\bh,\mathbf{l}}\bigl(d\pi(Y),d\pi(Y)\bigr) = \|Y\|^2_{L^2} = \sum_{e\in E}\int_{I_e}\bigl(\|Y^{e}_2\|^2 + \|Y^{e}_3\|^2\bigr)dt = \sum_{e\in E} l_e\|\zeta_e\|^2,
\]
where $\|\zeta_e\|^2 = \|Y^{e}_2\|^2 + \|Y^e_3\|^2$ is the Hermitian norm on $\g_{\CC}$. The variation $Y$ leaves $\alpha$, hence the holonomies, unchanged and varies $\beta^{e}(a_e)$ by $\zeta_e$; thus $d\Psi_{\mathbf{l}}(d\pi(Y)) = \hat{\zeta}$, a tangent vector at $x_0$ which does not depend on $\mathbf{l}$. This proves the formula for $g^{X}_{\mathbf{l}}(\hat{\zeta},\hat{\zeta})$, and the first consequence follows by subtraction. For $\mathbf{l}' = \lambda\mathbf{l}$ with $\lambda\neq 1$, choose any $\zeta\in K(\Gamma)\setminus\{0\}$, which exists since $\dim_{\R}K(\Gamma) = 2(|E|-|\Gamma_{\nt}|)\dim G > 0$; then $\sum_e (l_e - \lambda l_e)\|\zeta_e\|^2 = (1-\lambda)\sum_e l_e\|\zeta_e\|^2 \neq 0$.
\end{proof}

\begin{Rem} \label{rem:homothety}
Part (1) shows that the inhomogeneous transformation laws obtained when only $A_0$ is treated as a $1$-form and $A_1,A_2,A_3$ as $0$-forms are not, in themselves, an obstruction: treating all four components as $1$-forms, Nahm's equations are invariant under rescaling, and a uniform rescaling of the edge lengths acts on the hyperkähler structure by a homothety, composed with the fibrewise dilation $\delta_{\lambda^{-1}}$. For a single edge, this recovers the fact that the Kronheimer metrics on $T^*G_{\CC}$ obtained from intervals of different lengths are homothetic. Part (2) shows that, beyond this homothety, the transported metrics $g^{X}_{\mathbf{l}}$ are genuinely different; Theorem \ref{thm:hklr-metric-gluing} below shows that subdividing an edge while preserving its total length does not change the metric.
\end{Rem}

\begin{Rem}[Rescaled per-edge functional] \label{rem:rescaled}
For an edge of length $l$ carrying $(\alpha,\beta) = (0,x)$, we write $E_x(p,q;l)$ for the infimum of $\int_0^l \bigl(\|\dot{h}\|^2_H + V_x(h)\bigr)dt$ over paths $h$ from $p$ to $q$. Setting $\tilde{h}(s) \coloneqq h(sl)$, $s\in[0,1]$, so that $\tilde{h}' = l\dot{h}$, one gets
\[
\int_0^l \bigl(\|\dot{h}\|^2_H + V_x(h)\bigr)dt = \int_0^1 \Bigl(\frac{1}{l}\|\tilde{h}'(s)\|^2_H + l\,V_x(\tilde{h}(s))\Bigr)ds,
\]
and the Euler--Lagrange equation \eqref{euler eq} becomes $\nabla_s\tilde{h}' = \frac{l^2}{2}\grad_H V_x(\tilde{h})$. By the Cauchy--Schwarz inequality, $E_x(p,q;l) \geq d_H(p,q)^2/l$.
\end{Rem}

\section{The 2D Hamiltonian TQFT via Complex Nahm--Kirchhoff Moduli Spaces}
\label{sec:2d-tqft-complex}

In this section, we provide the complete construction of a two-dimensional topological quantum field theory (2D TQFT) taking values in the category of complex Hamiltonian manifolds, using the moduli spaces $\M_{\CC}(\Gamma).$

\subsection{The Target Category $\mathbf{Ham}_{\mathbb{C}}$}

Following the framework of Moore--Tachikawa \cite{MooreTachikawa} and its categorical completion via Wehrheim--Woodward correspondences \cite{WehrheimWoodward2010Functoriality}, we first introduce a partial category of complex Hamiltonian manifolds and complete it to a symmetric monoidal category $\mathbf{Ham}_{\CC}$.

\begin{Def}[The Partial Category of Complex Hamiltonian Spaces]
The partial category $\mathbf{Ham}_{\mathbb{C}}^{\mathrm{pre}}$ is defined as follows:
\begin{itemize}
    \item \textbf{Objects:} Connected complex reductive Lie groups, with the trivial group $\{*\} = G^0$ serving as the monoidal unit.
    
    \item \textbf{Morphisms:} A morphism $M : G \to H$ is an isomorphism class of smooth holomorphic symplectic manifolds $(M, \Omega)$ equipped with a holomorphic Hamiltonian action of $G \times H$ and an equivariant holomorphic moment map
    \[
    \mu = (\mu_{\mathrm{in}}, \mu_{\mathrm{out}}) : M \longrightarrow \mathfrak{g}^* \times \mathfrak{h}^*.
    \]
    An isomorphism between morphisms $(M, \Omega, \mu)$ and $(M', \Omega', \mu')$ is an equivariant biholomorphism $f: M \to M'$ such that $f^*\Omega' = \Omega$ and $\mu' \circ f = \mu.$
    
    \item \textbf{Composition:} Two morphisms $M : G \to H$ and $N : H \to K$ are \textit{composable} if the diagonal action of $H$ on $\mu_{\Delta}^{-1}(0) \subset M \times N$ is free and proper, where
    \[
    \mu_{\Delta}(x, y) := \mu_{\mathrm{out}}(x) - \mu_{\mathrm{in}}(y) : M \times N \longrightarrow \mathfrak{h}^*.
    \]
    
    In this case, the composite morphism $N \circ M: G \to K$ is defined by the holomorphic symplectic quotient (Marsden--Weinstein reduction):
    \[
    N \circ M := (M \times N) /\!\!/\!\!/_{\mathbb{C}}\, H = \mu_{\Delta}^{-1}(0) / H_{\mathbb{C}}.
    \]
    The residual $G \times K$-action on $\mu_{\Delta}^{-1}(0)$ descends to a holomorphic Hamiltonian action on $N \circ M.$
    \item \textbf{Identity Morphisms:} For any complexified Lie group $G_{\mathbb{C}}$, the identity morphism  $\mathrm{Id}_{G_{\mathbb{C}}} \in \operatorname{Mor}(G_{\mathbb{C}}, G_{\mathbb{C}})$ is the cotangent bundle $T^* G_{\mathbb{C}} \cong G_{\mathbb{C}} \times \mathfrak{g}_{\mathbb{C}}$ equipped with its canonical holomorphic symplectic form $\Omega_{\mathrm{can}}$ and the Hamiltonian $G_{\mathbb{C}} \times G_{\mathbb{C}}$-action generated by
    \[
    \mu(g, \xi) = (-\xi, \, \mathrm{Ad}_g^* \xi).
    \]
\end{itemize}
\end{Def}

\begin{Def}[Categorical Completion $\mathbf{Ham}_{\mathbb{C}}$]
Applying the Wehrheim--Woodward construction \cite{WehrheimWoodward2010Functoriality} to the partial category $\mathbf{Ham}_{\mathbb{C}}^{\mathrm{pre}}$ yields the completed category $\mathbf{Ham}_{\CC}:$
\begin{itemize}
    \item The objects of $\mathbf{Ham}_{\mathbb{C}}$ are finite Cartesian products of connected complex reductive Lie groups $\prod_{i=1}^n G_{\mathbb{C}}^{(i)}.$
    \item A morphism from $G$ to $H$ is an equivalence class of formal finite sequences $(M_1, M_2, \dots, M_k)$ of composable basic morphisms, modulo the equivalence relation generated by composing adjacent pairs whenever the intermediate holomorphic reduction is smooth (see \cite{maiza2026multiplicative,CrooksMayrand2024MooreTachikawa}).
    \item The monoidal structure is given by the Cartesian product on objects and morphisms, making $\mathbf{Ham}_{\mathbb{C}}$ a symmetric monoidal category (see also \cite{maiza2026multiplicative,maiza2026lax}).
\end{itemize}
\end{Def}

\subsection{Cobordisms and Quivers}

Let $\mathbf{Cob}_2$ be the symmetric monoidal category whose objects are finite disjoint unions of oriented circles $\bigsqcup^n S^1$ and whose morphisms are diffeomorphism classes (relative to the boundary) of compact oriented 2-dimensional cobordisms $\Sigma : \overline{\Sigma}_{\mathrm{in}} \to \Sigma_{\mathrm{out}}.$

\begin{Def}
Every connected oriented cobordism $\Sigma$ with $m$ incoming and $n$ outgoing boundary components deformation-retracts onto an oriented quiver with boundary $\Gamma = (E, V, s, t)$ such that:
\begin{enumerate}
    \item the boundary vertices satisfy $|\partial\Gamma^-| = m$ and $|\partial\Gamma^+| = n;$
    \item the topological thickening $\Sigma(\Gamma)$, obtained by replacing edges with cylinders $I_e \times [-\epsilon, \epsilon]$ and interior vertices with oriented discs glued according to cyclic orderings of incident half-edges, is diffeomorphic to $\Sigma$.
\end{enumerate}
\end{Def}

Given two quivers $\Gamma_1, \Gamma_2$ with non-empty boundaries and an identification of boundary vertices $D \coloneqq  \partial\Gamma_1^+ = \partial\Gamma_2^-$, the glued quiver $\Gamma_1 *_D \Gamma_2$ is obtained by identifying each $v \in D$ with an interior vertex of the glued quiver.

\begin{Prop}[Symplectic Cutting and Gluing Law]
\label{thm:cutting-gluing}
Let $\Gamma_1$ and $\Gamma_2$ be connected quivers with $D := \partial\Gamma_1^+ = \partial\Gamma_2^- \neq \emptyset.$ Then there is a canonical isomorphism of complex Hamiltonian $G_{\CC}^{\partial\Gamma_1^-} \times G_{\CC}^{\partial\Gamma_2^+}$-spaces:
\[
\M_{\CC}(\Gamma_1 *_D \Gamma_2) \cong \left( \M_{\CC}(\Gamma_1) \times \M_{\CC}(\Gamma_2) \right) /\!\!/\!\!/_{\CC}\, G_{\mathbb{C}}^D,
\]
where the reduction is taken with respect to the moment map $\mu_{\Delta, D} = \mu_{\mathrm{out}}^{(1)} - \mu_{\mathrm{in}}^{(2)}.$
\end{Prop}

\begin{proof}
This follows immediately from Theorem \ref{thm:finite-dim-reduction} and reduction by stages.
\end{proof}

\subsection{Quiver Independence (Topological Invariance)}

\begin{Def} \label{def:homotopic}
An \defn{edge contraction} of a quiver $\Gamma$ with boundary removes an edge $e_0$ whose endpoints are two \emph{distinct interior} vertices $v_1 = s(e_0) \neq v_2 = t(e_0)$ and identifies $v_1$ with $v_2$ (Figure \ref{fig:homotopy}); its inverse is called an \defn{edge expansion}. Two quivers $\Gamma_{1}$ and $\Gamma_2$ are \defn{homotopic} if one can be obtained from the other by a finite sequence of edge contractions and edge expansions.
\end{Def}

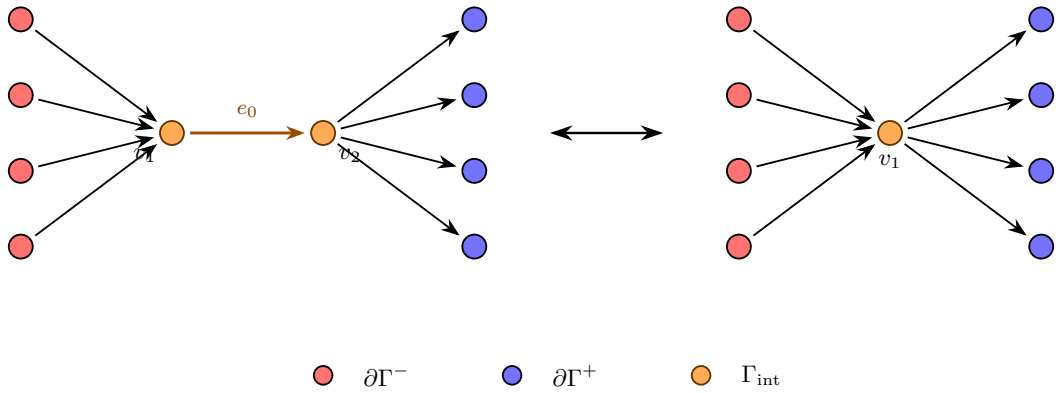
\begin{figure}[ht]
\centering
\begin{tikzpicture}[
    vertex/.style={circle, draw=black, line width=0.7pt, minimum size=9pt, inner sep=0pt},
    bdin/.style={vertex, fill=red!55},
    bdout/.style={vertex, fill=blue!55},
    int/.style={vertex, fill=black!70},
    contract/.style={vertex, fill=orange!65, draw=orange!40!black},
    arr/.style={-{Stealth[length=2.5mm, width=1.8mm]}, line width=0.7pt, shorten <=2pt, shorten >=2pt},
    contract_arr/.style={arr, orange!60!black, line width=1.1pt},
    lbl/.style={font=\footnotesize}
]
    \begin{scope}[shift={(0,0)}]
        \node[bdin] (a1) at (0, 1.5) {};
        \node[bdin] (a2) at (0, 0.5) {};
        \node[bdin] (a3) at (0, -0.5) {};
        \node[bdin] (a4) at (0, -1.5) {};
        
        \node[contract] (v1) at (2, 0) {};
        \node[contract] (v2) at (4, 0) {};
        
        \node[bdout] (b1) at (6, 1.5) {};
        \node[bdout] (b2) at (6, 0.5) {};
        \node[bdout] (b3) at (6, -0.5) {};
        \node[bdout] (b4) at (6, -1.5) {};
        
        \draw[arr] (a1) -- (v1);
        \draw[arr] (a2) -- (v1);
        \draw[arr] (a3) -- (v1);
        \draw[arr] (a4) -- (v1);
        
        \draw[contract_arr] (v1) -- (v2) node[midway, above=2pt, font=\scriptsize, orange!50!black] {$e_0$};
        
        \draw[arr] (v2) -- (b1);
        \draw[arr] (v2) -- (b2);
        \draw[arr] (v2) -- (b3);
        \draw[arr] (v2) -- (b4);
        
        \node[lbl, below left=0.1cm] at (v1) {$v_1$};
        \node[lbl, below right=0.1cm] at (v2) {$v_2$};
    \end{scope}
    
    \draw[{Stealth[length=3mm, width=2mm]}-{Stealth[length=3mm, width=2mm]}, line width=0.9pt] 
        (7, 0) -- (8.5, 0);
    
    \begin{scope}[shift={(9.5,0)}]
        \node[bdin] (ra1) at (0, 1.5) {};
        \node[bdin] (ra2) at (0, 0.5) {};
        \node[bdin] (ra3) at (0, -0.5) {};
        \node[bdin] (ra4) at (0, -1.5) {};
        
        \node[contract] (rv) at (2, 0) {};
        
        \node[bdout] (rb1) at (4, 1.5) {};
        \node[bdout] (rb2) at (4, 0.5) {};
        \node[bdout] (rb3) at (4, -0.5) {};
        \node[bdout] (rb4) at (4, -1.5) {};
        
        \draw[arr] (ra1) -- (rv);
        \draw[arr] (ra2) -- (rv);
        \draw[arr] (ra3) -- (rv);
        \draw[arr] (ra4) -- (rv);
        \draw[arr] (rv) -- (rb1);
        \draw[arr] (rv) -- (rb2);
        \draw[arr] (rv) -- (rb3);
        \draw[arr] (rv) -- (rb4);
        
        \node[lbl, below=0.15cm] at (rv) {$v_1$};
    \end{scope}
    
    \begin{scope}[shift={(4,-3.2)}]
        \node[bdin, minimum size=7pt] at (0, 0) {};
        \node[lbl, right=0.2cm] at (0.2, 0) {$\partial\Gamma^-$};
        
        \node[bdout, minimum size=7pt] at (2.5, 0) {};
        \node[lbl, right=0.2cm] at (2.7, 0) {$\partial\Gamma^+$};
        
        \node[contract, minimum size=7pt] at (5, 0) {};
        \node[lbl, right=0.2cm] at (5.2, 0) {$\Gamma_{\mathrm{int}}$};
    \end{scope}
\end{tikzpicture}
\caption{The homotopy move on quivers: contracting an edge $e_0$ between two interior vertices $v_1, v_2 \in \Gamma_{\mathrm{int}}$ yields a homotopy-equivalent quiver.}
\label{fig:homotopy}
\end{figure}

\begin{Prop}[\cite{maiza2026lax,maiza2026multiplicative}][Invariance under homotopy moves]
\label{thm:spine-invariance}
Let $\Gamma$ and $\Gamma_1$ be two homotopic oriented quivers. Then, there exists a canonical isomorphism of complex Hamiltonian spaces
\[
\M_{\CC}(\Gamma) \cong \M_{\CC}(\Gamma_1).
\]
Consequently, the isomorphism class of $\mathcal{M}_{\mathbb{C}}(\Sigma) := \mathcal{M}_{\mathbb{C}}(\Gamma)$ depends only on the underlying cobordism $\Sigma$.
\end{Prop}

\begin{proof}
This follows from the length-independence of $\M_{\CC}(\Gamma)$    (Proposition \ref{prop:length-invariance-complex}) and the proof of \cite[\text{Theorem 6.4}]{maiza2026lax}. 
\end{proof}

\subsection{The 2D TQFT Functor}

\begin{Thm}[The 2D TQFT Functor]
\label{thm:2d-tqft-functor}

There exists a unique symmetric monoidal functor
\[
\mathcal{Z} : \mathbf{Cob}_2 \longrightarrow \mathbf{Ham}_{\mathbb{C}}
\]
defined by:
\begin{itemize}
    \item \textbf{On Objects:} For $n \ge 0$ disjoint circles,
    \[
    \mathcal{Z}\left(\bigsqcup^n S^1\right) := G_{\mathbb{C}}^n, \qquad \mathcal{Z}(\emptyset) := \{*\}.
    \]
    \item \textbf{On Morphisms:} For any cobordism $\Sigma : \bigsqcup^m S^1 \to \bigsqcup^n S^1$ represented by an oriented quiver with boundary $\Gamma$,
    \[
    \mathcal{Z}(\Sigma) := \mathcal{M}_{\mathbb{C}}(\Gamma),
    \]
    equipped with its descended holomorphic symplectic structure and boundary evaluation moment map $(\mu_{\mathrm{in}}, \mu_{\mathrm{out}}) : \mathcal{M}_{\mathbb{C}}(\Gamma) \to (\mathfrak{g}_{\mathbb{C}}^*)^m \times (\mathfrak{g}_{\mathbb{C}}^*)^n$.
\end{itemize}
\end{Thm}

\begin{proof}
The proof follows the same steps as the proofs of \cite[\text{Theorem 7.4}]{maiza2026multiplicative} and \cite[\text{Theorem 7.2}]{maiza2026lax}.
\end{proof}

\subsection{The Obstruction to a Hyperkähler TQFT}

Let $\Gamma_e$ denote the quiver with a single edge $e$ joining two boundary vertices, which is a spine of the cylinder, and let $\M_{\bh}(\Gamma_e,L)$ denote its Nahm--Kirchhoff moduli space for $I_e = [0,L]$. It carries the residual Hamiltonian action of $G\times G = G^{\partial\Gamma_e}$. Repeating the integration by parts of the proof of Proposition \ref{moment iden} for $u\in \Lie(\G(\Gamma_e))$ not vanishing at $\partial\Gamma_e$, the boundary terms do not cancel, and one finds that, up to an overall sign convention, the hyperkähler moment map of this action is
\[
\mu^{\partial}_a([A]) = \bigl(-A_a(0),\, A_a(L)\bigr) \in \g\times\g \cong \g^*\times\g^*, \qquad a = 1,2,3.
\]
We denote by $g_{\mathrm{bi}}$ the bi-invariant metric on $G$ induced by the $\Ad_G$-invariant inner product on $\g$.

\begin{Prop} \label{prop:cylinder-length}
Assume $\dim G > 0$.
\begin{enumerate}
    \item The zero fibre $Z_L \coloneqq (\mu^{\partial})^{-1}(0) \subset \M_{\bh}(\Gamma_e,L)$ is the set of classes $[A]$ with $A_1 = A_2 = A_3 = 0$. It is a single $G\times G$-orbit, and the holonomy map $Z_L \to G$ is a diffeomorphism which is an isometry from the induced metric onto $(G, L^{-1}g_{\mathrm{bi}})$.
    \item For $L\neq L'$, the spaces $\M_{\bh}(\Gamma_e,L)$ and $\M_{\bh}(\Gamma_e,L')$ are not isomorphic as hyperkähler Hamiltonian $G\times G$-spaces, i.e.\ there is no $G\times G$-equivariant hyperkähler isometry between them intertwining the moment maps.
\end{enumerate}
\end{Prop}

\begin{proof}
(1) Let $[A]\in Z_L$, so that $A_a(0) = A_a(L) = 0$ for $a = 1,2,3$. With $\alpha = A_0 + iA_1$ and $\beta = A_2 + iA_3$, the complex Nahm equation $\dot{\beta} + [\alpha,\beta] = 0$ gives $\beta(t) = \Ad_{g_\alpha(t)}\beta(0) = 0$, so $A_2 = A_3 = 0$. The first Nahm equation then reads $\dot{A}_1 + [A_0,A_1] = 0$, so $A_1$ is parallel and $A_1(0) = 0$ forces $A_1 \equiv 0$. Conversely, every $A = (A_0,0,0,0)$ solves Nahm's equations (there are no interior vertices). Two connections $A_0, A'_0$ on $[0,L]$ are related by a gauge transformation trivial at both ends if and only if they have the same holonomy $h = g_{A_0}(L)$, and $(g_0,g_1)\in G\times G$ acts on holonomies by $h\mapsto g_1hg_0^{-1}$. Hence the holonomy map $Z_L\to G$ is a $G\times G$-equivariant bijection onto a homogeneous space, so $Z_L$ is a single orbit of a compact group, hence a compact embedded submanifold, and the holonomy map is a diffeomorphism.

To compute the induced metric, note that since $G$ is compact and connected, every $h\in G$ is of the form $\exp(-L\xi)$ for some $\xi\in\g$, so every point of $Z_L$ is represented by a constant $A_0\equiv\xi$. The horizontal tangent vectors to $Z_L$ at $[A]$ are the $Y = (Y_0,0,0,0)$ with $D^{*}_AY = \dot{Y}_0 + [\xi,Y_0] = 0$ (there are no vertex terms, and the orbit directions $D_Au$ also have vanishing Higgs components), that is, $Y_0(t) = \Ad_{\exp(-t\xi)}c$ for some $c\in\g$, with $\|Y\|^2_{L^2} = L\|c\|^2$. Writing the variation of the holonomy as $\delta g_{A_0}(t) = g_{A_0}(t)\eta(t)$ with $g_{A_0}(t) = \exp(-t\xi)$, the linearized transport equation gives $\dot{\eta} = -g_{A_0}^{-1}Y_0\,g_{A_0} = -c$, so $\eta(L) = -Lc$ and $\|\delta h\|_{g_{\mathrm{bi}}} = L\|c\|$. Therefore the induced metric is $\|\delta h\|^2_{g_{\mathrm{bi}}}/L$, as claimed.

(2) An isomorphism $f$ of hyperkähler Hamiltonian $G\times G$-spaces satisfies $\mu'^{\partial}\circ f = \mu^{\partial}$, hence maps $Z_L$ onto $Z_{L'}$ by a Riemannian isometry. By (1), $\mathrm{vol}(Z_L) = L^{-\dim G/2}\,\mathrm{vol}_{g_{\mathrm{bi}}}(G)$, which differs from $\mathrm{vol}(Z_{L'})$ when $L\neq L'$ because $\dim G > 0$; contradiction.
\end{proof}

\begin{Rem}[The obstruction to a hyperkähler TQFT]
\label{rem:obstruction}
One might ask whether the hyperkähler moduli spaces $\mathcal{M}_{\mathbb{H}}(\Gamma)$ assemble into a 2D TQFT valued in a hyperkähler category whose morphisms are hyperkähler Hamiltonian manifolds composed by hyperkähler reduction $/\!\!/\!\!/_{\mathbb{H}}$, by setting $\Sigma\mapsto \M_{\bh}(\Gamma)$ for any spine $\Gamma$ of $\Sigma$.
\begin{enumerate}
    \item This assignment is not well defined on $\mathbf{Cob}_2$: the cylinder admits the spines $(\Gamma_e,L)$ for every $L > 0$, and by Proposition \ref{prop:cylinder-length} the corresponding hyperkähler spaces are pairwise non-isomorphic.
    \item The obstruction is due entirely to the total length of the edges: by Theorem \ref{thm:hklr-metric-gluing}, subdividing an edge into two sub-edges whose lengths add up to the original length is a hyperkähler isometry. It is the total length, which $\mathbf{Cob}_2$ does not record, that changes the hyperkähler isomorphism class.
    \item By Proposition \ref{prop:metric-length-dependence}(1), a uniform rescaling of all lengths acts by a homothety. Hence the obstruction for the cylinder disappears if hyperkähler spaces are only considered up to homothety. For general spines, Proposition \ref{prop:metric-length-dependence}(2) shows that the transported metrics $g^{X}_{\mathbf{l}}$ are not related by the canonical identifications; we make no claim about the existence of other identifications.
\end{enumerate}
This motivates retaining the lengths as part of the data, which we do in Section \ref{section 6}.
\end{Rem}

\section{Metric Cobordisms, Hyperkähler Geometry, and Tropical Families} \label{section 6}

In this section, we formulate the metric refinement of the 2D field theory using the Nahm--Kirchhoff moduli spaces $\M_{\bh}(\Gamma, \mathbf{l}).$ While topological invariance fails for the hyperkähler metric (Proposition \ref{prop:cylinder-length}), the moduli spaces assemble functorially into a \emph{metric field theory}, namely a functor of non-unital monoidal categories, and over each cone of the tropical moduli space of curves they form a family with fixed complex-symplectic fibre.

\subsection{The Category $\mathbf{MetCob}_2$ of Metric Cobordisms}

\begin{Def}[Metric Cobordisms]
The non-unital monoidal category of metric 2-cobordisms, denoted $\mathbf{MetCob}_2$, is defined as follows:
\begin{itemize}
    \item \textbf{Objects:} Finite disjoint unions of metric circles
    \[
    \bigsqcup_{i=1}^n (S^1, L_i), \qquad L_i \in \mathbb{R}_{>0},
    \]
    where $L_i$ denotes the assigned metric perimeter of the $i$-th circle. The empty union $\emptyset$ serves as the monoidal unit.
    \item \textbf{Morphisms:} A morphism from $\bigsqcup_{i=1}^m (S^1, L_i^{\mathrm{in}})$ to $\bigsqcup_{j=1}^n (S^1, L_j^{\mathrm{out}})$ is an equivalence class of triples $(\Sigma, \Gamma, \mathbf{l})$, where:
    \begin{enumerate}
        \item $\Sigma$ is a smooth compact oriented 2-manifold with boundary decomposed as $\partial\Sigma = \overline{\Sigma}_{\mathrm{in}} \sqcup \Sigma_{\mathrm{out}};$
        \item $\Gamma = (E, V, s, t)$ is an embedded quiver retract of $\Sigma$ with boundary vertices $\partial\Gamma = \partial\Gamma^- \sqcup \partial\Gamma^+$ satisfying $|\partial\Gamma^-| = m$ and $|\partial\Gamma^+| = n;$
        \item $\mathbf{l} : E \to \mathbb{R}_{>0}$, $e \mapsto l_e$, is an edge-length function such that the boundary leaves $e_v$ incident to $v \in \partial\Gamma$ satisfy $l_{e_v} = L_v.$
    \end{enumerate}
    \item \textbf{Equivalence Relation:} Two metric oriented quivers $(\Gamma_1, \mathbf{l}_1)$ and $(\Gamma_2, \mathbf{l}_2)$ representing the same topological surface $\Sigma$ are equivalent if they are related by a finite sequence of \emph{metric edge subdivisions and amalgamations} ($S$-moves), where a subdivision replaces an edge $e$ of length $l_e$ by two edges $e_1, e_2$ with $l_e = l_{e_1} + l_{e_2}$ ($l_{e_1}, l_{e_2} > 0$) joined at a new interior vertex.
    \item \textbf{Composition (Metric Concatenation):} Given $(\Sigma_1, \Gamma_1, \mathbf{l}_1) : \mathbf{L}_1 \to \mathbf{L}_2$ and $(\Sigma_2, \Gamma_2, \mathbf{l}_2) : \mathbf{L}_2 \to \mathbf{L}_3$, the composite morphism $(\Sigma_{12}, \Gamma_{12}, \mathbf{l}_{12})$ is formed by gluing $\Sigma_1$ and $\Sigma_2$ along the common boundary circles. At the level of quivers, for each identified boundary component with perimeter $L$, the outgoing edge $e_1 \in \partial\Gamma_1^+$ and the incoming edge $e_2 \in \partial\Gamma_2^-$ are concatenated into a single edge $e = e_1 \cup e_2$ with assigned length
    \[
    l_e := l_{e_1} + l_{e_2} = 2L.
    \]
    \item \textbf{Monoidal Structure:} Disjoint union of metric circles and disjoint union of metric surfaces.
    \item \textbf{Absence of identities:} The total length $\sum_{e\in E} l_e$ is invariant under $S$-moves, and the total length of a composite is the sum of the total lengths of its factors. Since all lengths are positive, no morphism can act as an identity. Composition is associative, so $\mathbf{MetCob}_2$ is a non-unital monoidal category (a monoidal semicategory), as is usual for bordism categories with metric data.
\end{itemize}
\end{Def}

\subsection{The Hyperkähler Target Category $\mathbf{HK}$}

\begin{Def}[The partial category $\mathbf{HK}^{\text{pre}}$]
The non-unital partial category $\mathbf{HK}^{\text{pre}}$ is defined by:
\begin{itemize}
    \item \textbf{Objects:} Compact connected Lie groups $G$ (with $G^0 = \{*\}$).
    \item \textbf{Morphisms:} A morphism from $G^m$ to $G^n$ is an isomorphism class of smooth (not necessarily complete) hyperkähler manifolds $(M, g_{\mathbb{H}}, I_1, I_2, I_3, \omega_1, \omega_2, \omega_3)$ equipped with a tri-Hamiltonian isometric action of $G^m \times G^n$ admitting an equivariant hyperkähler moment map
    \[
    \mu = (\mu_1, \mu_2, \mu_3) : M \longrightarrow (\mathfrak{g}^* \times \mathfrak{g}^* \times \mathfrak{g}^*)^{m+n}.
    \]
    An isomorphism between two morphisms is a $G^m\times G^n$-equivariant hyperkähler isometry intertwining the moment maps.
    \item \textbf{Composition:} Let $M \in \operatorname{Mor}(G^m, G^k)$ and $N \in \operatorname{Mor}(G^k, G^n)$. If the diagonal action of $G^k$ on $\mu_{\Delta}^{-1}(0)$ is free and proper, the composition $N \circ M \in \operatorname{Mor}(G^m, G^n)$ is defined by the HKLR hyperkähler quotient
    \[
    N \circ M := (M \times N) /\!\!/\!\!/_{\mathbb{H}}\, G^k := \left( \bigcap_{a=1}^3 \mu_{\Delta, a}^{-1}(0) \right) / G^k,
    \]
    where $\mu_{\Delta, a}(x, y) = \mu_{\mathrm{out}, a}^{(M)}(x) - \mu_{\mathrm{in}, a}^{(N)}(y)$ for $a \in \{1, 2, 3\}$.
    \item \textbf{No identities:} We do not adjoin identity morphisms; see Remark \ref{rem:no-identity} below.
\end{itemize}
\end{Def}

\begin{Def}[Categorical Completion $\mathbf{HK}$]
Applying the Wehrheim--Woodward construction \cite{WehrheimWoodward2010Functoriality} to the composition law of $\mathbf{HK}^{\mathrm{pre}}$, using only non-empty formal sequences (so that no identity morphisms are adjoined), yields the non-unital category $\mathbf{HK}:$
\begin{itemize}
    \item The objects of $\mathbf{HK}$ are finite Cartesian products of connected compact Lie groups $\prod_{i=1}^n G^{(i)}.$
    \item A morphism from $G$ to $H$ is an equivalence class of non-empty formal finite sequences $(M_1, M_2, \dots, M_k)$ of composable basic morphisms, modulo the equivalence relation generated by composing adjacent pairs whenever the intermediate hyperkähler reduction is smooth (see \cite{maiza2026multiplicative,CrooksMayrand2024MooreTachikawa}).
    \item The Cartesian product on objects and morphisms makes $\mathbf{HK}$ a non-unital monoidal category (see also \cite{maiza2026multiplicative,maiza2026lax}).
\end{itemize}
\end{Def}

\begin{Rem} \label{rem:no-identity}
The single-edge spaces $\M_{\bh}(\Gamma_e,L)$ cannot serve as identity morphisms. Indeed, by Theorem \ref{thm:hklr-metric-gluing} below, $\M_{\bh}(\Gamma_e,L')\circ \M_{\bh}(\Gamma_e,L) \cong \M_{\bh}(\Gamma_e,L+L')$, which is not isomorphic to $\M_{\bh}(\Gamma_e,L')$ in $\mathbf{HK}^{\mathrm{pre}}$ by Proposition \ref{prop:cylinder-length}. This is the hyperkähler counterpart of the absence of identities in $\mathbf{MetCob}_2$.
\end{Rem}

\subsection{Metric Gluing via Hyperkähler Reduction}

\begin{Thm} \label{thm:hklr-metric-gluing}

Let $\Gamma_1$ and $\Gamma_2$ be metric quivers with length assignments $\mathbf{l}_1, \mathbf{l}_2$. Let $v_1 \in \partial\Gamma_1^+$ and $v_2 \in \partial\Gamma_2^-$ be boundary vertices with incident edges $e_1, e_2$ of lengths $l_{e_1}, l_{e_2}$. Let $\Gamma_1 *_{v} \Gamma_2$ be the quiver obtained by identifying $v_1$ and $v_2$ into an interior vertex $v$, and let $(\Gamma_{12}, \mathbf{l}_{12})$ be the quiver obtained by amalgamating $e_1$ and $e_2$ into a single edge $e = e_1 \cup e_2$ of length $l_e = l_{e_1} + l_{e_2}$. 

Then there are canonical hyperkähler isometries
\[
(\mathcal{M}_{\mathbb{H}}(\Gamma_1, \mathbf{l}_1) \times \mathcal{M}_{\mathbb{H}}(\Gamma_2, \mathbf{l}_2)) /\!\!/\!\!/_{\mathbb{H}}\, G_v \;\cong\; \mathcal{M}_{\mathbb{H}}(\Gamma_1 *_v \Gamma_2, \mathbf{l}_1 \cup \mathbf{l}_2) \;\cong\; \mathcal{M}_{\mathbb{H}}(\Gamma_{12}, \mathbf{l}_{12}).
\]
\end{Thm}

\begin{proof}
Let $I_1 = [0, l_{e_1}]$ and $I_2 = [0, l_{e_2}]$. The unglued ambient space of fields decomposes as
\[
\mathcal{B}_{\mathbb{H}}(\Gamma_1) \times \mathcal{B}_{\mathbb{H}}(\Gamma_2) = C^1(I_1, \mathfrak{g}^4)^{E_1} \times C^1(I_2, \mathfrak{g}^4)^{E_2}.
\]
Evaluating the diagonal hyperkähler moment map $\mu_{\Delta} = (\mu_{\Delta, 1}, \mu_{\Delta, 2}, \mu_{\Delta, 3})$ at $v$ yields
\[
\mu_{\Delta, a}(A_1, A_2) = A_a^{e_2}(0) - A_a^{e_1}(l_{e_1}) \qquad (a = 1, 2, 3).
\]
Setting $\mu_{\Delta, a} = 0$ for all $a \in \{1, 2, 3\}$ forces exact continuity of the three Higgs fields across the junction:
\[
A_a^{e_1}(l_{e_1}) = A_a^{e_2}(0) \qquad (a \in \{1, 2, 3\}).
\]
Furthermore, the gauge group $G_v$ acts on connection components by $g_v \cdot A_0^{e_2}(0) = \mathrm{Ad}_{g_v}(A_0^{e_2}(0)) - \dot{g}_v g_v^{-1}$. Taking the quotient identifies gauge transformations across the vertex, matching the boundary trivialization sequence
\[
1 \longrightarrow \mathcal{G}_0(\Gamma_1) \times \mathcal{G}_0(\Gamma_2) \longrightarrow \mathcal{G}_0(\Gamma_1 *_v \Gamma_2) \longrightarrow G_v \longrightarrow 1.
\]
By infinite-dimensional Marsden--Weinstein reduction by stages applied to each of the three symplectic forms $\omega_1, \omega_2, \omega_3$ (Theorem \ref{infinite red}) and the quaternionic slice criterion (Proposition \ref{q slice}), the intermediate reduction yields
\[
(\mathcal{M}_{\mathbb{H}}(\Gamma_1, \mathbf{l}_1) \times \mathcal{M}_{\mathbb{H}}(\Gamma_2, \mathbf{l}_2)) /\!\!/\!\!/_{\mathbb{H}}\, G_v \cong \mathcal{M}_{\mathbb{H}}(\Gamma_1 *_v \Gamma_2, \mathbf{l}_1 \cup \mathbf{l}_2).
\]
Finally, on $\Gamma_1 *_v \Gamma_2$, $v$ is an interior vertex incident to the edges $e_1, e_2$. The Kirchhoff matching conditions at $v$ are
\[
A_a^{e_1}(l_{e_1}) = A_a^{e_2}(0) \qquad (a=1,2,3).
\]
A pair of $C^1$ solutions $A^{e_1} \in C^1([0, l_{e_1}], \mathfrak{g}^4)$ and $A^{e_2} \in C^1([0, l_{e_2}], \mathfrak{g}^4)$ satisfying these matching conditions, modulo gauge transformations matching at $v$, is uniquely equivalent to a single $C^1$ solution on the concatenated interval $[0, l_{e_1} + l_{e_2}]$ modulo gauge transformations on $[0, l_{e_1} + l_{e_2}]$. The descended $L^2$-metrics agree:
\[
\int_0^{l_{e_1}} \sum_{i=0}^3 \|X_i\|^2 dt + \int_0^{l_{e_2}} \sum_{i=0}^3 \|X_i\|^2 dt = \int_0^{l_{e_1} + l_{e_2}} \sum_{i=0}^3 \|X_i\|^2 dt.
\]
Hence, the identification with $\mathcal{M}_{\mathbb{H}}(\Gamma_{12}, \mathbf{l}_{12})$ is a hyperkähler isometry.
\end{proof}

\subsection{The Metric Field Theory Functor}

\begin{Thm} \label{thm:metric-functor}
The assignment
\[
\mathcal{Z}_{\mathbb{H}} : \mathbf{MetCob}_2 \longrightarrow \mathbf{HK}
\]
defined by
\begin{itemize}
    \item on objects: $\mathcal{Z}_{\mathbb{H}}\left(\bigsqcup_{i=1}^n (S^1, L_i)\right) := G^n$, with $\mathcal{Z}_{\mathbb{H}}(\emptyset) := \{*\}$,
    \item on morphisms: for any metric cobordism $(\Sigma, \Gamma, \mathbf{l})$,
    \[
    \mathcal{Z}_{\mathbb{H}}(\Sigma, \Gamma, \mathbf{l}) := \mathcal{M}_{\mathbb{H}}(\Gamma, \mathbf{l}),
    \]
    equipped with its descended hyperkähler structure $(g_{\mathbb{H}}, I_a, \omega_a)$ and boundary hyperkähler moment map,
\end{itemize}
is a well-defined functor of non-unital monoidal categories: it is compatible with composition and with disjoint unions.
\end{Thm}

\begin{proof}
We verify the axioms; there are no identity axioms to check, since neither category has identities.
\begin{enumerate}
    \item \textbf{Well-definedness on Morphisms:} By Theorem~\ref{thm:hklr-metric-gluing}, subdividing or amalgamating edges ($S$-moves) preserves the isomorphism class in $\mathbf{HK}$, so $\mathcal{Z}_{\mathbb{H}}$ is independent of the representative quiver in the equivalence class.
    \item \textbf{Compatibility with composition:} Let $(\Sigma_{12}, \Gamma_{12}, \mathbf{l}_{12})$ be the composition of $(\Sigma_1, \Gamma_1, \mathbf{l}_1)$ and $(\Sigma_2, \Gamma_2, \mathbf{l}_2)$ along $k$ boundary circles. Concatenating the boundary edges corresponds to hyperkähler reduction by $G^k$ at the identified junctions by Theorem~\ref{thm:hklr-metric-gluing}:
    \[
    \mathcal{Z}_{\mathbb{H}}(\Sigma_2 \circ \Sigma_1) \cong \mathcal{Z}_{\mathbb{H}}(\Sigma_2) \circ \mathcal{Z}_{\mathbb{H}}(\Sigma_1).
    \]
    \item \textbf{Monoidal Multiplicativity:} Disjoint unions of quivers have decoupled Nahm equations and orthogonal $L^2$-metrics, yielding $\mathcal{Z}_{\mathbb{H}}(\Sigma_1 \sqcup \Sigma_2) \cong \mathcal{Z}_{\mathbb{H}}(\Sigma_1) \times \mathcal{Z}_{\mathbb{H}}(\Sigma_2)$.
\end{enumerate}
Thus, $\mathcal{Z}_{\mathbb{H}}$ is a functor of non-unital monoidal categories.
\end{proof}

\subsection{Families over the Tropical Moduli Space $M_{g,n}^{\mathrm{trop}}$}

Let $\Sigma_{g, n}$ be a fixed topological surface of genus $g$ with $n$ labeled boundary components.

\begin{Def}[Tropical Moduli Space]
The \emph{tropical moduli space} $M_{g,n}^{\mathrm{trop}}$ is the polyhedral cone complex parameterizing isomorphism classes of metric ribbon trees/graphs $(\Gamma, \mathbf{l})$ of genus $g$ with $n$ boundary leaves:
\[
M_{g,n}^{\mathrm{trop}} := \left( \bigsqcup_{\Gamma} \mathbb{R}_{>0}^{|E_{\mathrm{int}}(\Gamma)|} \right) \Big/ \sim_{\mathrm{trop}}
\]
where the disjoint union runs over all combinatorial ribbon graph types $\Gamma$ with $\Sigma(\Gamma) \cong \Sigma_{g,n}$, and $\sim_{\mathrm{trop}}$ is the equivalence relation induced by graph automorphisms and identifying faces when internal edge lengths tend to zero (tropical edge contraction). Here $E_{\mathrm{int}}(\Gamma)$ is the set of edges joining two interior vertices; the lengths of the boundary leaves are fixed by the boundary data.
\end{Def}

\begin{Thm}[Hyperkähler families over the cones of $M_{g,n}^{\mathrm{trop}}$]
\label{thm:tropical-bundle}
Fix positive leaf lengths, and let $\Gamma$ be a combinatorial type with $\Sigma(\Gamma)\cong \Sigma_{g,n}$. Let $\mathcal{C}_{\Gamma} \coloneqq \R^{E_{\mathrm{int}}(\Gamma)}_{>0}$ be the corresponding open cone, and for $\mathbf{l}\in \mathcal{C}_\Gamma$ extend $\mathbf{l}$ to all of $E$ by the fixed leaf lengths.
\begin{enumerate}
    \item For every $\mathbf{l}\in\mathcal{C}_\Gamma$, the map $\Psi_{\mathbf{l}}:\M_{\bh}(\Gamma,\mathbf{l})\to X(\Gamma)$ of Subsection \ref{subsec:metric-length} is a biholomorphism for $I_1$ which intertwines $\omega_2 + i\omega_3$ with the reduced holomorphic symplectic form, and $X(\Gamma)\cong \M_{\CC}(\Sigma_{g,n})$ by Proposition \ref{thm:spine-invariance}. Hence the family $\{\M_{\bh}(\Gamma,\mathbf{l})\}_{\mathbf{l}\in\mathcal{C}_\Gamma}$ is identified with the product $\mathcal{C}_\Gamma\times \M_{\CC}(\Sigma_{g,n})$ as a family of complex-symplectic manifolds, carrying the fibrewise hyperkähler metrics $g^{X}_{\mathbf{l}}$.
    \item Assume $\dim G > 0$ and $E_{\mathrm{int}}(\Gamma)\neq\emptyset$. For $\mathbf{l}\in\mathcal{C}_\Gamma$ and $\lambda > 0$, let $\mathbf{l}^{\lambda}$ be obtained from $\mathbf{l}$ by multiplying the internal edge lengths by $\lambda$. If $\lambda\neq 1$, then $g^{X}_{\mathbf{l}^{\lambda}}\neq g^{X}_{\mathbf{l}}$. In particular, the hyperkähler metric varies non-trivially along every ray of $\mathcal{C}_\Gamma$.
\end{enumerate}
\end{Thm}

\begin{proof}
(1) The map $\Psi_{\mathbf{l}}$ is a biholomorphism by Theorems \ref{complex q} and \ref{thm:finite-dim-reduction}. The forms $\omega_2 + i\omega_3$ on $\M_{\bh}(\Gamma,\mathbf{l})$ and the reduced holomorphic symplectic form on $\M_{\CC}(\Gamma,\mathbf{l})$ are both characterized by the property that their pullbacks to $\A_{\bh}(\Gamma)\subset \A_{\CC}(\Gamma)$ coincide with the restriction of the ambient form $\omega_2 + i\omega_3 = \Omega_{\CC}$, so $\Phi$ intertwines them; $\varphi$ is a holomorphic symplectomorphism by Theorem \ref{thm:finite-dim-reduction}. The target $X(\Gamma)$ does not depend on $\mathbf{l}$ (Proposition \ref{prop:length-invariance-complex} and Subsection \ref{subsec:metric-length}).

(2) By Proposition \ref{prop:metric-length-dependence}(2), for every $\zeta\in K(\Gamma)$,
\[
g^{X}_{\mathbf{l}^{\lambda}}(\hat{\zeta},\hat{\zeta}) - g^{X}_{\mathbf{l}}(\hat{\zeta},\hat{\zeta}) = (\lambda - 1)\sum_{e\in E_{\mathrm{int}}(\Gamma)} l_e\,\|\zeta_e\|^2.
\]
It therefore suffices to find $\zeta\in K(\Gamma)$ with $\zeta_{e}\neq 0$ for some internal edge $e$. Fix $w\in\g_{\CC}\setminus\{0\}$. If $C$ is a cycle, or a path joining two boundary vertices, in the underlying undirected graph of $\Gamma$, set $\zeta_e \coloneqq \pm w$ for $e\in C$, with the sign $+$ if and only if the orientation of $e$ agrees with the direction of traversal, and $\zeta_e \coloneqq 0$ otherwise. At each interior vertex of $C$ one traversed edge enters and one leaves, so their contributions to the Kirchhoff sum cancel, and $\zeta\in K(\Gamma)$. It remains to show that some internal edge lies on such a $C$. Let $e_1$ be an internal edge. If $e_1$ lies on a cycle, or on a path between two boundary vertices, we are done. Otherwise, $e_1$ is a bridge and one of the two components, say $C_1$, of the graph obtained by deleting $e_1$ contains no boundary vertex. Let $w_0$ be the endpoint of $e_1$ in $C_1$. Every vertex of $C_1$ is interior and therefore has degree at least $2$ in $\Gamma$; hence every vertex of $C_1$ other than $w_0$ has degree at least $2$ in $C_1$, and $w_0$ has degree at least $1$ in $C_1$. If $C_1$ has a single vertex, this forces a loop at $w_0$. Otherwise, if $C_1$ were a tree it would have at least two vertices of degree $1$, one of which would differ from $w_0$, a contradiction. In both cases $C_1$ contains a cycle, all of whose edges join interior vertices, i.e.\ are internal. The corresponding $\zeta$ has $\zeta_e = \pm w\neq 0$ on these internal edges, and $\sum_{e\in E_{\mathrm{int}}}l_e\|\zeta_e\|^2 > 0$.
\end{proof}

\begin{Rem}
Theorem \ref{thm:tropical-bundle} describes the family over each open cone of $M^{\mathrm{trop}}_{g,n}$ separately. Passing to $M^{\mathrm{trop}}_{g,n}$ itself requires, in addition, dividing by the automorphisms of $\Gamma$, which act compatibly on $\mathcal{C}_\Gamma$ and on $X(\Gamma)$ by permuting edges, and gluing the families across the faces where internal edges collapse. The behaviour of the harmonic representatives across such faces is described in Proposition \ref{prop:edge-collapse} below; the corresponding statement for the metrics is Conjecture \ref{conj:edge-collapse}. We therefore do not assert the existence of a global smooth fibre bundle over the cone complex $M^{\mathrm{trop}}_{g,n}$.
\end{Rem}

\subsection{Asymptotic Limits and Tropical Degenerations}

\begin{Prop}[Edge collapse: convergence of harmonic representatives] \label{prop:edge-collapse}
Let $e_0\in E$ be an edge joining two distinct interior vertices $v_1 = s(e_0)\neq v_2 = t(e_0)$, and let $\Gamma/\{e_0\}$ be the contracted quiver, which is connected with non-empty boundary. Fix all edge lengths except $\ell \coloneqq l_{e_0}\in(0,1]$, and fix the data $(x_e)_{e\in E}$ and the boundary values defining the reduced Kempf--Ness functionals, in a gauge in which the holonomy along $e_0$ is trivial (this can be achieved with the factor $G_{\CC}^{v_2}$ of the gauge action), so that the term of $e_0$ in $\bar{F}_\ell$ is $E_{x_{e_0}}(p_{v_1},p_{v_2};\ell)$, with the notation of Remark \ref{rem:rescaled}.
\begin{enumerate}
    \item For every compact $K\subset H$ there is $C_K > 0$ such that, for all $p,q\in K$ and $\ell\in(0,1]$, the rescaled minimizer $\tilde{h}^{*}_\ell$ of Remark \ref{rem:rescaled} and the geodesic $\gamma_{p,q}:[0,1]\to H$ from $p$ to $q$ satisfy
    \[
    \sup_{s\in[0,1]} d_H\bigl(\tilde{h}^{*}_\ell(s),\gamma_{p,q}(s)\bigr) \leq C_K\,\ell^2,
    \qquad
    \Bigl| E_{x_{e_0}}(p,q;\ell) - \frac{d_H(p,q)^2}{\ell} - \ell\int_0^1 V_{x_{e_0}}(\gamma_{p,q}(s))\,ds \Bigr| \leq C_K\,\ell^3.
    \]
    \item Let $p^{*}(\ell)\in H^{\Gamma_{\nt}}$ be the unique minimizer of $\bar{F}_\ell$, and let $p^{*}_{\mathrm{col}}$ be the unique minimizer of the reduced Kempf--Ness functional $\bar{F}_{\mathrm{col}}$ of $\Gamma/\{e_0\}$ with the same data. Then $d_H(p^{*}_{v_1}(\ell),p^{*}_{v_2}(\ell)) = O(\ell^{1/2})$, and $p^{*}(\ell)$ converges, as $\ell\to 0$, to the point of $H^{\Gamma_{\nt}}$ obtained from $p^{*}_{\mathrm{col}}$ by assigning to both $v_1$ and $v_2$ the value of $p^{*}_{\mathrm{col}}$ at the merged vertex.
\end{enumerate}
\end{Prop}

\begin{proof}
(1) Write $x = x_{e_0}$. Choose $o\in H$ and $\rho_K > 0$ with $K\subset \bar{B}(o,\rho_K)$. Closed balls in the Cartan--Hadamard manifold $H$ are geodesically convex and compact, so $M_K \coloneqq \sup_{\bar{B}(o,\rho_K)}V_x < \infty$ and $\gamma_{p,q}$ stays in $\bar{B}(o,\rho_K)$. Testing the rescaled functional on $\gamma_{p,q}$, whose speed is $d_H(p,q)$, gives
\begin{equation} \label{eq:collapse-upper}
E_x(p,q;\ell) \leq \frac{d_H(p,q)^2}{\ell} + \ell\int_0^1 V_x(\gamma_{p,q}(s))\,ds \leq \frac{d_H(p,q)^2}{\ell} + \ell M_K.
\end{equation}
Since $V_x\geq 0$, this yields $\int_0^1\|(\tilde{h}^{*}_\ell)'\|^2\,ds \leq \ell E_x(p,q;\ell) \leq 4\rho_K^2 + M_K$, so the length of $\tilde{h}^{*}_\ell$ is at most $\Lambda_K \coloneqq (4\rho_K^2 + M_K)^{1/2}$, and $\tilde{h}^{*}_\ell$ takes values in the compact convex set $K' \coloneqq \bar{B}(o,\rho_K + \Lambda_K)$. Let $G_K \coloneqq \sup_{K'}\|\grad V_x\|$, which also bounds the Lipschitz constant of $V_x$ on $K'$. By the rescaled Euler--Lagrange equation, $\|\nabla_s(\tilde{h}^{*}_\ell)'\| \leq \frac{\ell^2}{2}G_K$.

Let $f(s) \coloneqq d_H(\tilde{h}^{*}_\ell(s),\gamma_{p,q}(s))$, so that $f(0) = f(1) = 0$. On a Cartan--Hadamard manifold, the distance function $d_H$ on $H\times H$ is convex, smooth away from the diagonal, and its gradient with respect to each argument has norm $1$ there. Hence, wherever $f > 0$,
\[
f'' = \mathrm{Hess}\,d_H\bigl(((\tilde{h}^{*}_\ell)',\gamma_{p,q}'),((\tilde{h}^{*}_\ell)',\gamma_{p,q}')\bigr) + \bigl\langle \grad d_H, (\nabla_s(\tilde{h}^{*}_\ell)', \nabla_s\gamma_{p,q}')\bigr\rangle \geq -\tfrac{\ell^2}{2}G_K,
\]
since $\gamma_{p,q}$ is a geodesic. For $\delta > 0$ let $u_\delta(s) \coloneqq (\frac{\ell^2}{4}G_K + \delta)s(1-s)$, so that $u_\delta'' = -\frac{\ell^2}{2}G_K - 2\delta$. If $f - u_\delta$ had a positive maximum, it would be attained at an interior point where $f > u_\delta \geq 0$, so that $f$ is smooth there and $(f - u_\delta)'' \geq 2\delta > 0$, a contradiction. Letting $\delta\to 0$ gives $f \leq \frac{\ell^2}{4}G_K\,s(1-s) \leq \frac{\ell^2}{16}G_K$, which is the first estimate with any $C_K\geq G_K/16$.

For the second estimate, the length of $\tilde{h}^{*}_\ell$ is at least $d_H(p,q)$, so $\int_0^1\|(\tilde{h}^{*}_\ell)'\|^2 \geq d_H(p,q)^2$ by Cauchy--Schwarz, and therefore
\[
E_x(p,q;\ell) \geq \frac{d_H(p,q)^2}{\ell} + \ell\int_0^1 V_x(\tilde{h}^{*}_\ell(s))\,ds \geq \frac{d_H(p,q)^2}{\ell} + \ell\int_0^1 V_x(\gamma_{p,q}(s))\,ds - \ell\cdot G_K\cdot\frac{\ell^2}{16}G_K.
\]
Together with \eqref{eq:collapse-upper}, this gives the second estimate with $C_K \coloneqq \max(G_K/16, G_K^2/16)$.

(2) For edges $e\neq e_0$, Remark \ref{rem:rescaled} gives $E_{x_e}(p,q;l_e)\geq d_H(p,q)^2/l_e$; for $e_0$ and $\ell\leq 1$, $E_{x_{e_0}}(p,q;\ell)\geq d_H(p,q)^2/\ell\geq d_H(p,q)^2$. Hence, for $\ell\leq 1$,
\[
\bar{F}_\ell(p) \geq D(p) \coloneqq \sum_{e\neq e_0}\frac{d_H(p_{s(e)},p_{t(e)})^2}{l_e} + d_H(p_{v_1},p_{v_2})^2,
\]
and $D$ is proper on $H^{\Gamma_{\nt}}$ by the argument of Proposition \ref{proper}. Let $p^{c}\in H^{\Gamma_{\nt}}$ be obtained from $p^{*}_{\mathrm{col}}$ by assigning its value $y$ at the merged vertex to both $v_1$ and $v_2$. Using the constant path on $e_0$, $E_{x_{e_0}}(y,y;\ell)\leq \ell V_{x_{e_0}}(y)$, so
\begin{equation} \label{eq:collapse-bound}
\bar{F}_\ell(p^{*}(\ell)) \leq \bar{F}_\ell(p^{c}) \leq \bar{F}_{\mathrm{col}}(p^{*}_{\mathrm{col}}) + \ell V_{x_{e_0}}(y) \leq C_0 \coloneqq \bar{F}_{\mathrm{col}}(p^{*}_{\mathrm{col}}) + V_{x_{e_0}}(y).
\end{equation}
Since all terms of $\bar{F}_\ell$ are non-negative, $d_H(p^{*}_{v_1}(\ell),p^{*}_{v_2}(\ell))^2/\ell\leq C_0$, which is the first claim. Moreover, $\{p^{*}(\ell)\}_{\ell\in(0,1]}\subset\{D\leq C_0\}$, a compact set. Let $\bar{p}$ be the limit of a sequence $p^{*}(\ell_k)$ with $\ell_k\to 0$. Then $\bar{p}_{v_1} = \bar{p}_{v_2}$, so $\bar{p}$ defines a point of $H^{(\Gamma/\{e_0\})_{\nt}}$, and by continuity of the convex functions $E_{x_e}(\cdot,\cdot;l_e)$ (Proposition \ref{jgc}) and \eqref{eq:collapse-bound},
\[
\bar{F}_{\mathrm{col}}(\bar{p}) = \lim_{k\to\infty}\sum_{e\neq e_0}E_{x_e}\bigl(p^{*}_{s(e)}(\ell_k),p^{*}_{t(e)}(\ell_k);l_e\bigr) \leq \liminf_{k\to\infty}\bar{F}_{\ell_k}(p^{*}(\ell_k)) \leq \bar{F}_{\mathrm{col}}(p^{*}_{\mathrm{col}}).
\]
Since $\Gamma/\{e_0\}$ is connected with non-empty boundary, $p^{*}_{\mathrm{col}}$ is the unique minimizer of $\bar{F}_{\mathrm{col}}$ (Theorem \ref{kempf ness}), so $\bar{p} = p^{c}$. As every sequence $\ell_k\to 0$ has a subsequence along which $p^{*}(\ell_k)\to p^{c}$, we conclude that $p^{*}(\ell)\to p^{c}$.
\end{proof}

By Proposition \ref{thm:spine-invariance}, $X(\Gamma)\cong X(\Gamma/\{e_0\})$, so the transported metrics of $\Gamma$ and of $\Gamma/\{e_0\}$ live on the same complex manifold. Proposition \ref{prop:edge-collapse} controls the vertex values of the harmonic representatives; controlling the metric additionally requires uniform estimates for the horizontal spaces $T_A\A_{\bh}(\Gamma)\cap\Ker D^{*}_A$ along $e_0$, which we do not establish here.

\begin{Conj}[Edge collapse of the metric] \label{conj:edge-collapse}
In the situation of Proposition \ref{prop:edge-collapse}, the transported hyperkähler metrics $g^{X}_{\mathbf{l}}$ converge, as $l_{e_0}\to 0$ and uniformly on compact subsets of $X(\Gamma)\cong X(\Gamma/\{e_0\})$, to the transported metric of $\Gamma/\{e_0\}$ with the induced lengths. In particular, the pointed hyperkähler manifolds $(\M_{\bh}(\Gamma,\mathbf{l}),\Psi_{\mathbf{l}}^{-1}(x))$ converge in the pointed Gromov--Hausdorff sense, for every $x\in X(\Gamma)$.
\end{Conj}

We now turn to long edges. Recall that $H = G\backslash G_{\CC}$ is the space of cosets $Gh$, that right translations by elements of $G_{\CC}$ act on $H$ by isometries (the geodesics through $[h]$ are $s\mapsto G\exp(isw)h$), and that $V_x(Gh) = \|\Ad_h x\|^2$. Let $x\in\g_{\CC}$ be regular semisimple. Then $x = \Ad_k x_0$ for some $k\in G_{\CC}$ and $x_0\in\t_{\CC}$, where $\t$ is the Lie algebra of a maximal torus $T\subset G$, and $V_x(Gh) = V_{x_0}(Ghk)$. Since right translation by $k$ is an isometry of $H$, we may and do assume $x\in\t_{\CC}$ in the next two lemmas. The centralizer of $x$ in $\g_{\CC}$ is then $\t_{\CC}$, and we set
\[
Z(x) \coloneqq \{G\tau : \tau\in T_{\CC}\} = \{G\exp(iY) : Y\in\t\}\subset H,
\]
which is the orbit through $[\mathbf{1}]$ of the right action of $\exp(i\t)$. Let $\g^{\perp} \coloneqq \bigl(\bigoplus_{\alpha\in\Delta^+}(\g_\alpha\oplus\g_{-\alpha})\bigr)\cap\g$ be the sum of the root spaces, so that $\g = \t\oplus\g^{\perp}$ orthogonally.

\begin{Lem} \label{lem 6.10}
Let $x\in\t_{\CC}$ be regular semisimple, and set $c_1 \coloneqq \min\{\|[w,x]\|^2 : w\in\g^{\perp},\ \|w\| = 1\}$. Then:
\begin{enumerate}
    \item $Z(x)$ is a complete, closed, totally geodesic, flat submanifold of $H$, the normal exponential map $\exp^{\perp}:\nu(Z(x))\to H$ is a diffeomorphism, and $V_x$ attains its minimum $\|x\|^2$ at every point of $Z(x)$; moreover $c_1 > 0$.
    \item There is $\epsilon\in(0,1]$ such that, for every unit-speed geodesic $\sigma$ with $\sigma(0)\in Z(x)$ and $\dot{\sigma}(0)\perp T_{\sigma(0)}Z(x)$, one has $\mathrm{Hess}\,V_x(\dot{\sigma}(s),\dot{\sigma}(s))\geq c_1$ for $s\in[0,\epsilon]$.
    \item For every $h\in H$ with $d_H(h,Z(x))\leq\epsilon$, $V_x(h) - \|x\|^2\geq \frac{c_1}{2}\,d_H(h,Z(x))^2$.
\end{enumerate}
\end{Lem}

\begin{proof}
(1) The right action of $\exp(i\t)$ on $H$ is isometric, and $Z(x)$ is its orbit through $[\mathbf{1}]$. The curves $s\mapsto G\exp(isY)$, $Y\in\t$, are geodesics, and since $\t$ is abelian the sectional curvature $-\frac{1}{4}\|[u,v]\|^2$ vanishes on $T_{[\mathbf{1}]}Z(x) = i\t$; hence $Z(x)$ is the image of the totally geodesic flat $\exp(i\t)\cong\R^{\mathrm{rank}\,G}$, which is complete and closed. A closed, complete, totally geodesic submanifold of a Cartan--Hadamard manifold is convex, so its normal exponential map is a diffeomorphism. Since $\tau\in T_{\CC}$ centralizes $x$, $V_x(G\tau) = \|x\|^2$, and $V_x$ is invariant under the right action of $T_{\CC}$. At $[\mathbf{1}]$ and for $w\in\g$, decompose $w = w_{\t} + w^{\perp}$; then $[iw_{\t},x] = 0$ and $[iw^{\perp},x]$ lies in the sum of the root spaces of $\g_{\CC}$, which is orthogonal to $\t_{\CC}\ni x$ for the Hermitian extension of the inner product. Hence
\[
\frac{d}{ds}\Big|_{s=0}V_x(G\exp(isw)) = 2\,\mathrm{Re}\,\langle [iw,x],x\rangle = 0,
\]
so $[\mathbf{1}]$, and by invariance every point of $Z(x)$, is a critical point of $V_x$. Since $V_x$ is geodesically convex (Proposition \ref{potential}) on the Cartan--Hadamard manifold $H$, every critical point is a global minimum. Finally, $w\in\g^{\perp}\setminus\{0\}$ does not lie in the centralizer $\t_{\CC}$ of $x$, so $[w,x]\neq 0$; by compactness of the unit sphere of $\g^{\perp}$, $c_1 > 0$.

(2) By invariance of $V_x$ and of the metric under the right action of $\exp(i\t)$, which is transitive on $Z(x)$, it suffices to treat $\sigma(0) = [\mathbf{1}]$, where $\nu_{[\mathbf{1}]}Z(x) = i\g^{\perp}$. For a unit vector $n = iw$, $w\in\g^{\perp}$, let $\sigma_n(s) = G\exp(isw)$. By Proposition \ref{potential} at $s = 0$, $\mathrm{Hess}\,V_x(\dot{\sigma}_n(0),\dot{\sigma}_n(0)) = 2\|[w,x]\|^2\geq 2c_1$. The function $(n,s)\mapsto\mathrm{Hess}\,V_x(\dot{\sigma}_n(s),\dot{\sigma}_n(s))$ is continuous on the compact set $S(i\g^{\perp})\times[0,1]$, hence uniformly continuous, so there is $\epsilon\in(0,1]$ such that it is at least $c_1$ for $s\in[0,\epsilon]$.

(3) Let $r \coloneqq d_H(h,Z(x))\in(0,\epsilon]$ (the case $r = 0$ is trivial). By (1), $h = \sigma(r)$ for a unit-speed normal geodesic $\sigma$ from the closest point $z = \sigma(0)\in Z(x)$. Set $\psi(s) \coloneqq V_x(\sigma(s))$. Then $\psi(0) = \|x\|^2$, $\psi'(0) = 0$ by (1), and $\psi''(s) = \mathrm{Hess}\,V_x(\dot{\sigma},\dot{\sigma})\geq c_1$ on $[0,r]$ by (2). Integrating twice gives $\psi(r) - \psi(0)\geq \frac{c_1}{2}r^2$.
\end{proof}

\begin{Lem} \label{lem:transverse_decay}
Let $x\in\t_{\CC}$ be regular semisimple, let $c_1,\epsilon$ be as in Lemma \ref{lem 6.10}, and let $h^{*}\in C^2([0,l],H)$ be the minimizer of $\int_0^l(\|\dot{h}\|^2_H + V_x(h))\,dt$ with $h^{*}(0) = p$ and $h^{*}(l) = q$. Set $d_p \coloneqq d_H(p,Z(x))$, $d_q \coloneqq d_H(q,Z(x))$,
\[
R \coloneqq \max(d_p,d_q,\epsilon), \qquad \lambda \coloneqq \Bigl(\frac{c_1\,\epsilon}{2R}\Bigr)^{1/2} > 0.
\]
Then $\rho(t) \coloneqq d_H(h^{*}(t),Z(x))\leq R$ for all $t\in[0,l]$, and
\begin{equation}
\label{eq:barrier_estimate}
\rho(t) \le \frac{\sinh(\lambda(l - t))}{\sinh(\lambda l)}\, d_p + \frac{\sinh(\lambda t)}{\sinh(\lambda l)}\, d_q \leq \frac{d_p\,e^{-\lambda t} + d_q\,e^{-\lambda(l-t)}}{1 - e^{-2\lambda l}} \qquad (t\in[0,l]).
\end{equation}
In particular, if $d_p,d_q\leq\epsilon$, one may take $\lambda = (c_1/2)^{1/2}$.
\end{Lem}

\begin{proof}
Let $r \coloneqq d_H(\cdot,Z(x))$. By Lemma \ref{lem 6.10}(1), $r$ is smooth on $H\setminus Z(x)$, with $\grad_H r(\sigma(s)) = \dot{\sigma}(s)$ along every unit-speed normal geodesic $\sigma$ from $Z(x)$, and $r$ is geodesically convex, being the distance to a closed convex subset of a Cartan--Hadamard manifold; in particular $\mathrm{Hess}_H r\geq 0$ on $H\setminus Z(x)$.

\medskip
\noindent\textbf{Step 1: A differential inequality.} Let $h\in H\setminus Z(x)$, write $h = \sigma(r(h))$ for a unit-speed normal geodesic $\sigma$, and set $\psi(s) \coloneqq V_x(\sigma(s))$. As in the proof of Lemma \ref{lem 6.10}(3), $\psi'(0) = 0$, $\psi''\geq 0$ everywhere by geodesic convexity of $V_x$, and $\psi''\geq c_1$ on $[0,\epsilon]$. Hence
\begin{equation} \label{eq:radial-derivative}
\langle\grad_H V_x(h),\grad_H r(h)\rangle = \psi'(r(h)) = \int_0^{r(h)}\psi''(s)\,ds \geq c_1\min(r(h),\epsilon).
\end{equation}
At every $t$ with $\rho(t) > 0$, $\rho$ is $C^2$ near $t$, and using the Euler--Lagrange equation $\nabla_t\dot{h}^{*} = \frac{1}{2}\grad_H V_x(h^{*})$,
\[
\ddot{\rho}(t) = \mathrm{Hess}_H r(\dot{h}^{*},\dot{h}^{*}) + \langle\grad_H r,\nabla_t\dot{h}^{*}\rangle \geq \frac{1}{2}\langle\grad_H r,\grad_H V_x\rangle \geq \frac{c_1}{2}\min(\rho(t),\epsilon).
\]

\medskip
\noindent\textbf{Step 2: $\rho\leq R$.} If $\rho$ attained a maximum $\rho(t_0) > \max(d_p,d_q)$ at some $t_0\in(0,l)$, then $\rho(t_0) > 0$, so $\ddot{\rho}(t_0)\geq\frac{c_1}{2}\min(\rho(t_0),\epsilon) > 0$, contradicting maximality. Hence $\rho\leq\max(d_p,d_q)\leq R$.

\medskip
\noindent\textbf{Step 3: A linear inequality.} For $0 < \rho\leq R$ we have $\min(\rho,\epsilon)\geq\frac{\epsilon}{R}\rho$: if $\rho\leq\epsilon$ this follows from $\epsilon\leq R$, and if $\rho > \epsilon$ from $\rho\leq R$. Hence, by Steps 1 and 2,
\begin{equation} \label{eq:diff_ineq}
\ddot{\rho}(t)\geq\lambda^2\rho(t) \quad\text{at every } t\in(0,l) \text{ with } \rho(t) > 0.
\end{equation}

\medskip
\noindent\textbf{Step 4: Comparison.} Let $u$ be the solution of $\ddot{u} = \lambda^2u$ on $[0,l]$ with $u(0) = d_p$ and $u(l) = d_q$, namely
\begin{equation} \label{eq:explicit_u}
u(t) = \frac{\sinh(\lambda(l-t))}{\sinh(\lambda l)}\,d_p + \frac{\sinh(\lambda t)}{\sinh(\lambda l)}\,d_q \geq 0.
\end{equation}
Let $w \coloneqq \rho - u$, so that $w(0) = w(l) = 0$. Suppose $w$ attains a positive maximum at $t_0\in(0,l)$. Then $\rho(t_0) > u(t_0)\geq 0$, so \eqref{eq:diff_ineq} applies at $t_0$, giving $\ddot{w}(t_0)\geq\lambda^2 w(t_0) > 0$, which contradicts $\ddot{w}(t_0)\leq 0$. Hence $\rho\leq u$, which is the first inequality in \eqref{eq:barrier_estimate}. For the second, if $0\leq a\leq b$ then $\sinh a\leq\frac{1}{2}e^{a}$ and $\sinh b = \frac{1}{2}e^{b}(1 - e^{-2b})$, so $\sinh a/\sinh b\leq e^{a-b}/(1 - e^{-2b})$; apply this with $(a,b) = (\lambda(l-t),\lambda l)$ and $(a,b) = (\lambda t,\lambda l)$. If $d_p,d_q\leq\epsilon$, then $R = \epsilon$ and $\lambda = (c_1/2)^{1/2}$.
\end{proof}

\begin{Cor} \label{cor:middle}
In the situation of Lemma \ref{lem:transverse_decay}, for every $\delta\in(0,\frac{1}{2})$,
\[
\sup_{t\in[\delta l,(1-\delta)l]} d_H\bigl(h^{*}(t),Z(x)\bigr)\leq\frac{(d_p + d_q)\,e^{-\lambda\delta l}}{1 - e^{-2\lambda l}}.
\]
In particular, for fixed endpoints $p,q$, the minimizer converges to $Z(x)$ exponentially fast in $l$, uniformly on the middle portion $[\delta l,(1-\delta)l]$ of the edge.
\end{Cor}

\begin{proof}
For $t\in[\delta l,(1-\delta)l]$, both $t$ and $l - t$ are at least $\delta l$; apply \eqref{eq:barrier_estimate}. Since $\lambda$ depends only on $x$, $d_p$ and $d_q$, it does not depend on $l$.
\end{proof}

\begin{Rem}[Long edges: what remains open] \label{rem:long-edge}
Let $e_0$ be a separating internal edge, and consider the limit $l_{e_0}\to\infty$ at a point of $X(\Gamma)$ where $x_{e_0}$ is regular semisimple.
\begin{enumerate}
    \item Lemma \ref{lem:transverse_decay} and Corollary \ref{cor:middle} concern a single edge with fixed endpoint values. In the quiver, the endpoint values $p^{*}_{v_1},p^{*}_{v_2}$ of the global minimizer depend on $l_{e_0}$; bounds on $d_H(p^{*}_{v_i},Z(x_{e_0}))$ that are uniform as $l_{e_0}\to\infty$ are needed to apply them to $\M_{\bh}(\Gamma,\mathbf{l})$.
    \item One expects the Nahm data on the long edge to become asymptotically abelian, with root-space components decaying exponentially away from the endpoints. Deducing this from Corollary \ref{cor:middle} requires, in addition, a bound on the $Z(x_{e_0})$-component of the minimizer that is uniform in $l_{e_0}$ (the operators $\Ad_{\exp(iY)}$, $Y\in\t$, are unbounded on the root spaces), together with an estimate for the transverse velocity of the minimizer. We do not establish these here.
    \item Let $\Gamma_1,\Gamma_2$ be obtained by cutting $e_0$ at an interior point, so that $\Gamma$ is recovered from $\Gamma_1*_v\Gamma_2$ by an $S$-move. Since $|E_1| + |E_2| = |E| + 1$ and $|(\Gamma_1)_{\nt}| + |(\Gamma_2)_{\nt}| = |\Gamma_{\nt}|$, Theorem \ref{smooth quotient} gives $\dim\M_{\bh}(\Gamma_1) + \dim\M_{\bh}(\Gamma_2) = \dim\M_{\bh}(\Gamma) + 4\dim G$. Consequently $(\M_{\bh}(\Gamma_1)\times\M_{\bh}(\Gamma_2))/\!\!/\!\!/_{\bh}T_G$ has dimension $\dim\M_{\bh}(\Gamma) + 4(\dim G - \mathrm{rank}\,G)$, which exceeds $\dim\M_{\bh}(\Gamma)$ when $G$ is non-abelian. A splitting of the form $\lim_{l_{e_0}\to\infty}\M_{\bh}(\Gamma,\mathbf{l})\cong(\M_{\bh}(\Gamma_1)\times\M_{\bh}(\Gamma_2))/\!\!/\!\!/_{\bh}T_G$ therefore cannot hold with these factors. A correct formulation should replace them by moduli spaces of Nahm data on the two halves with a semi-infinite edge and regular semisimple asymptotics, in the sense of Kronheimer \cite{Kronheimer1990coadjoint} (see also \cite{Biquard1996Nahm,Kovalev1996Nahm}), glued along the residual torus symmetry. We leave this to future work.
\end{enumerate}
\end{Rem}

\bibliographystyle{plain}
\bibliography{Nahm.bib}
\end{document}